\documentclass{article}
\usepackage{amsmath,amssymb,amscd,amsthm}
\usepackage{stmaryrd}

\usepackage[cal=boondox,calscaled=.96]{mathalfa}
\usepackage{hyperref}

\long\def\inhibe#1\endinhibe{\relax}

\newcommand{\QQ}{\mathbb{Q}}
\newcommand{\ZZ}{\mathbb{Z}}
\newcommand{\N}{\mathbb{N}}

\newcommand{\fn}{\mathfrak{n}}
\newcommand{\ft}{\mathfrak{t}}

\newcommand{\DD}{\mathcal{D}}

\newcommand{\sbullet}{{\hspace{.1em}\scriptstyle\bullet\hspace{.1em}}}

\DeclareMathOperator{\ord}{ord}

\DeclareMathOperator{\Aut}{Aut}

\DeclareMathOperator{\supp}{\rm supp}

\DeclareMathOperator{\U}{\mathcal{U}}
\DeclareMathOperator{\Rees}{\mathcal{R}}
\DeclareMathOperator{\Sim}{Sym}

\DeclareMathOperator{\Der}{Der}

\DeclareMathOperator{\Hom}{Hom}

\DeclareMathOperator{\diff}{\mathcal{D}iff}

\DeclareMathOperator{\End}{End}

\DeclareMathOperator{\gr}{\rm gr}

\newcommand{\pcirc}{{\scriptstyle \,\circ\,}}
\newcommand{\smallcirc}{{\scriptstyle \circ}}
\newcommand\Id{{\rm Id}}

\newcommand{\btimes}{{\scriptstyle \,\boxtimes\,}}

\DeclareMathOperator{\Par}{\mathcal{P}}
\DeclareMathOperator{\Partes}{\mathfrak{P}}

\newcommand{\bfu}{{\bf u}}
\newcommand{\bfv}{{\bf v}}
\newcommand{\bfs}{{\bf s}}
\newcommand{\bft}{{\bf t}}

\DeclareMathOperator{\HS}{HS}
\DeclareMathOperator{\Ider}{Ider}

\DeclareMathOperator{\IHS}{IHS}
\DeclareMathOperator{\Sub}{\mathcal{S}}

\DeclareMathOperator{\opvarphi}{\varphi}

\newcommand{\elln}{\ell}

\newtheorem{theorem}{Theorem}

\newtheorem{corollary}{Corollary}
\newtheorem{proposition}{Proposition}
\newtheorem{lemma}{Lemma}

\newcounter{def}
\newenvironment{definition}{\medskip
\refstepcounter{def}\noindent {\bf Definition \thedef.}\ }{\vspace{1ex}\par}

\newcounter{rem}
\newenvironment{remark}{\medskip
\refstepcounter{rem}\noindent {\it Remark \therem.}\ }{\vspace{1ex}\par}

\newcounter{numero}
\newcommand{\numero}{\refstepcounter{numero}\noindent {\bf  \thenumero.\ \ }}

\newcounter{notacion}
\newenvironment{notation}{\medskip
\refstepcounter{notacion}\noindent {\bf Notation \thenotacion.}\ }{\vspace{1ex}\par}

\newcounter{exam}
\newenvironment{example}{\medskip
\refstepcounter{exam}\noindent {\it Remark \theexam.}\ }{\vspace{1ex}\par}

\title{On Hasse--Schmidt derivations: the action of substitution maps}

\author{L. Narv\'{a}ez Macarro\thanks{Partially supported by MTM2016-75027-P, P12-FQM-2696
 and FEDER. }}

\date{}

\begin{document}

\maketitle
\begin{center} \it Dedicated to Antonio Campillo on the ocassion of his 65th birthday\end{center}
\bigskip

\begin{abstract}
{We study the action of substitution maps between power series rings as an additional algebraic structure on the groups of Hasse--Schmidt derivations. This structure appears as a counterpart
 of the module structure on classical derivations. }
\end{abstract}

\section*{Introduction}
\addcontentsline{toc}{section}{Introduction}

For any commutative algebra $A$ over a commutative ring $k$, the set $\Der_k(A)$ of $k$-derivations of $A$ is an ubiquous object in Commutative Algebra and Algebraic Geometry. It carries  an $A$-module structure and a $k$-Lie algebra structure. Both structures give rise to a {\em Lie-Rinehart  algebra} structure over $(k,A)$. The $k$-derivations of $A$ are contained in the filtered ring of $k$-linear differential operators $\DD_{A/k}$, whose graded ring is commutative and we obtain a canonical map of graded $A$-algebras
$$ \tau: \Sim_A \Der_k(A) \longrightarrow \gr \DD_{A/k}.
$$
If $\QQ \subset k$ and $\Der_k(A)$ is a finitely generated projective $A$-module, the map $\tau$ is an isomorphism (\cite[Corollary 2.17]{nar_2009}) and we can deduce that the ring $\DD_{A/k}$ is the enveloping algebra of the Lie-Rinehart algebra $\Der_k(A)$ (cf. \cite[Proposition 2.1.2.11]{nar_Maringa}).

If we are not in characteristic $0$, even if $A$ is ``smooth'' (in some sense) over $k$, e.g. $A$ is a polynomial or a power series ring with coefficients in $k$, the map $\tau$ has no chance to be an isomorphism.

In \cite{nar_2009} we have proved that, if we denote by $\Ider_k(A) \subset \Der_k(A)$ the $A$-module of {\em integrable derivations} in the sense of Hasse--Schmidt (see Definition \ref{def:HS-integ}), then there is a canonical map of graded $A$-algebras
$$ \vartheta: \Gamma_A \Ider_k(A) \longrightarrow \gr \DD_{A/k},
$$
where $\Gamma_A(-)$ denotes the {\em divided power algebra} functor,
such that:
\begin{enumerate}
\item[(i)] $\tau=\vartheta$ when $\QQ \subset k$  (in that case $\Ider_k(A) = \Der_k(A)$ and $\Gamma_A = \Sim_A$).
\item[(ii)] $\vartheta$ is an isomorphism whenever $\Ider_k(A) = \Der_k(A)$ and $\Der_k(A)$ is a finitely generated projective $A$-module.
\end{enumerate}
The above result suggests an idea: under the ``smoothness'' hypothesis (ii), can be the ring $\DD_{A/k}$ and their modules functorially reconstructed from Hasse--Schmidt derivations?
To tackle it, we first need to explore the algebraic structure of Hasse--Schmidt derivations. 
\medskip

Hasse--Schmidt derivations of length $m\geq 1$ form a group, non-abelian for $m\geq 2$, which coincides with the (abelian) additive group of usual derivations $\Der_k(A)$ for $m=1$. But $\Der_k(A)$ has also an $A$-module structure and a natural questions arises: Do Hasse--Schmidt derivations of any length have some natural structure extending the $A$-module structure of $\Der_k(A)$ for length $=1$?
\medskip

This paper is devoted to study the action of {\em substitution maps} (between power series rings) on Hasse--Schmidt derivations as an answer to the above question. This action plays a key role in \cite{nar_in_prep}.
\medskip

Now let us comment on the content of the paper.
\medskip

In Section 1 we have gathered, due to the lack of convenient references, some basic facts and constructions about rings of formal power series in an arbitrary number of variables with coefficients in a non-necessarily commutative ring. In the case of a finite number of variables many results and proofs become simpler, but we need the infinite case in order to study $\infty$-variate Hasse-Schmidt derivations later. 
\medskip

Sections 2 and 3 are devoted to the study of substitution maps between power series rings and their action on power series rings with coefficients on a (bi)module.
\medskip

In Section 4 we study multivariate (possibly $\infty$-variate) Hasse--Schmidt deri\-vations. They are a natural generalization of usual Hasse--Schmidt derivations and they provide a convenient framework to deal with Hasse--Schmidt derivations.
\medskip

In Section 5 we see how substitution maps act on Hasse--Schmidt derivations and we study some compatibilities on this action with respect to the group structure.
\medskip

In Section 6 we show how the action of substitution maps allows us to express any HS-derivation in terms of a fixed one under some natural hypotheses.
This result generalizes Theorem 2.8 in
\cite{magda_nar_hs} and provides a conceptual proof of it.


\section{Rings and (bi)modules of formal power series}

From now on $R$ will be a ring, $k$ will be a commutative ring and $A$ a commutative $k$-algebra. 
A general reference for some of the constructions and results of this section is \cite[\S 4]{bourbaki_algebra_II_chap_4_7}.
\medskip

Let $\bfs$ be a 
set and consider the free commutative monoid $\N^{(\bfs)}$ of maps $\alpha:\bfs \to \N$ such that the set $\supp \alpha := \{s\in \bfs\ |\ \alpha(s)\neq 0\}$ is finite. If $\alpha \in \N^{(\bfs)}$ and $s\in \bfs$ we will write $\alpha_s$ instead of $\alpha(s)$. The elements of the canonical basis of $\N^{(\bfs)}$ will be denoted by $\bfs^t$, $t\in \bfs$:
$\bfs^t_u = \delta_{tu}$ for $t,u\in \bfs$. 
For each $\alpha \in \N^{(\bfs)}$ we have $\alpha = \sum_{\scriptscriptstyle t\in \bfs} \alpha_t \bfs^t$.

The monoid $\N^{(\bfs)}$ is endowed with a natural partial ordering. Namely, for $\alpha,\beta\in \N^{(\bfs)}$, we define
$$ \alpha \leq \beta\quad \stackrel{\text{def.}}{\Longleftrightarrow}\quad \exists \gamma \in \N^{(\bfs)}\ \text{such that}\ \beta = \alpha + \gamma\quad \Leftrightarrow\quad \alpha_s \leq \beta_s\quad \forall s\in \bfs.$$
Clearly, $t\in \supp \alpha$ $\Leftrightarrow$ $\bfs^t \leq \alpha$. The partial ordered set $(\N^{(\bfs)},\leq)$ is a directed ordered set: for any $\alpha,\beta\in \N^{(\bfs)}$, $\alpha,\beta \leq \alpha \vee \beta$ where $(\alpha \vee \beta)_t := \max \{\alpha_t,\beta_t\}$ for all $t\in\bfs$. We will write $\alpha < \beta$ when $\alpha \leq \beta$ and $\alpha \neq \beta$.
\medskip

For a given $\beta \in \N^{(\bfs)}$ the set of $\alpha \in \N^{(\bfs)}$ such that $\alpha \leq \beta$ is finite.
We define $|\alpha| := \sum_{\scriptscriptstyle s\in \bfs} \alpha_s = \sum_{\scriptscriptstyle s\in \supp \alpha} \alpha_s \in \N$. If $\alpha \leq \beta$ then $|\alpha|\leq |\beta|$. Moreover, if $\alpha \leq \beta$ and $|\alpha| = |\beta|$, then $\alpha=\beta$.
The $\alpha \in \N^{(\bfs)}$ with $|\alpha| = 1$ are exactly the elements $\bfs^t$, $t\in \bfs$, of the canonical basis.
\medskip

A {\em formal power series} \index{formal power series} in $\bfs$ with coefficients in $R$ is a formal expression $\sum_{\scriptscriptstyle \alpha \in \N^{(\bfs)}} r_\alpha \bfs^\alpha$ with $r_\alpha \in R$ and $\bfs^\alpha = \prod_{s\in \bfs} s^{\alpha_s} = \prod_{s\in \supp \alpha} s^{\alpha_s}$. Such a formal expression is uniquely determined by the family of coefficients $a_\alpha$, $\alpha \in \N^{(\bfs)}$.
\medskip

If $r=\sum_{\scriptscriptstyle \alpha \in \N^{(\bfs)}} r_\alpha \bfs^\alpha$ and $r'=\sum_{\scriptscriptstyle \alpha \in \N^{(\bfs)}} r'_\alpha \bfs^\alpha$ are two formal power series in $\bfs$ with coefficients in $R$, their sum and their product are defined in the usual way
\begin{eqnarray*}
\displaystyle   r+r' :=  \sum_{\scriptscriptstyle \alpha \in \N^{(\bfs)}} S_\alpha \bfs^\alpha, &\displaystyle   S_\alpha := r_\alpha + r'_\alpha,&\\
\displaystyle r r' :=  \sum_{\scriptscriptstyle \alpha \in \N^{(\bfs)}} P_\alpha \bfs^\alpha, &\displaystyle   P_\alpha := \sum_{\scriptscriptstyle \beta + \gamma=\alpha} r_\beta  r'_\gamma. &
\end{eqnarray*}

The set of formal power series in $\bfs$ with coefficients in $R$ endowed with the above internal operations is a ring called the {\em ring of formal power series in $\bfs$ with coefficients in $R$}\index{ring of formal power series} and is
denoted by $R[[\bfs]]$. It contains the polynomial ring $R[\bfs]$ (and so the ring $R$) and all the monomials $\bfs^\alpha$ are in the center of $R[[\bfs]]$. There is a natural ring epimorphism, that we call the {\em augmentation}, given by
\begin{equation} \label{eq:augmen-psr}
 \sum_{\scriptscriptstyle \alpha \in \N^{(\bfs)}} r_\alpha \bfs^\alpha \in R[[\bfs]] \longmapsto r_0 \in R,
\end{equation}
which is a retraction of the inclusion $R\subset R[[\bfs]]$. Clearly, the ring $R[[\bfs]]$ is commutative if and only if $R$ is commutative and $R^{\text{opp}}[[\bfs]] = R[[\bfs]]^{\text{opp}} $.
\medskip

Any ring homomorphism $f: R \to R'$ induces a ring homomorphism
\begin{equation} \label{eq:f-bar-ring}
 \overline{f}: \sum_{\scriptscriptstyle \alpha \in \N^{(\bfs)}} r_\alpha \bfs^\alpha \in R[[\bfs]] \longmapsto \sum_{\scriptscriptstyle \alpha \in \N^{(\bfs)}} f(r_\alpha) \bfs^\alpha \in R'[[\bfs]],
\end{equation}
and clearly the correspondences $R \mapsto R[[\bfs]]$ and $f \mapsto \overline{f}$ define a functor from the category of rings to itself. If $\bfs=\emptyset$, then $R[[\bfs]] = R$ and the above functor is the identity.
\medskip

\begin{definition}\label{def:k-algebra-over-A}
A {\em $k$-algebra over $A$}  is a (non-necessarily commutative) $k$-algebra $R$ endowed with a map of $k$-algebras $\iota:A \to R$. A map between two $k$-algebras $\iota:A \to R$ and $\iota':A \to R'$ over $A$ is a map $g:R\to R'$ of $k$-algebras such that $\iota' = g \pcirc \iota$.
\end{definition}

If $R$ is a $k$-algebra (over $A$), then $R[[\bfs]]$ is also a $k[[\bfs]]$-algebra (over $A[[\bfs]]$).
\medskip

If $M$ is an $(A;A)$-bimodule, we define in a completely similar way the set of formal power series in $\bfs$ with coefficients in $M$, denoted by $M[[\bfs]]$. It carries an addition $+$, for which it is an abelian group, and  left and right products by elements of $A[[\bfs]]$. With these operations $M[[\bfs]]$ becomes an $(A[[\bfs]];A[[\bfs]])$-bimodule containing the polynomial $(A[\bfs];A[\bfs])$-bimodule $M[\bfs]$. There is also a natural {\em augmentation} $M[[\bfs]] \to M$ which is a section of the inclusion $M \subset M[\bfs]$ and
$M^{\text{opp}}[[\bfs]] = M[[\bfs]]^{\text{opp}}$. If $\bfs=\emptyset$, then $M[[\bfs]]=M$. 
\medskip

The {\em support}  of a series $m=\sum_\alpha m_\alpha \bfs^\alpha \in M[[\bfs]]$ is $\supp(x) := \{  \alpha \in \N^{(\bfs)}| m_\alpha \neq 0\} \subset \N^{(\bfs)}$. It is clear that $m=0\Leftrightarrow \supp(m) = \emptyset$. 
The {\em order}  of a non-zero series $m=\sum_\alpha m_\alpha \bfs^\alpha \in M[[\bfs]]$ is
$ \ord (m) := \min \{ |\alpha|\ |\ \alpha \in \supp(m) \} \in \N$. 
If $m=0$  we define $\ord (0) = \infty$. It is clear that for $a\in A[[\bfs]]$ and $m,m'\in M[[\bfs]]$ we have 
$\supp(m+m') \subset \supp(m) \cup \supp(m')$, $\supp (am), \supp(ma) \subset \supp(m) + \supp(a)$,
$\ord(m+m') \geq \min \{\ord(m),\ord(m')\}$ and 
$\ord (am), \ord (ma) \geq \ord(a) + \ord(m)$. Moreover, if $\ord(m') > \ord(m)$, then $\ord(m+m') = \ord(m)$.
\medskip

Any $(A;A)$-linear map $h:M \to M'$ between two $(A;A)$-bimodules induces in an obvious way and $(A[[\bfs]];A[[\bfs]])$-linear map 
\begin{equation} \label{eq:h-bar-mod}
 \overline{h}: \sum_{\scriptscriptstyle \alpha \in \N^{(\bfs)}} m_\alpha \bfs^\alpha \in M[[\bfs]] \longmapsto \sum_{\scriptscriptstyle \alpha \in \N^{(\bfs)}} h(m_\alpha) \bfs^\alpha \in M'[[\bfs]],
\end{equation}
and clearly the correspondences $M \mapsto M[[\bfs]]$ and $h \mapsto \overline{h}$ define a functor from the category of $(A;A)$-bimodules to the category $(A[[\bfs]];A[[\bfs]])$-bimodules. 
\medskip

For each $\beta \in M^{(\bfs)}$, let us denote by $\fn^M_\beta(\bfs)$ the subset of $M[[\bfs]]$ whose elements are the formal power series $\sum m_\alpha \bfs^\alpha$ with $m_\alpha = 0$ for all $\alpha \leq \beta$. One has $\fn^M_\beta(\bfs) \subset \fn^M_\gamma(\bfs)$ whenever $\gamma \leq \beta$, and $ \fn^M_{\alpha \vee\beta}(\bfs) \subset \fn^M_\alpha(\bfs) \cap \fn^M_\beta(\bfs)$. \medskip

It is clear that the $\fn^M_\beta(\bfs)$ are sub-bimodules of $M[[\bfs]]$ and $\fn^A_\beta(\bfs) M[[\bfs]] \subset \fn^M_\beta(\bfs)$ and  $ M[[\bfs]] \fn^A_\beta(\bfs) \subset \fn^M_\beta(\bfs)$. For $\beta=0$, $\fn^M_0(\bfs)$ is the kernel of the augmentation $M[[\bfs]] \to M$.

In the case of a ring $R$,
 the $\fn^R_\beta(\bfs)$ are two-sided ideals of $R[[\bfs]]$, and $\fn^R_0(\bfs)$ is the kernel of the augmentation $R[[\bfs]] \to R$.
\medskip

We will consider $R[[\bfs]]$ as a topological ring with $\{\fn^R_\beta(\bfs), \beta \in \N^{(\bfs)}\}$ as a fundamental system of neighborhoods of $0$. We will also consider $M[[\bfs]]$ as a topological $(A[[\bfs]];A[[\bfs]])$-bimodule with $\{\fn^M_\beta(\bfs), \beta \in \N^{(\bfs)}\}$ as a fundamental system of neighborhoods of $0$ for both, a topological left $A[[\bfs]]$-module structure and a topological right $A[[\bfs]]$-module structure. If $\bfs$ is finite, then $\fn^M_\beta(\bfs)= \sum_{s\in\bfs} s^{\beta_s +1} M[[\bfs]] = \sum_{s\in\bfs}  M[[\bfs]] s^{\beta_s +1}$ and so the above topologies on $R[[\bfs]]$, and so on $A[[\bfs]]$, and on $M[[\bfs]]$ coincide with the $\langle \bfs \rangle$-adic topologies.
\medskip

Let us denote by $\fn^M_\beta(\bfs)^{\bf c} \subset M[\bfs]$ the intersection of $\fn^M_\beta(\bfs)$ with  $M[\bfs]$, i.e. the subset of $M[\bfs]$ whose elements are the finite sums $\sum m_\alpha \bfs^\alpha$ with $m_\alpha = 0$ for all $\alpha \leq \beta$. It is clear that the natural map
$ R[\bfs]/\fn^R_\beta(\bfs)^{\bf c} \longrightarrow R[[\bfs]]/\fn^R_\beta(\bfs)$
is an isomorphism of rings and the quotient $R[[\bfs]]/\fn^R_\beta(\bfs)$ is a finitely generated free left (and right) $R$-module with basis the set of the classes of monomials $\bfs^\alpha$, $\alpha\leq \beta$. 
\medskip

In the same vein, the $\fn^M_\beta(\bfs)^{\bf c}$ are sub-$(A[\bfs];A[\bfs])$-bimodules of 
$M[\bfs]$ and the natural map
$ M[\bfs]/\fn^M_\beta(\bfs)^{\bf c} \longrightarrow M[[\bfs]]/\fn^M_\beta(\bfs)$
is an isomorphism of bimodules over $(A[\bfs]/\fn^A_\beta(\bfs)^{\bf c};A[\bfs]/\fn^A_\beta(\bfs)^{\bf c})$. Moreover, we have a commutative diagram of natural $\ZZ$-linear isomorphisms
\begin{equation} \label{eq:varrho-lambda}
\begin{CD}
A[\bfs]/\fn^A_\beta(\bfs)^{\bf c} \otimes_A M @>{\varrho}>{\simeq}> M[\bfs]/\fn^M_\beta(\bfs)^{\bf c} @<{\lambda}<{\simeq}< M \otimes_A A[\bfs]/\fn^A_\beta(\bfs)^{\bf c}\\
@V{\text{nat.} \otimes \Id}V{\simeq}V @VV{\simeq }V @V{\simeq}V{ \Id \otimes \text{nat.} }V\\
A[[\bfs]]/\fn^A_\beta(\bfs) \otimes_A M @>{\varrho'}>{\simeq}> M[[\bfs]]/\fn^M_\beta(\bfs) @<{\lambda'}<{\simeq}< M \otimes_A A[[\bfs]]/\fn^A_\beta(\bfs)
\end{CD}
\end{equation}
where $\varrho$ (resp. $\varrho'$) is an isomorphism of 
  $(A[\bfs]/\fn^A_\beta(\bfs)^{\bf c};A)$-bimodules (resp. of  $(A[[\bfs]]/\fn^A_\beta(\bfs);A)$-bimodules   ) and $\lambda$ (resp. $\lambda'$) is an isomorphism of bimodules over $(A;A[\bfs]/\fn^A_\beta(\bfs)^{\bf c})$(resp. over $(A;A[[\bfs]]/\fn^A_\beta(\bfs))$.
\medskip

It is clear that the natural map 
$$ R[[\bfs]] \longrightarrow  \lim_{\stackrel{\longleftarrow}{\beta\in \N^{(\bfs)}}} R[[\bfs]]/\fn^R_\beta(\bfs) \equiv \lim_{\stackrel{\longleftarrow}{\beta\in \N^{(\bfs)}}} R[\bfs]/\fn^R_\beta(\bfs)^{\bf c}  $$
is an isomorphism of rings and so $R[[\bfs]]$ is complete (hence, separated). Moreover, $R[[\bfs]]$ appears as the completion of the polynomial ring $R[\bfs]$ endowed with the topology with $\{\fn^R_\beta(\bfs)^{\bf c}, \beta\in \N^{(\bfs)}\}$ as a fundamental system of neighborhoods of $0$. 

Similarly, the natural map 
$$ M[[\bfs]] \longrightarrow  \lim_{\stackrel{\longleftarrow}{\beta\in \N^{(\bfs)}}} M[[\bfs]]/\fn^M_\beta(\bfs) \equiv \lim_{\stackrel{\longleftarrow}{\beta\in \N^{(\bfs)}}} M[\bfs]/\fn^M_\beta(\bfs)^{\bf c}  $$
is an isomorphism of $(A[[\bfs]];A[[\bfs]])$-bimodules, and so $M[[\bfs]]$ is complete (hence, separated). Moreover, $M[[\bfs]]$ appears as the completion of the bimo\-dule $M[\bfs]$ over $(A[\bfs];A[\bfs])$ endowed with the topology with $\{\fn^M_\beta(\bfs)^{\bf c}, \beta\in \N^{(\bfs)}\}$ as a fundamental system of neighborhoods of $0$.
\medskip

Since the subsets $\{\alpha \in 
\N^{(\bfs)}\ |\ \alpha \leq \beta\}$, $\beta\in \N^{(\bfs)}$, are cofinal among the finite subsets of $\N^{(\bfs)}$, the additive isomorphism
$$ \sum_{\scriptscriptstyle \alpha \in \N^{(\bfs)}} m_\alpha \bfs^\alpha \in M[[\bfs]]  \mapsto \{m_\alpha\}_{\alpha \in \N^{(\bfs)}} \in M^{\N^{(\bfs)}}$$
is a homeomorphism, where $M^{\N^{(\bfs)}}$ is endowed with the product of discrete topologies on each copy of $M$. In particular, any formal power series $ \sum m_\alpha \bfs^\alpha$ is the limit of its finite partial sums $ \sum_{\alpha \in F} m_\alpha \bfs^\alpha$, over the filter of finite subsets $F\subset \N^{(\bfs)}$.
\medskip

Since the quotients $A[[\bfs]]/\fn^A_\beta(\bfs)$ are free  $A$-modules,
we have exact sequences
$$0 \longrightarrow \fn^A_\beta(\bfs) \otimes_A M \longrightarrow A[[\bfs]]\otimes_A M \longrightarrow \frac{A[[\bfs]]}{\fn^A_\beta(\bfs)} \otimes_A M \longrightarrow 0
$$  
and the tensor product $A[[\bfs]]\otimes_A M$ is a topological left $A[[\bfs]]$-module with $\{\fn^A_\beta(\bfs)\otimes_A M, \beta \in \N^{(\bfs)}\}$ as a fundamental system of neighborhoods of $0$. The natural $(A[[\bfs]];A)$-linear map
$$ A[[\bfs]]\otimes_A M \longrightarrow M[[\bfs]]
$$
is continuous and, if we denote by $ A[[\bfs]]\widehat{\otimes}_A M$ the completion of $ A[[\bfs]]\otimes_A M$, the induced map $ A[[\bfs]]\widehat{\otimes}_A M \longrightarrow M[[\bfs]]$ is an isomorphism of $(A[[\bfs]];A)$-bimodules, since
we have natural $(A[[\bfs]];A)$-linear isomorphisms
$$ \left( A[[\bfs]]\otimes_A M \right)/\left( \fn^A_\beta(\bfs)\otimes_A M \right) \simeq  \left( A[[\bfs]]/\fn^A_\beta(\bfs) \right)\otimes_A M \simeq M[[\bfs]]/\fn^M_\beta(\bfs) 
$$
for $\beta \in \N^{(\bfs)}$, and so
\begin{equation} \label{eq:A[[s]]wotM}
 A[[\bfs]]\widehat{\otimes}_A M = \lim_{\stackrel{\longleftarrow}{\beta\in \N^{(\bfs)}}} 
\left(\frac{ A[[\bfs]]\otimes_A M }{ \fn^A_\beta(\bfs)\otimes_A M }\right) \simeq 
\lim_{\stackrel{\longleftarrow}{\beta\in \N^{(\bfs)}}}  \left( \frac{M[[\bfs]]}{\fn^M_\beta(\bfs)} \right) \simeq M[[\bfs]].
\end{equation}
Similarly, the natural $(A;A[[\bfs]])$-linear map $M\otimes_A A[[\bfs]] \to M[[\bfs]]$ induces an isomorphism 
$ M \widehat{\otimes}_A  A[[\bfs]] \xrightarrow{\sim} M[[\bfs]]$ of $(A;A[[\bfs]])$-bimodules.
\medskip

If $h: M \to M'$ is an $(A;A)$-linear map between two $(A;A)$-bimodules, the induced map $\overline{h}:M[[\bfs] \to M'[[\bfs]$ (see (\ref{eq:h-bar-mod})) is clearly continuous 
and there is a commutative diagram
$$
\begin{CD}
A[[\bfs]]\widehat{\otimes}_A M @>{\simeq}>> M[[\bfs]] @<{\simeq}<< M \widehat{\otimes}_A  A[[\bfs]] \\
@V{\Id\widehat{\otimes} h }VV @V{\overline{h}}VV @V{h\widehat{\otimes} \Id }VV\\
A[[\bfs]]\widehat{\otimes}_A M' @>{\simeq}>> M'[[\bfs]] @<{\simeq}<< M' \widehat{\otimes}_A  A[[\bfs]].
\end{CD}
$$
Similarly, for any ring homomorphism $f: R \to R'$, the induced ring homomorphism
$ \overline{f}: R[[\bfs]] \to  R'[[\bfs]]$ is also continuous.

\begin{definition} We say that a subset $\Delta \subset \N^{(\bfs)}$ is an {\em ideal} of $\N^{(\bfs)}$ \index{ideal (of $\N^{(\bfs)}$)} (resp. a
{\em co-ideal}\index{co-ideal (of $\N^{(\bfs)}$)} of $\N^{(\bfs)}$)
if whenever $\alpha\in\Delta$ and $\alpha\leq \alpha'$ (resp. $\alpha'\leq \alpha$), then $\alpha'\in\Delta$.
\end{definition}
It is clear that $\Delta$ is an ideal if and only if its complement $\Delta^c$ is a co-ideal, and that the union and the intersection of any family of ideals (resp. of co-ideals)  of $\N^{(\bfs)}$ is again an ideal (resp. a co-ideal) of $\N^{(\bfs)}$. 
Examples of ideals (resp. of co-ideals) of $\N^{(\bfs)}$ are the $\beta + \N^{(\bfs)}$
(resp. the $\fn_\beta(\bfs) := \{\alpha \in \N^{(\bfs)}\ |\ \alpha\leq \beta\}$) with $\beta \in \N^{(\bfs)}$. The $\ft_m(\bfs) := \{\alpha \in \N^{(\bfs)}\ |\ |\alpha|\leq m\}$ with $m\geq 0$ are also co-ideals.
Actually, a subset $\Delta  \subset \N^{(\bfs)}$ is an ideal (resp. a co-ideal) if and only if
$\Delta = \cup_{\beta\in \Delta} \left(\beta + \N^{(\bfs)}\right) = \Delta + \N^{(\bfs)} $
  (resp. $\Delta = \cup_{\beta\in \Delta} \fn_\beta(\bfs)$).
\medskip

We say that a co-ideal $\Delta \subset \N^{(\bfs)}$ is bounded if there is an integer $m\geq 0$ such that  $|\alpha|\leq m$ for all $\alpha\in\Delta$. In other words, a co-ideal $\Delta\subset \N^{(\bfs)}$ is bounded if and only if there is an integer $m\geq 0$ such that $\Delta \subset \ft_m(\bfs)$. Also, a co-ideal $\Delta\subset \N^{(\bfs)}$ is non-empty if and only if $\ft_0(\bfs) = \fn_0(\bfs) =\{0\} \subset \Delta$.
\medskip

For a co-ideal $\Delta \subset \N^{(\bfs)}$ and an integer $m\geq 0$, we denote $\Delta^m := \Delta \cap \ft_m(\bfs)$.
\medskip\\

For each co-ideal $\Delta \subset \N^{(\bfs)}$, we denote by $\Delta_M$ the sub-$(A[[\bfs];A[[\bfs]])$-bimodule of $M[[\bfs]]$ whose elements are the formal power series $\sum_{\alpha\in\N^{(\bfs)}} m_\alpha \bfs^\alpha$ such that $m_\alpha=0$ whenever $\alpha\in \Delta$. One has
\begin{eqnarray*}
&\displaystyle
 \Delta_M = \cdots = 
\left\{m\in M[[\bfs]]\ |\ \supp(m) \subset \bigcap_{\beta\in \Delta} \fn_\beta(\bfs)^c\right\}=
&\\
&\displaystyle
\bigcap _{\beta\in \Delta} \left\{m\in M[[\bfs]]\ |\ \supp(m) \subset  \fn_\beta(\bfs)^c\right\} =
 \bigcap_{\beta\in \Delta} \fn^M_\beta(\bfs),
\end{eqnarray*}
and so $\Delta_M$ is closed in $M[[\bfs]]$. Let $\Delta' \subset \N^{(\bfs)}$ be another co-ideal.
We have
$$ \Delta_M + \Delta'_M = (\Delta \cap \Delta')_M.
$$
If $\Delta \subset \Delta'$, then $\Delta'_M \subset \Delta_M$, 
and if $a\in \Delta'_A$, $m\in \Delta_M$ we have
$$ \supp (am) \subset \supp(a) + \supp(m) \subset \left(\Delta'\right)^c + \Delta^c \subset \left(\Delta'\right)^c \cap \Delta^c =\left( \Delta' \cup \Delta\right)^c ,
$$
and so $ \Delta'_A \Delta_M \subset (\Delta' \cup \Delta)_M$. Is a similar way we obtain 
$ \Delta_M \Delta'_A  \subset (\Delta' \cup \Delta)_M$.
\medskip\\

Let us denote by $M[[\bfs]]_\Delta := M[[\bfs]]/\Delta_M$ endowed with the quotient topology. The elements in $M[[\bfs]]_\Delta$ are power series of the form 
$$\sum_{\scriptscriptstyle \alpha\in\Delta} m_\alpha \bfs^\alpha,\quad m_\alpha \in M.
$$
It is clear that $M[[\bfs]]_\Delta$ is a topological $(A[[\bfs]]_\Delta;A[[\bfs]]_\Delta)$-bimodule.
A fundamental system of neighborhoods of $0$ in $M[[\bfs]]_\Delta$ consist of
$$\frac{\fn^M_\beta(\bfs) + \Delta_M}{\Delta_M} = \frac{(\fn_\beta(\bfs) \cap\Delta)_M}{\Delta_M},\quad \beta \in \N^{(\bfs)},
$$
and since the subsets 
$\fn_\beta(\bfs) \cap \Delta, \beta \in \N^{(\bfs)}$, 
are cofinal among the finite subsets of $\Delta$, we conclude that the additive isomorphism
$$ \sum_{\scriptscriptstyle \alpha\in\Delta} m_\alpha \bfs^\alpha \in M[[\bfs]]_\Delta  \mapsto \{m_\alpha\}_{\alpha \in \Delta} \in M^\Delta    $$
is a homeomorphism, where $M^\Delta$ is endowed with the product of discrete topologies on each copy of $M$. 
\medskip\\

\noindent For $\Delta \subset \Delta'$ co-ideals of $\N^{(\bfs)}$, we have natural 
continuous $(A[[\bfs]]_{\Delta'};A[[\bfs]]_{\Delta'})$-linear projections $\tau_{\Delta' \Delta}:M[[\bfs]]_{\Delta'} \longrightarrow M[[\bfs]]_\Delta$, that we also call {\em truncations}, \index{truncation} 
$$\tau_{\Delta' \Delta} : \sum_{\scriptscriptstyle \alpha\in\Delta'} m_\alpha \bfs^\alpha \in M[[\bfs]]_{\Delta'} \longmapsto \sum_{\scriptscriptstyle \alpha\in\Delta} m_\alpha \bfs^\alpha \in M[[\bfs]]_{\Delta},
$$
and continuous $(A;A)$-linear scissions 
$$\sum_{\scriptscriptstyle \alpha\in\Delta} m_\alpha \bfs^\alpha \in M[[\bfs]]_\Delta \longmapsto \sum_{\scriptscriptstyle \alpha\in\Delta} m_\alpha \bfs^\alpha \in M[[\bfs]]_{\Delta'}.
$$ 
which are topological immersions.
\medskip

In particular we have natural continuous $(A;A)$-linear topological embeddings $M[[\bfs]]_\Delta \hookrightarrow M[[\bfs]]$ and we define the {\em support} (resp. the {\em order}) of any element in $M[[\bfs]]_\Delta$ as its support (resp. its order) as element of $M[[\bfs]]$.
\medskip

We have a bicontinuous isomorphism of $(A[[\bfs]]_\Delta;A[[\bfs]]_\Delta)$-bimodules
$$ M[[\bfs]]_\Delta = \lim_{\stackrel{\longleftarrow}{m\in\N}} M[[\bfs]]_{\Delta^m}.
$$
For a ring $R$, the $\Delta_R$ are two-sided closed ideals of $R[[\bfs]]$,
$\Delta_R\Delta'_R \subset (\Delta\cup\Delta')_R$
 and we have a bicontinuous ring isomorphism
$$ R[[\bfs]]_\Delta = \lim_{\stackrel{\longleftarrow}{m\in\N}} R[[\bfs]]_{\Delta^m}.
$$
When $\bfs$ is finite, $\ft_m(\bfs)_R$ coincides with the $(m+1)$-power of the two-sided ideal generated by all the variables $s\in \bfs$. 
\medskip

As in (\ref{eq:A[[s]]wotM}) one proves that $A[[\bfs]]_\Delta \otimes_A M$ (resp. $M\otimes_A A[[\bfs]]_\Delta$) is endowed with a natural topology in such a way that the natural map $A[[\bfs]]_\Delta \otimes_A M \to M[[\bfs]]_\Delta$ (resp. $M\otimes_A A[[\bfs]]_\Delta \to M[[\bfs]]_\Delta$)
is continuous and gives rise to a $(A[[\bfs]]_\Delta;A)$-linear (resp. to a $(A;A[[\bfs]]_\Delta)$-linear) isomorphism
$$ A[[\bfs]]_\Delta \widehat{\otimes}_A M \xrightarrow{\sim} M[[\bfs]]_\Delta\quad\quad \text{(resp. $
M\widehat{\otimes}_A A[[\bfs]]_\Delta \xrightarrow{\sim} M[[\bfs]]_\Delta$)}.
$$
If $h: M \to M'$ is an $(A;A)$-linear map between two $(A;A)$-bimodules, the  map $\overline{h}:M[[\bfs] \to M'[[\bfs]$ (see (\ref{eq:h-bar-mod})) obviously satisfies $\overline{h}(\Delta_M) \subset \Delta_{M'}$, and so induces another natural $(A[[\bfs]]_\Delta;A[[\bfs]]_\Delta)$-linear continuous map $M[[\bfs]_\Delta \to M'[[\bfs]_\Delta$, that will be still denoted by 
$\overline{h}$. We have a commutative diagram
$$
\begin{CD}
A[[\bfs]]_\Delta\widehat{\otimes}_A M @>{\simeq}>> M[[\bfs]]_\Delta @<{\simeq}<< M \widehat{\otimes}_A  A[[\bfs]]_\Delta \\
@V{\Id\widehat{\otimes} h }VV @V{\overline{h}}VV @V{h\widehat{\otimes} \Id }VV\\
A[[\bfs]]_\Delta\widehat{\otimes}_A M' @>{\simeq}>> M'[[\bfs]]_\Delta @<{\simeq}<< M' \widehat{\otimes}_A  A[[\bfs]]_\Delta.
\end{CD}
$$
\begin{remark} In the same way that the correspondences $M \mapsto M[[\bfs]]$ and $h \mapsto \overline{h}$ define a functor from the category of $(A;A)$-bimodules to the category of $(A[[\bfs]];A[[\bfs]])$-bimodules, we may consider functors $M \mapsto M[[\bfs]]_\Delta$ and $h \mapsto \overline{h}$ from the category of $(A;A)$-bimodules to the category of $(A[[\bfs]]_\Delta;A[[\bfs]]_\Delta)$-bimodules. We may also consider functors 
$R \mapsto R[[\bfs]]_\Delta$ and $f \mapsto \overline{f}$ from the category of rings to itself. Moreover, if $R$ is a $k$-algebra (over $A$), then $R[[\bfs]]_\Delta$ is a $k[[\bfs]]_\Delta$-algebra (over $A[[\bfs]]_\Delta$).
\end{remark}

\begin{lemma} \label{lema:topol-of-M[[s]]}
Under the above hypotheses, $\Delta_M$ is the closure of $\Delta_\ZZ M[[\bfs]]$.
\end{lemma}

\begin{proof} 
Any element in $\Delta_M$ is of the form $\sum_{\scriptscriptstyle \alpha\in\Delta} m_\alpha \bfs^\alpha$, but $\bfs^\alpha m_\alpha \in \Delta_\ZZ M[[\bfs]]$ whenever $\alpha\in\Delta$ and so it belongs to the closure of $\Delta_\ZZ M[[\bfs]]$.
\end{proof}

\begin{lemma} \label{lemma:unit-psr}
Let $R$ be a ring, $\bfs$ a set and $\Delta\subset \N^{(\bfs)}$ a non-empty co-ideal. 
The units in $R[[\bfs]]_\Delta$ are those power series $r=\sum r_\alpha \bfs^\alpha$ such that $r_0$ is a unit in $R$. Moreover, in the special case where $r_0=1$, the inverse $r^*= \sum r^*_\alpha \bfs^\alpha$ of $r$ is given by
$r^*_0 = 1$ and
$$ r^*_\alpha =  \sum_{\scriptscriptstyle d=1}^{\scriptscriptstyle |\alpha|} (-1)^d
\sum_{\scriptscriptstyle \alpha^\bullet \in \Par(\alpha,d)}
r_{\alpha^1} \cdots  r_{\alpha^d}\quad \text{for\ }\ \alpha\neq 0,
$$
where $\Par(\alpha,d)$ is the set of $d$-uples $\alpha^\bullet=(\alpha^1,\dots,\alpha^d)$ with $\alpha^i \in \N^{(\bfs)}$, $\alpha^i\neq 0$, and $\alpha^1+\cdots + \alpha^d=\alpha$.
\end{lemma}

\begin{proof} The proof is standard and it is left to the reader.
\inhibe
It is clear that $r_0$ to be a unit in $R$ is a necessary condition for $r$ being a unit in $R[[\bfs]]_\Delta$.
For the opposite, it is enough to treat the case $r_0=1$. Let us consider the element $r^*\in R[[\bfs]]_\Delta$ defined above and set $u= r r^*$. We obviously have $u_0=1$. For $\alpha\in\Delta, \alpha \neq 0$ we have
\begin{eqnarray*}
& \displaystyle
 u_\alpha = \sum_{\scriptscriptstyle \beta + \gamma=\alpha} r^{}_\beta r^*_\gamma = r_\alpha + 
\sum_{\substack{\scriptscriptstyle \beta + \gamma=\alpha\\ \scriptscriptstyle |\gamma|>0 }} r_\beta \left(
\sum_{\scriptscriptstyle d=1}^{\scriptscriptstyle |\gamma|} (-1)^d \left( 
\sum_{\scriptscriptstyle \gamma^\bullet \in \Par(\gamma,d)}
r_{\gamma^1} \cdots  r_{\gamma^d} \right) \right) = &
\\
& \displaystyle
 r_\alpha + \sum_{\scriptscriptstyle d=1}^{\scriptscriptstyle |\alpha|} (-1)^d \left(
\sum_{\scriptscriptstyle \alpha^\bullet \in \Par(\alpha,d)}
r_{\alpha^1} \cdots  r_{\alpha^d} \right) + 
\sum_{\substack{\scriptscriptstyle \beta + \gamma=\alpha\\ \scriptscriptstyle |\beta|,|\gamma|>0 }} r_\beta \left(
\sum_{\scriptscriptstyle d=1}^{\scriptscriptstyle |\gamma|} (-1)^d \left( 
\sum_{\scriptscriptstyle \gamma^\bullet \in \Par(\gamma,d)}
r_{\gamma^1} \cdots  r_{\gamma^d} \right) \right) = &
\\
& \displaystyle
r_\alpha + \sum_{\scriptscriptstyle d=1}^{\scriptscriptstyle |\alpha|} (-1)^d \left(
\sum_{\scriptscriptstyle \alpha^\bullet \in \Par(\alpha,d)}
r_{\alpha^1} \cdots  r_{\alpha^d} \right) + 
\sum_{\scriptscriptstyle d=1}^{\scriptscriptstyle |\alpha|-1} (-1)^d \left( 
\sum_{\scriptscriptstyle \alpha^\bullet \in \Par(\alpha,d+1)}
r_{\alpha^1} \cdots  r_{\alpha^{d+1}} \right) = &
\\
& \displaystyle
\sum_{\scriptscriptstyle d=2}^{\scriptscriptstyle |\alpha|} (-1)^d \left(
\sum_{\scriptscriptstyle \alpha^\bullet \in \Par(\alpha,d)}
r_{\alpha^1} \cdots  r_{\alpha^d} \right) +
\sum_{\scriptscriptstyle d=2}^{\scriptscriptstyle |\alpha|} (-1)^{d-1} \left( 
\sum_{\scriptscriptstyle \alpha^\bullet \in \Par(\alpha,d)}
r_{\alpha^1} \cdots  r_{\alpha^d} \right) =0
\end{eqnarray*}
and so $r r^* =1$. In a similar way one proves that $r^* r=1$. We conclude that $r$ is a unit and its inverse $r^*$ is given by the expression above.
\endinhibe
\end{proof}

\begin{notation} \label{notacion:Ump} Let $R$ be a ring, $\bfs$ a set and  $\Delta\subset \N^{(\bfs)}$ a non-empty co-ideal. We denote by $\U^\bfs(R;\Delta)$ the multiplicative sub-group of the units of $R[[\bfs]]_\Delta$ whose 0-degree coefficient is $1$. Clearly, $\U^\bfs(R;\Delta)^{\text{\rm opp}} = \U^\bfs(R^{\text{\rm opp}};\Delta)$.
For $\Delta\subset \Delta'$ co-ideals we have $\tau_{\Delta'\Delta}\left(\U^\bfs(R;\Delta')\right) \subset \U^\bfs(R;\Delta)$ and the truncation map
$\tau_{\Delta'\Delta}:\U^\bfs(R;\Delta') \to \U^\bfs(R;\Delta)$ is a group homomorphisms. Clearly, we have
$$   \U^\bfs(R;\Delta) = \lim_{\stackrel{\longleftarrow}{m\in\N}} \U^\bfs(R;\Delta^m).
$$
For any ring homomorphism $f:R\to R'$, the induced ring homomorphism $\overline{f}: R[[\bfs]]_\Delta \to R'[[\bfs]]_\Delta$ sends $\U^\bfs(R;\Delta)$ into $\U^\bfs(R';\Delta)$ and so it induces natural group homomorphisms
$\U^\bfs(R;\Delta) \to \U^\bfs(R';\Delta)$.
\end{notation}

\begin{definition} \label{defi:external-x}
Let $R$ be a ring, $\bfs,\bft$ sets and  $\nabla\subset \N^{(\bfs)}, \Delta\subset \N^{(\bft)}$ non-empty co-ideals.
For each $r\in R[[\bfs]]_\nabla, r'\in R[[\bft]]_\Delta$, the {\em external product} \index{external product (of power series)} $r\btimes r'\in R[[\bfs \sqcup \bft]]_{\nabla \times \Delta}$ is defined as
$$ r\btimes r' := \sum_{\scriptscriptstyle (\alpha,\beta)\in \nabla \times \Delta} r_\alpha r'_\beta \bfs^\alpha \bft^\beta.$$
\end{definition}

Let us notice that the above definition is consistent with the existence of natural isomorphism of $(R;R)$-bimodules $R[[\bfs]]_\nabla \widehat{\otimes}_R R[[\bft]]_\Delta \simeq R[[\bfs \sqcup \bft]]_{\nabla \times \Delta}\simeq   R[[\bft \sqcup \bfs]]_{\Delta \times \nabla} \simeq    R[[\bft]]_\Delta \widehat{\otimes}_R R[[\bfs]]_\nabla$. Let us also notice that $1\btimes 1 = 1$ and $r\btimes r' = (r \btimes 1) (1 \btimes r')$. Moreover, 
if
$r\in \U^\bfs(R;\nabla)$, $r'\in \U^\bft(R;\Delta)$, then 
$r\btimes r' \in \U^{\bfs\sqcup\bft}(R;\nabla\times \Delta)$
and  $(r\btimes r')^* = {r'}^* \btimes r^*$.
\bigskip

Let $k \to A$ be a ring homomorphism between commutative rings, $E,F$ two $A$-modules, $\bfs$ a set and $\Delta\subset \N^{(\bfs)}$ a non-empty co-ideal, i.e $\fn_0(\bfs) = \{0\} \subset \Delta$.

\begin{proposition} \label{prop:induced-cont-maps-M[[s]]_m} Under the above hypotheses, let $f:E[[\bfs]]_\Delta \to F[[\bfs]]_\Delta$ be a continuous $k[[\bfs]]_\Delta$-linear map. Then, for any co-ideal $\Delta'\subset \N^{(\bfs)}$ with $\Delta'\subset \Delta$ we have
$$   f\left(\Delta'_E/\Delta_E\right) \subset \Delta'_F/\Delta_F
$$
and so there is a unique continuous $k[[\bfs]]_{\Delta'}$-linear map $\overline{f}:E[[\bfs]]_{\Delta'} \to F[[\bfs]]_{\Delta'}$ such that the following diagram is commutative
$$
\begin{CD}
E[[\bfs]]_\Delta @>{f}>> F[[\bfs]]_\Delta\\
@V{\text{nat.}}VV @VV{\text{nat.}}V \\
E[[\bfs]]_{\Delta'} @>{\overline{f}}>> F[[\bfs]]_{\Delta'}.
\end{CD}
$$
\end{proposition}

\begin{proof} It is a straightforward consequence of Lemma \ref{lema:topol-of-M[[s]]}.
\end{proof}

\begin{notation} Under the above hypotheses, the set of all continuous $k[[\bfs]]_\Delta$-linear maps from  $E[[\bfs]]_\Delta$ to  $F[[\bfs]]_\Delta$ will be denoted by 
$$\Hom_{k[[\bfs]]_\Delta}^{\text{\rm top}}(E[[\bfs]]_\Delta,F[[\bfs]]_\Delta).
$$
 It is an $(A[[\bfs]]_\Delta;A[[\bfs]]_\Delta)$-bimodule central over $k[[\bfs]]_\Delta$. For any 
co-ideals $\Delta'\subset\Delta\subset  \N^{(\bfs)}$, 
Proposition \ref{prop:induced-cont-maps-M[[s]]_m} provides 
a natural $(A[[\bfs]]_\Delta;A[[\bfs]]_\Delta)$-linear map
$$ \Hom_{k[[\bfs]]_\Delta}^{\text{\rm top}}(E[[\bfs]]_\Delta,F[[\bfs]]_\Delta) \longrightarrow 
\Hom_{k[[\bfs]]_{\Delta'}}^{\text{\rm top}}(E[[\bfs]]_{\Delta'},F[[\bfs]]_{\Delta'}).
$$
For $E=F$, $\End_{k[[\bfs]]_\Delta}^{\text{\rm top}}(E[[\bfs]]_\Delta)$ is a $k[[\bfs]]_\Delta$-algebra over $A[[\bfs]]_\Delta$.
\end{notation}
\bigskip

\numero \label{nume:tilde}
For each $r= \sum_\beta r_\beta \bfs^\beta \in \Hom_k(E,F)[[\bfs]]_\Delta$ we define $\widetilde{r}: E[[\bfs]]_\Delta \to F[[\bfs]]_\Delta$ by
$$ \widetilde{r} \left( \sum_{\scriptscriptstyle \alpha\in \Delta} e_\alpha \bfs^\alpha \right) :=
\sum_{\scriptscriptstyle \alpha\in \Delta } \left( \sum_{\scriptscriptstyle \beta + \gamma=\alpha} r_\beta(e_\gamma) \right) \bfs^\alpha,
$$
which is obviously a continuous $k[[\bfs]]_\Delta$-linear map. 
\medskip

Let us notice that 
$ \widetilde{r} = \sum_\beta \bfs^\beta  \widetilde{r_\beta}$. It is clear that the map
\begin{equation} \label{eq:tilde-map}
 r\in \Hom_k(E,F)[[\bfs]]_\Delta \longmapsto \widetilde{r} \in \Hom_{k[[\bfs]]_\Delta}^{\text{\rm top}}(E[[\bfs]]_\Delta,F[[\bfs]]_\Delta)
\end{equation}
is $(A[[\bfs]]_\Delta;A[[\bfs]]_\Delta)$-linear.
\medskip

If $f:E[[\bfs]]_\Delta \to F[[\bfs]]_\Delta$ is a continuous $k[[\bfs]]_\Delta$-linear map, let us denote by $f_\alpha:E \to F$, $\alpha\in\Delta $, the $k$-linear maps defined by
$$ f(e) = \sum_{\scriptscriptstyle \alpha\in \Delta} f_\alpha(e) \bfs^\alpha,\quad \forall e\in E.
$$
If $g:E\to F[[\bfs]]_\Delta$ is a $k$-linear map, we denote by $g^e:E[[\bfs]]_\Delta \to F[[\bfs]]_\Delta$ the unique continuous $k[[\bfs]]_\Delta$-linear map extending $g$ to $E[[\bfs]]_\Delta = k[[\bfs]]_\Delta \widehat{\otimes}_k E$. It is given by
$$ g^e \left( \sum_\alpha e_\alpha \bfs^\alpha \right) := \sum_\alpha g(e_\alpha) \bfs^\alpha.
$$
We have a $k[[\bfs]]_\Delta$-bilinear and $A[[\bfs]]_\Delta$-balanced map
$$ \langle -,-\rangle : (r,e) \in \Hom_k(E,F)[[\bfs]]_\Delta \times E[[\bfs]]_\Delta \longmapsto \langle r,e\rangle := \widetilde{r}(e) \in F[[\bfs]]_\Delta.
$$
\begin{lemma} \label{lema:tilde-map}
With the above hypotheses, the following properties hold:
\begin{enumerate}
\item[(1)] The map (\ref{eq:tilde-map})
 is an isomorphism of $(A[[\bfs]]_\Delta;A[[\bfs]]_\Delta)$-bimodules. When $E=F$ it is an isomorphism of 
$k[[\bfs]]_\Delta$-algebras over $A[[\bfs]]_\Delta$.
\item[(2)] The restriction map 
$$f \in  \Hom_{k[[\bfs]]_\Delta}^{\text{\rm top}}(E[[\bfs]]_\Delta,F[[\bfs]]_\Delta) \mapsto f|_E \in \Hom_k(E,F[[\bfs]]_\Delta)$$ is an isomorphism of $(A[[\bfs]]_\Delta;A)$-bimodules.
\end{enumerate}
\end{lemma}

\begin{proof} (1) One easily sees that the inverse map of $r \mapsto \widetilde{r}$ is $f \mapsto \sum_\alpha f_\alpha \bfs^\alpha$.
\medskip

\noindent (2) One easily sees that the inverse map of the restriction map $f \mapsto f|_E$ is $g \mapsto g^e$.
\end{proof}

Let us call $R = \End_k(E)$. As a consequence of the above lemma, the composition of the maps
\begin{equation}  \label{eq:comple_formal_1}
 R[[\bfs]]_\Delta \xrightarrow{r \mapsto \widetilde{r}} \End_{k[[\bfs]]_\Delta}^{\text{\rm top}}(E[[\bfs]]_\Delta) \xrightarrow{f \mapsto f|_E} \Hom_k(E,E[[\bfs]]_\Delta)
\end{equation}
is an isomorphism of $(A[[\bfs]]_\Delta;A)$-bimodules, and so $\Hom_k(E,E[[\bfs]]_\Delta)$ inherits a natural structure of $k[[\bfs]]_\Delta$-algebra over $A[[\bfs]]_\Delta$. Namely, if $g,h\in \Hom_k(E,E[[\bfs]]_\Delta)$
with
$$ g(e)=\sum_{\scriptscriptstyle \alpha\in \Delta} g_\alpha(e)\bfs^\alpha,\ h(e)=\sum_{\scriptscriptstyle \alpha\in \Delta} h_\alpha(e)\bfs^\alpha,\quad \forall e\in E,\quad g_\alpha, h_\alpha \in \Hom_k(E,E),
$$
then the product $h g \in \Hom_k(E,E[[\bfs]]_\Delta)$ is given by
\begin{equation} \label{eq:product-HomEE[[s]]}
 (hg)(e) = \sum_{\scriptscriptstyle \alpha\in \Delta} \left( \sum_{\scriptscriptstyle \beta + \gamma = \alpha} (h_\beta \pcirc g_\gamma)(e) \right) \bfs^\alpha.
\end{equation}

\begin{definition}   \label{defi:external-prod-of-maps}
\index{external product (of endomorphisms)} 
Let $\bfs,\bft$ be sets and $\Delta\subset\N^{(\bfs)}, \nabla\subset\N^{(\bft)}$ non-empty co-ideals. 
For each $f\in\End_{k[[\bfs]]_\Delta}^{\text{\rm top}}(E[[\bfs]]_\Delta)$ and each $g\in\End_{k[[\bft]]_\nabla}^{\text{\rm top}}(E[[\bft]]_\nabla)$, with
$$ f(e)=\sum_{\scriptscriptstyle \alpha \in \Delta} f_\alpha(e) \bfs^\alpha, \quad g(e)=\sum_{\scriptscriptstyle \beta \in \nabla} g_\beta(e) \bft^\beta\quad \forall e\in E,
$$
we define $f\btimes g\in \End_{k[[\bfs\sqcup \bft]]_{\Delta\times\nabla}}^{\text{\rm top}}(E[[\bfs\sqcup \bft]]_{\Delta\times\nabla})$ as $f\btimes g := h^e$, with:
$$ h(x) := \sum_{\scriptscriptstyle (\alpha,\beta) \in \Delta\times\nabla} (f_\alpha \pcirc g_\beta)(x) \bfs^\alpha \bft^\beta\quad \forall x\in E.
$$
\end{definition}

The proof of the following lemma is clear and it is left to the reader.

\begin{lemma} \label{lemma:tilde-otimes}
With the above hypotheses, or each $r\in R[[\bfs]]_\Delta, r'\in R[[\bft]]_\nabla$, we have
$ \widetilde{r\btimes r'} = \widetilde{r}\btimes \widetilde{r'}$ (see Definition \ref{defi:external-x}).
\end{lemma}

\begin{lemma} Let us call $R = \End_k(E)$. For any $r\in R[[\bfs]]_\Delta$, the following properties are equivalent:
\begin{enumerate}
\item[(a)] $r_0 = \Id$.
\item[(b)] The endomorphism  $\widetilde{r}$ is compatible with the natural augmentation $E[[\bfs]]_\Delta \to E$, i.e. $\widetilde{r}(e)  \equiv e \mod \fn^E_0(\bfs)/\Delta_E$ for all $e\in E[[\bfs]]_\Delta$.
\end{enumerate}
Moreover, if the above properties hold, then $\widetilde{r}: E[[\bfs]]_\Delta \to E[[\bfs]]_\Delta$ is a bi-continuous $k[[\bfs]]_\Delta$-linear automorphism.
\end{lemma}

\begin{proof} The equivalence of (a) and (b) is clear. For the second part, $r$ is invertible since $r_0 = \Id$. So $\widetilde{r}$ is invertible too and $\widetilde{r}^{-1} = \widetilde{r^{-1}} $ is also continuous.
\end{proof}

\begin{notation} \label{notacion:pcirc}
We denote:
\begin{eqnarray*}
& 
\Hom_k^\pcirc(E,E[[\bfs]]_\Delta) := 
&\\
&\left\{ f \in \Hom_k(E,E[[\bfs]]_\Delta)\ |\ f(e) \equiv e\!\!\!\!\mod \fn^E_0(\bfs)/\Delta_E\quad  \forall e\in E \right\}, 
\end{eqnarray*}
\begin{eqnarray*}
& 
\Aut_{k[[\bfs]]_\Delta}^\pcirc(E[[\bfs]]_\Delta) 
:=
&\\
& \left\{    
f \in \Aut_{k[[\bfs]]_\Delta}^{\text{\rm top}}(E[[\bfs]]_\Delta)\ |\ f(e) \equiv e_0\!\!\!\!\mod \fn^E_0(\bfs)/\Delta_E\quad \forall e\in E[[\bfs]]_\Delta  \right\}.
\end{eqnarray*}
Let us notice that a $f \in \Hom_k(E,E[[\bfs]]_\Delta)$, given by $f(e) = \sum_{\alpha\in \Delta} f_\alpha(e) \bfs^\alpha$, belongs to $\Hom_k^\pcirc(E,E[[\bfs]]_\Delta)$ if and only if $f_0=\Id_E$.
\end{notation}

The isomorphism in (\ref{eq:comple_formal_1}) gives rise to a group isomorphism
\begin{equation} \label{eq:U-iso-pcirc}
 r\in \U^\bfs(\End_k(E);\Delta) \stackrel{\sim}{\longmapsto} \widetilde{r} \in \Aut_{k[[\bfs]]_\Delta}^\pcirc(E[[\bfs]]_\Delta) 
\end{equation}
and to a bijection 
\begin{equation} \label{eq:Aut-iso-pcirc} 
f\in \Aut_{k[[\bfs]]_\Delta}^\pcirc(E[[\bfs]]_\Delta) \stackrel{\sim}{\longmapsto} f|_E \in \Hom_k^\pcirc(E,E[[\bfs]]_\Delta).
\end{equation}
So, $\Hom_k^\pcirc(E,E[[\bfs]]_\Delta)$ is naturally a group with the product described in (\ref{eq:product-HomEE[[s]]}).

\section{Substitution maps}
\label{nume:substitution_maps}

In this section we will assume that $k$ is a commutative ring and $A$ a commutative $k$-algebra.
The following notation will be used extensively.

\begin{notation} 
\begin{enumerate}
\item[(i)] For each integer $r\geq 0$ let us denote 
$[r]:= \{1,\dots,r\}$ if $r>0$ and $[0]=\emptyset$.
\item[(ii)] Let $\bfs$ be a set. Maps from a set $\Lambda$ to $\N^{(\bfs)}$ will be usually denoted as 
$\alpha^\bullet: l \in  \Lambda \longmapsto \alpha^l\in \N^{(\bfs)}$, and its {\em support}
is defined by $\supp \alpha^\bullet := \{l \in  \Lambda\ |\ \alpha^l\neq 0\}$.

\item[(iii)] For each set $\Lambda$ and for each map $\alpha^\bullet: \Lambda \to \N^{(\bfs)}$ with finite support, its {\em norm} is defined by $|\alpha^\bullet| := \sum_{l\in \supp \alpha^\bullet} \alpha^l =
\sum_{l\in \Lambda} \alpha^l$. When $\Lambda=\emptyset$, the unique map $\Lambda \to \N^{(\bfs)}$ is the inclusion $\emptyset \hookrightarrow \N^{(\bfs)}$ and its norm is $0\in \N^{(\bfs)}$.

\item[(iv)] If $\Lambda$ is a set and $e\in \N^{(\bfs)}$, we define
$$ 
\Par^\smallcirc(e,\Lambda) := \{ \alpha^\bullet : \Lambda \to \N^{(\bfs)}\ |\ \# \supp \alpha^\bullet < +\infty, |\alpha^\bullet| = e\}.
$$ 
If $F$ is a finite set and $e\in \N^{(\bfs)}$, we define
$$  \Par(e,F) := \{ \alpha : F \to \N^{(\bfs)}_*\ |\ |\alpha| = e\} \subset
\Par^\smallcirc(e,F) .
$$ 
It is clear that $\Par(e,F) = \emptyset$ whenever $\# F > |e|$, $\Par^\smallcirc(e,\emptyset)=\emptyset$ if $e\neq 0$, $\Par^\smallcirc(0,\Lambda)$ consists of only the constant map $0$ and that $\Par(0,\emptyset)=\Par^\smallcirc(0,\emptyset)$ consists of only the inclusion $\emptyset \hookrightarrow \N^{(\bfs)}_*$. If $\# F =1$ and $e\neq 0$, then $\Par(e,F)$ also consists of only one map: the constant map with value $e$. 

The natural map $\displaystyle \coprod_{\scriptstyle F\in \Partes_f(\Lambda)}  \Par(e, F) \longrightarrow \Par^\smallcirc(e,\Lambda)$ is obviously a bijection.

If $r\geq 0$ is an integer, we will denote 
$\Par(e,r) := \Par(e,[r])$.

\item[(v)] Assume that $\Lambda$ is a finite set, $\bft$ is an arbitrary set and $\pi:\Lambda \to \bft$ is map.
Then, there is a natural bijection 
$$    \Par^\smallcirc(e,\Lambda)  \leftrightarrow 
\coprod_{\scriptstyle e^\bullet \in \Par^\smallcirc(e,\bft)} \prod_{\scriptstyle t\in \bft} \Par^\smallcirc(e^t,\pi^{-1}(t)) = \coprod_{\scriptstyle e^\bullet \in \Par^\smallcirc(e,\bft)} \prod_{\scriptstyle t\in \supp e^\bullet} \Par^\smallcirc(e^t,\pi^{-1}(t)).
$$ 
Namely, to each $\alpha^\bullet \in \Par^\smallcirc(e,\Lambda)$ we associate $e^\bullet \in \Par^\smallcirc(e,\bft)$ defined by $e^t = \sum_{\pi(l)=t} \alpha^l$, and $\{ \alpha^{t\bullet} \}_{t\in \bft} \in \prod_{t\in \bft} \Par^\smallcirc(e^t,\pi^{-1}(t))$ with $\alpha^{t\bullet} =\alpha^\bullet|_{\pi^{-1}(t)}$. Let us notice that if for some $t_0\in\bft$ one has $\pi^{-1}(t_0)=\emptyset$ and $e^{t_0}\neq 0$, then $\Par^\smallcirc(e^{t_0},\pi^{-1}(t_0))=\emptyset$ and so $\prod_{t\in \bft} \Par^\smallcirc(e^t,\pi^{-1}(t)) =\emptyset$. Hence
\begin{eqnarray*}
&\displaystyle
\coprod_{\scriptstyle e^\bullet \in \Par^\smallcirc(e,\bft)} \prod_{\scriptstyle t\in \bft} \Par^\smallcirc(e^t,\Lambda_t) = \coprod_{\scriptstyle e^\bullet \in \Par_{\pi}^\smallcirc(e,\bft)} \prod_{\scriptstyle t\in \bft} \Par^\smallcirc(e^t,\pi^{-1}(t)) =
&\\
&\displaystyle
\coprod_{\scriptstyle e^\bullet \in \Par_{\pi}^\smallcirc(e,\bft)} \prod_{\scriptstyle t\in \supp e^\bullet} \Par^\smallcirc(e^t,\pi^{-1}(t)),
\end{eqnarray*}
where $\Par_{\pi}^\smallcirc(e,\bft)$ is the subset of $\Par^\smallcirc(e,\bft)$ whose elements are the  $e^\bullet \in \Par^\smallcirc(e,\bft)$ such that 
$ e^t = 0 $ whenever $\pi^{-1}(t)=\emptyset$ and $|e^t| \geq \# \pi^{-1}(t)$ otherwise.
\medskip

The preceding bijection induces a bijection 
\begin{equation} \label{eq:bijec-par}
  \Par(e,\Lambda) \longleftrightarrow
\coprod_{\scriptstyle e^\bullet \in \Par_{\pi}^\smallcirc(e,\bft)} \prod_{\scriptstyle t\in \bft} \Par(e^t,\pi^{-1}(t))=
\coprod_{\scriptstyle e^\bullet \in \Par_{\pi}^\smallcirc(e,\bft)} \prod_{\scriptstyle t\in \supp e^\bullet} \Par(e^t,\pi^{-1}(t)).
\end{equation}
\item[(vi)] If $\alpha\in \N^{(\bft)}$, we denote
$$ [\alpha] := \{ (t,r) \in \bft \times \N_*\ |\ 1\leq r\leq \alpha_t \}
$$
endowed with the projection $\pi:[\alpha] \to \bft$. 
It is clear that $|\alpha|= \# [\alpha]$, and so
 $\alpha=0$ $\Longleftrightarrow$ $[\alpha] =\emptyset$. We denote $\Par(e,\alpha):= \Par(e,[\alpha])$. Elements in $ \Par(e,\alpha)$ will be written as 
$$\mathcal{b}^{\bullet\bullet}: (t,r)\in [\alpha]  \longmapsto \mathcal{b}^{tr}\in \N^{(\bfs)},\quad \text{with}\ \sum_{\scriptscriptstyle (t,r)\in [\alpha]} \mathcal{b}^{tr} =e.
$$
For each $\mathcal{b}^{\bullet\bullet} \in \Par(e,\alpha)$ and each $t\in\bft$, we denote
$$ \mathcal{b}^{t\bullet}: r\in [\alpha_t] \longmapsto \mathcal{b}^{tr} \in \N^{(\bfs)},\quad
[\mathcal{b}]^\bullet: t\in \bft \longmapsto [\mathcal{b}]^t := |\mathcal{b}^{t\bullet}|= \sum_{\scriptscriptstyle  r=1}^{\scriptscriptstyle \alpha_t} \mathcal{b}^{tr} \in \N^{(\bfs)}.
$$
Notice that $|[\mathcal{b}]^t| \geq \alpha_t$, $[\mathcal{b}]^t=0$ whenever $\alpha_t=0$ and
$\left| [\mathcal{b}]^\bullet \right| = e$. The bijection (\ref{eq:bijec-par}) gives rise to a bijection
\begin{equation} \label{eq:bijec-par-alpha}
 \Par(e,\alpha) \longleftrightarrow
\coprod_{\scriptstyle e^\bullet \in {\Par}_\alpha^\smallcirc(e,\bft)} \prod_{\scriptstyle t\in \bft} \Par(e^t,\alpha_t)  = \coprod_{\scriptstyle e^\bullet \in {\Par}_\alpha^\smallcirc(e,\bft)} \prod_{\scriptstyle t\in \supp e^\bullet} \Par(e^t,\alpha_t),
\end{equation}
where $ {\Par}_\alpha^\smallcirc(e,\bft)$ is the subset of $\Par^\smallcirc(e,\bft)$ whose elements are the  $e^\bullet \in \Par^\smallcirc(e,\bft)$ such that 
$ e^t = 0 $ if $\alpha_t=0$ and $|e^t| \geq \alpha_t$ otherwise.
\end{enumerate}
\end{notation}

\numero \label{nume:explicit-substitution}
Let $\bft$, $\bfu$ be sets and $\Delta\subset \N^{(\bfu)}$ a non-empty co-ideal.
Let $\varphi_0: A[\bft] \xrightarrow{} A[[\bfu]]_\Delta $ be an $A$-algebra map given by:
$$\varphi_0(t)=: c^t = 
\sum_{\substack{\scriptscriptstyle \beta\in\Delta\\ \scriptscriptstyle 0<|\beta| }} c^t_\beta \bfu^\beta \in \fn^A_0(\bfu)/\Delta_A  \subset  A[[\bfu]]_\Delta,\ \  t\in \bft.
$$
Let us write down the expression of the image $\varphi_0(a)$ of any $a\in A[\bft]$ in terms of the coefficients of $a$ and the $c^t, t\in \bft$. 
First, for each $r\geq  0$ and for each $t\in \bft$ we have
$$ \varphi_0(t^r)= (c^t)^r = \dots = \sum_{\substack{\scriptscriptstyle  e\in \Delta\\ \scriptscriptstyle   |e|\geq r }} \left( \sum_{\scriptscriptstyle  \beta^\bullet \in \Par(e,r)}  \prod_{\scriptstyle k=1}^{\scriptstyle r} c^t_{\beta^k}\right) \bfu^e.
$$
Observe that
\begin{equation} \label{eq:convention}
\sum_{\scriptscriptstyle  \beta^\bullet \in \Par(e,r)}  \prod_{\scriptscriptstyle k=1}^r c^t_{\beta^k} = \left\{ \begin{array}{ll} 1 & \text{if $|e|=r=0$}\\
0 & \text{if $|e|>r=0$.} \end{array} \right.
\end{equation}
So, for each $\alpha \in \N^{(\bft)}$ we have
\begin{eqnarray*}  
&\displaystyle  \varphi_0(\bft^\alpha)  = \prod_{\scriptscriptstyle t\in \bft} (c^t)^{\alpha_t} = \prod_{\scriptscriptstyle t\in \supp \alpha} (c^t)^{\alpha_t} =
\prod_{\scriptscriptstyle t\in \supp \alpha} \left( 
\sum_{\substack{\scriptscriptstyle  e\in \Delta\\ \scriptscriptstyle   |e|\geq \alpha_t }} \left( 
\sum_{\scriptscriptstyle \beta^\bullet \in \Par(e,\alpha_t)}
\prod_{\scriptscriptstyle k=1}^{\scriptscriptstyle \alpha_t} c^t_{\beta^k}\right) \bfu^e
\right)=
&\\
&\displaystyle    \sum_{\substack{\scriptscriptstyle  e^t\in \Delta, t\in \supp \alpha\\ \scriptscriptstyle  |e^t|\geq \alpha_t }} \prod_{\scriptscriptstyle t\in \supp \alpha} 
\left(\left( 
\sum_{\scriptscriptstyle \beta^\bullet \in \Par(e^t,\alpha_t)} 
\prod_{\scriptscriptstyle k=1}^{\scriptscriptstyle \alpha_t} c^t_{\beta^k}\right) \bfu^{e^t}\right) =
\end{eqnarray*}
\begin{eqnarray*}  
&\displaystyle    \sum_{\substack{\scriptscriptstyle  e^t\in \Delta, t\in \supp \alpha\\ \scriptscriptstyle  |e^t|\geq \alpha_t }} 
\left(
\sum_{\substack{\scriptscriptstyle \beta^{t \bullet} \in \Par(e^t,\alpha_t)\\ \scriptscriptstyle  t\in \supp \alpha}}
\left( \prod_{\scriptscriptstyle t\in \supp \alpha} 
\prod_{\scriptscriptstyle k=1}^{\scriptscriptstyle \alpha_t} c^t_{\beta^{tk}}\right) \right)
\left( \prod_{\scriptscriptstyle t\in \supp \alpha} 
\bfu^{e^t}\right) =
&\\
&\displaystyle  
\sum_{\substack{\scriptscriptstyle  e\in \Delta\\ \scriptscriptstyle  |e|\geq |\alpha| }}
\left( 
\sum_{\substack{\scriptscriptstyle  e^t\in \Delta, t\in \supp \alpha\\ \scriptscriptstyle   |e^t|\geq \alpha_t\\ \scriptscriptstyle  |e^\bullet|=e }} 
\left(
\sum_{\substack{\scriptscriptstyle \beta^{t \bullet} \in \Par(e^t,\alpha_t)\\ \scriptscriptstyle  t\in \supp \alpha}}
\left( \prod_{\scriptscriptstyle t\in \supp \alpha} 
\prod_{\scriptscriptstyle k=1}^{\scriptscriptstyle \alpha_t} c^t_{\beta^{tk}}\right) \right) 
\right)
\bfu^e =
&\\
&\displaystyle  
\sum_{\substack{\scriptscriptstyle  e\in \Delta\\ \scriptscriptstyle |e|\geq |\alpha| }} 
\left(
\sum_{\scriptscriptstyle e^\bullet \in \Par_\alpha^\smallcirc(e,\bft)}
\left(
\sum_{\substack{\scriptscriptstyle \beta^{t \bullet} \in \Par(e^t,\alpha_t)\\ \scriptscriptstyle  t\in \supp \alpha}}
\left( \prod_{\scriptscriptstyle t\in \supp \alpha} 
\prod_{\scriptscriptstyle k=1}^{\scriptscriptstyle \alpha_t} c^t_{\beta^{tk}}\right)
\right)
\right) \bfu^e=
\sum_{\substack{\scriptscriptstyle  e\in \Delta\\ \scriptscriptstyle  |e|\geq |\alpha| }} 
{\bf C}_e(\varphi_0,\alpha) \bfu^e,
\end{eqnarray*}
with (see (\ref{eq:bijec-par-alpha})):
\begin{equation} \label{eq:explicit-C_e(varphi,alpha)}
{\bf C}_e(\varphi_0,\alpha)=  \sum_{\scriptscriptstyle \beta^{\bullet\bullet} \in \Par(e,\alpha)}
C_{\beta^{\bullet\bullet}},\quad
 C_{\beta^{\bullet\bullet}} = \prod_{\scriptstyle t\in \supp \alpha} \prod_{\scriptstyle r=1}^{\scriptstyle \alpha_t}  c^t_{\beta^{tr}},\quad
 \text{for\ }\ |\alpha| \leq |e|.
\end{equation}
We have
${\bf C}_0(\varphi_0,0) = 1$ and ${\bf C}_e(\varphi_0,0)=0$ for $e\neq 0$. 
For a fixed $e\in \N^{(\bfu)}$ the support of any $\alpha\in \N^{(\bft)}$ such that $|\alpha|\leq |e|$ and ${\bf C}_e(\varphi_0,\alpha)\neq 0$ is contained in the set
$$ \bigcup_{\substack{\scriptscriptstyle \beta \in \Delta\\ \scriptscriptstyle \beta \leq e}} \{t\in \bft\ |\ c^t_\beta \neq 0\}
$$
and so the set of such $\alpha$'s is finite provided that property (\ref{eq:cond-fini-c}) holds. 
We conclude that
\begin{equation} \label{eq:exp-substi}
\varphi_0\left( \sum_{\scriptscriptstyle\alpha\in\N^{(\bft)}} a_\alpha \bft^\alpha \right) = \sum_{\scriptscriptstyle \alpha\in \N^{(\bft)}} a_\alpha c^\alpha=
\sum_{\scriptscriptstyle e\in \Delta}
\left(  \sum_{\substack{\scriptscriptstyle \alpha\in\N^{(\bft)}\\ \scriptscriptstyle  |\alpha|\leq |e| }} {\bf C}_e(\varphi_0,\alpha) a_\alpha \right) \bfu^e.
\end{equation}

Observe that for each non-zero $\alpha\in\N^{(\bft)}$ we have:
\begin{equation}  \label{eq:condition-supp-substi}
\supp(\varphi_0(\bft^\alpha)) = \supp\left( \prod_{\scriptstyle t\in\supp \alpha} \left(c^t\right)^{\alpha_t}\right)
 \subset \sum_{\scriptscriptstyle t\in \supp (\alpha)} \alpha_t \cdot \supp(c^t).
\end{equation}

Let us notice that if we assign the weight $|\beta|$ to
$c^t_{\beta}$, then 
${\bf C}_e(\varphi_0,\alpha)$ is a quasi-homogeneous polynomial in the variables $c^t_\beta$, $t\in\supp \alpha$, $|\beta| \leq |e|$, of weight $|e|$.
\medskip

The proof of the following lemma is easy and it is left to the reader.

\begin{lemma} \label{lemma:behavior_C_alpha}
For each $e\in\Delta$ and for each $\alpha\in \N^{(\bft)}$ with $0<|\alpha|\leq |e|$,
the following properties hold:
\begin{enumerate}
\item[(1)] If $|\alpha|=1$, then ${\bf C}_e(\varphi_0,\alpha) = c_e^s$, where $\supp \alpha = \{s\}$, i.e. $\alpha = \bft^s$ ($\bft^s_t = \delta_{st}$).
\item[(2)] If $|\alpha|=|e|$, then
$$
{\bf C}_e(\varphi_0,\alpha) = 
\sum_{\substack{\scriptscriptstyle  e^t\in \Delta, t\in \supp \alpha\\ \scriptscriptstyle  |e^t|= \alpha_t, |e^\bullet|=e }} 
\left( \prod_{\scriptscriptstyle t\in \supp \alpha} 
\prod_{\scriptscriptstyle v\in \supp e^t } \left(c^t_{\bfu^v}\right)^{e^t_v} \right).
$$
\end{enumerate}
\end{lemma}

\begin{proposition} Let $\bft, \bfu$ be sets and $\Delta\subset \N^{(\bfu)}$ a non-empty co-ideal. For each family
$$c=\left\{c^t = \sum_{\substack{\scriptscriptstyle \beta\in \Delta\\ \scriptscriptstyle \beta\neq 0}} c^t_\beta \bfu^\beta \in \fn^A_0(\bfu)/\Delta_A \subset  A[[\bfu]]_\Delta,\ t\in \bft\right\}
$$ 
(we are assuming that $c^t_0 = 0$) satisfying the following property
\begin{equation} \label{eq:cond-fini-c}
 \#\{t\in \bft\ |\ c^t_\beta \neq 0\} < \infty\quad\quad \text{for all\ }\ \beta \in \Delta,
\end{equation}
there is a unique continuous $A$-algebra map $\varphi: A[[\bft]] \xrightarrow{} A[[\bfu]]_\Delta$ such that $\varphi(t) = c^t$ for all $t\in\bft$. Moreover, if $\nabla\subset \N^{(\bft)}$ is a non-empty co-ideal such that $\varphi(\nabla_A) =0$, then $\varphi$ induces a unique continuous $A$-algebra map $A[[\bft]]_\nabla \xrightarrow{} A[[\bfu]]_\Delta$ sending (the class of) each $t\in\bft$ to $c^t$.
\end{proposition}

\begin{proof}  
Let us consider the unique $A$-algebra map $\varphi_0: A[\bft] \xrightarrow{} A[[\bfu]]_\Delta$ defined by $\varphi_0(t) = c^t$ for all $t\in\bft$. From (\ref{eq:explicit-C_e(varphi,alpha)}) and  (\ref{eq:exp-substi}) in \ref{nume:explicit-substitution}, we know that  
$$
\varphi_0\left( \sum_{\substack{\scriptscriptstyle\alpha\in\N^{(\bft)}\\ \scriptscriptstyle \text{finite}}} a_\alpha \bft^\alpha \right) = 
\sum_{\scriptscriptstyle e\in \Delta}
\left(  \sum_{\substack{\scriptscriptstyle \alpha\in\N^{(\bft)}\\ \scriptscriptstyle  |\alpha|\leq |e| }} {\bf C}_e(\varphi_0,\alpha) a_\alpha \right) \bfu^e.
$$
Since for a fixed $e\in \N^{(\bfu)}$ the support of the $\alpha\in \N^{(\bft)}$ such that $|\alpha|\leq |e|$ and ${\bf C}_e(\varphi_0,\alpha)\neq 0$ is contained in the finite set
$$ \bigcup_{\substack{\scriptscriptstyle \beta \in \Delta\\ \scriptscriptstyle \beta \leq e}} \{t\in \bft\ |\ c^t_\beta \neq 0\},
$$
the set of such $\alpha$'s is always finite and we deduce that $\varphi_0$ is continuous, and so there is a unique continuous extension $\varphi: A[[\bft]] \xrightarrow{} A[[\bfu]]_\Delta$ such that $\varphi(t) = c^t$ for all $t\in\bft$.

The last part is clear.
\end{proof}

\begin{remark} Let us notice that, after (\ref{eq:condition-supp-substi}),  to get the equality $\varphi(\nabla_A) =0$ in the above proposition it is enough to have for each $\alpha \in \nabla^c$ (actually, it will be enough to consider the $\alpha \in \nabla^c$ minimal with respect to the ordering $\leq$ in $\N^{(\bft)}$):
$$ \sum_{\scriptscriptstyle t\in \supp (\alpha)} \alpha_t \cdot \supp(c^t) \subset \Delta^c.
$$
\end{remark}

\begin{definition}  \label{def:substitution_maps} Let $\nabla\subset \N^{(\bft)}, \Delta\subset \N^{(\bfu)}$ be non-empty co-ideals. An $A$-algebra map $\varphi:A[[\bft]]_\nabla \xrightarrow{} A[[\bfu]]_\Delta$
will be called a {\em substitution map} \index{substitution map} if the following properties hold:
\begin{enumerate}
\item[(1)] $\varphi$ is continuous.
\item[(2)] $\varphi(t)\in \fn^A_0(\bfu)/\Delta_A$ for all $t\in \bft$.
\item[(3)] The family $c=\{\varphi(t), t\in \bft\}$ satisfies property (\ref{eq:cond-fini-c}).
\end{enumerate}
\medskip

The set of substitution maps $A[[\bft]]_\nabla \xrightarrow{} A[[\bfu]]_\Delta$ will be denoted by $\Sub_A(\bft,\bfu;\nabla,\Delta)$.
The {\em trivial} substitution map \index{trivial substitution map} $A[[\bft]]_\nabla \xrightarrow{} A[[\bfu]]_\Delta $ is the one sending any $t\in \bft$ to $0$. It will be denoted by $\mathbf{0}$. 
\end{definition}

\begin{remark}
In the above definition, 
a such $\varphi$ is uniquely determined by the family $c=\{\varphi(t), t\in \bft\}$, and will be called the {\em substitution map associated} with $c$. Namely, the family $c$ can be lifted to $A[[\bfu]]$ by means of the natural $A$-linear scission $A[[\bfu]]_\Delta \hookrightarrow A[[\bfu]]$ and we may consider the unique continuous $A$-algebra map $\psi: A[[\bft]] \to A[[\bfu]]$ such that $\psi(s) = c^s$ for all $s\in\bfs$. Since $\varphi$ is continuous, we have a commutative diagram
$$
\begin{CD}
A[[\bft]] @>{\psi}>>A[[\bfu]] \\
@V{\text{proj.}}VV @VV{\text{proj.}}V\\
A[[\bft]]_\nabla @>{\varphi}>>A[[\bfu]]_\Delta,
\end{CD}
$$
and so $\psi(\nabla_A) \subset \Delta_A$. Then, we may indentify
$$ \Sub_A(\bft,\bfu;\nabla,\Delta) \equiv \left\{\overline{\psi}\in \Sub_A(\bft,\bfu;\N^{(\bft)},\Delta)\ |\ \overline{\psi}(\nabla_A)=0 \right\}.
$$
For $\alpha \in \nabla$ and $e\in \Delta$ with $|\alpha|\leq |e|$ we will write ${\bf C}_e(\varphi,\alpha) := {\bf C}_e(\varphi_0,\alpha)$, where $\varphi_0: A[\bft] \to A[[\bfu]]_\Delta$ is the $A$-algebra map given by $\varphi_0(t) = \varphi(t)$ for all $t\in\bft$ (see (\ref{eq:explicit-C_e(varphi,alpha)}) in \ref{nume:explicit-substitution}).
\end{remark}

\begin{example} For any family of integers $\nu =\{\nu_t\geq 1, t\in\bft\}$, we will denote
$ [\nu]: A[[\bft]]_\nabla \xrightarrow{} A[[\bft]]_{\nu\nabla}$ the substitution map determined by
$[\nu](t) = t^{\nu_t}$ for all $t\in\bft$, where
$$ \nu\nabla := \{\gamma \in \N^{(\bft)}\ |\ \exists \alpha \in\nabla, \gamma\leq \nu \alpha \}.
$$
We obviously have $[\nu  \nu'] = [\nu] \pcirc [\nu']$. 
\end{example}

\begin{lemma} \label{lemma:compos_subst_maps}
The composition of two substitution maps $A[[\bft]]_\nabla \stackrel{\varphi}{\to} A[[\bfu]]_\Delta \stackrel{\psi}{\to} A[[\bfs]]_\Omega$ is a substitution map and we have
$$
{\bf C}_f(\psi \pcirc \varphi,\alpha) = \sum_{\substack{\scriptscriptstyle  e\in \Delta\\ \scriptscriptstyle  |f|\geq |e|\geq |\alpha| }} 
{\bf C}_e(\varphi,\alpha) {\bf C}_f(\psi,e),\quad \forall f\in\Omega, \forall \alpha\in\nabla, |\alpha|\leq |f|.
$$
Moreover, if one of the substitution maps is trivial, then the composition is trivial too.
\end{lemma}

\begin{proof} 
Properties (1) and (2) in Definition \ref{def:substitution_maps} are clear. Let us see property (3).
For each $t\in\bft$ let us write: 
$$ \varphi(t) =: c^t = 
\sum_{\substack{\scriptscriptstyle \beta\in\Delta\\ \scriptscriptstyle 0<|\beta| }} c^t_\beta \bfu^\beta \in \fn^A_0(\bfu)/\Delta_A  \subset  A[[\bfu]]_\Delta,
$$
and so
\begin{eqnarray*}
&\displaystyle
(\psi \pcirc \varphi)(t) = \psi\left(\sum_{\substack{\scriptscriptstyle \beta\in\Delta\\ \scriptscriptstyle 0<|\beta| }} c^t_\beta \bfu^\beta \right) = \sum_{\substack{\scriptscriptstyle \beta\in\Delta\\ \scriptscriptstyle 0<|\beta| }} c^t_\beta 
\left(
\sum_{\substack{\scriptscriptstyle  f\in \Omega\\ \scriptscriptstyle  |f|\geq |\beta| }} 
{\bf C}_f(\psi,\beta) \bfs^f
\right)=
\sum_{\substack{\scriptscriptstyle  f\in \Omega\\ \scriptscriptstyle  |f|>0 }} 
d^t_f
\bfs^f
\end{eqnarray*}
with
$$ d^t_f=
\sum_{\substack{\scriptscriptstyle \beta\in\Delta\\ \scriptscriptstyle 0<|\beta|\leq |f| }} c^t_\beta {\bf C}_f(\psi,\beta) 
$$
and for a fixed $f\in\Omega$ the set
$$ \{t\in\bft\ |\ d^t_f\neq 0\} \subset \bigcup_{\substack{\scriptscriptstyle  \beta\in\nabla, |\beta|\leq|f|\\
\scriptscriptstyle {\bf C}_f(\psi,\beta) \neq 0}} \{t\in\bft\ |\ c^t_\beta\neq 0\}
$$
is finite.
On the other hand
\begin{eqnarray*}
&\displaystyle
 (\psi\pcirc\varphi)(\bft^\alpha) = \psi\left(
\sum_{\substack{\scriptscriptstyle  e\in \Delta\\ \scriptscriptstyle  |e|\geq |\alpha| }} 
{\bf C}_e(\varphi,\alpha) \bfu^e
\right) = 
\sum_{\substack{\scriptscriptstyle  e\in \Delta\\ \scriptscriptstyle  |e|\geq |\alpha| }} 
{\bf C}_e(\varphi,\alpha) 
\left(
\sum_{\substack{\scriptscriptstyle  f\in \Omega\\ \scriptscriptstyle  |f|\geq |e| }} 
{\bf C}_f(\psi,e) \bfs^f
\right)=
&
\\
&\displaystyle
\sum_{\substack{\scriptscriptstyle  f\in \Omega\\ \scriptscriptstyle  |f|\geq |\alpha| }} 
\left(
\sum_{\substack{\scriptscriptstyle  e\in \Delta\\ \scriptscriptstyle  |f|\geq |e|\geq |\alpha| }} 
{\bf C}_e(\varphi,\alpha) {\bf C}_f(\psi,e)
\right) \bfu^f
\end{eqnarray*}
and so
$$
{\bf C}_f(\psi\pcirc\varphi,\alpha) = \sum_{\substack{\scriptscriptstyle  e\in \Delta\\ \scriptscriptstyle  |f|\geq |e|\geq |\alpha| }} 
{\bf C}_e(\varphi,\alpha) {\bf C}_f(\psi,e),\quad \forall f\in\Omega, \forall \alpha\in\nabla, |\alpha|\leq |f|.
$$
\end{proof}

If $B$ is a commutative $A$-algebra, then any subtitution map $\varphi: A[[\bfs]]_\nabla \to A[[\bft]]_\Delta$ induces a natural substitution map $\varphi_B:B[[\bfs]]_\nabla \to B[[\bft]]_\Delta$ making the following diagram commutative
$$
\begin{CD}
B\widehat{\otimes}_A A[[\bfs]]_\nabla @>{\Id \widehat{\otimes} \varphi}>> B\widehat{\otimes}_A A[[\bft]]_\Delta\\
@V{\text{nat.}}V{\simeq}V @V{\simeq}V{\text{nat.}}V\\
B[[\bfs]]_\nabla @>{\varphi_B}>> B[[\bft]]_\Delta.
\end{CD}
$$

\numero \label{nume:operaciones-con-substitutions} 
For any substitution map $\varphi: A[[\bfs]]_\nabla \xrightarrow{} A[[\bft]]_\Delta$ and for any integer $n\geq 0$ we have $\varphi(\nabla^n_A/\nabla_A) \subset \Delta^n_A/\Delta_A$ and so there are induced substitution maps $\tau_{n}(\varphi):A[[\bfs]]_{\nabla^n}  \to A[[\bft]]_{\Delta^n}$ making commutative the following diagram
$$
\begin{CD}
A[[\bfs]]_\nabla  @>{\varphi}>>  A[[\bft]]_\Delta\\
@V{\text{nat.}}VV @VV{\text{nat.}}V\\
A[[\bfs]]_{\nabla^n}  @>{\tau_{n}(\varphi)}>>  A[[\bft]]_{\Delta^n}.
\end{CD}
$$

Moreover, if $\varphi$ is the substitution map associated with a family $c=\{c^s, s\in \bfs\}$,
$$c^s = \sum_{\scriptscriptstyle \beta\in \Delta} c^s_\beta \bft^\beta \in \fn^A_0(\bft)/\Delta_A\subset  A[[\bft]]_\Delta,
$$ 
then $\tau_{n}(\varphi)$ is the substitution map associated with the family $\tau_{n}(c)=\{\tau_{n}(c)^s, s\in \bfs\}$, with
$$\tau_{n}(c)^s := \sum_{\substack{\scriptscriptstyle \beta\in \Delta\\ \scriptscriptstyle |\beta|\leq n}} c^s_\beta \bft^\beta \in \fn^A_0(\bft)/\Delta^n_A\subset  A[[\bft]]_{\Delta^n}.
$$
So, we have truncations $\tau_n: \Sub_A(\bfs,\bft;\nabla,\Delta)  \longrightarrow  \Sub_A(\bfs,\bft;\nabla^n,\Delta^n)$, for $n\geq 0$.
\medskip

We may also add two substitution maps $\varphi,\varphi':A[[\bfs]] \xrightarrow{} A[[\bft]]_\Delta$ to obtain
a new substitution map $\varphi+\varphi':A[[\bfs]] \xrightarrow{} A[[\bft]]_\Delta$ 
determined by\footnote{Pay attention that $(\varphi+\varphi')(r) \neq \varphi(r) + \varphi'(r)$ for arbitrary $r\in A[[\bfs]]_\nabla$.}:
$$ (\varphi+\varphi')(s) = \varphi(s)+\varphi'(s),\quad \text{for all\ }\ s\in \bfs.
$$
It is clear that $\Sub_A(\bfs,\bft;\N^{(\bfs)},\Delta)$ becomes an abelian group with the addition, the zero element being the trivial substitution map $\mathbf{0}$.
\smallskip

If $\psi: A[[\bft]]_\Delta \xrightarrow{} A[[\bfu]]_\Omega$ is another substitution map, we clearly have
$$\psi \pcirc (\varphi+\varphi') = \psi \pcirc \varphi + \psi \pcirc \varphi'.$$ 
However, if $\psi: A[[\bfu]] \xrightarrow{} A[[\bfs]]$ is a substitution map, we have in general
$$ (\varphi+\varphi')\pcirc \psi \neq \varphi\pcirc \psi +\varphi'\pcirc \psi.
$$

\begin{definition} We say that a substitution map $\varphi: A[[\bft]]_\nabla \xrightarrow{} A[[\bfu]]_\Delta $ has {\em constant coefficients} \index{substitution maps with constant coefficients}
if $c^t_\beta \in k$ for all $t\in \bft$ and all $\beta\in\Delta$, where
$$\varphi(t)=c^t = 
\sum_{\substack{\scriptscriptstyle \beta\in\Delta\\ \scriptscriptstyle 0<|\beta| }} c^t_\beta \bfu^\beta \in \fn^A_0(\bfu)/\Delta_A  \subset  A[[\bfu]]_\Delta.
$$
This is equivalent to saying that
 ${\bf C}_e(\varphi,\alpha)\in k$ for all $e\in\Delta$ and for all $\alpha\in \nabla$ with $0<|\alpha|\leq |e|$. Substitution maps which constant coefficients are induced by substitution maps $k[[\bft]]_\nabla \xrightarrow{} k[[\bfu]]_\Delta$.

We say that a substitution map $\varphi: A[[\bft]]_\nabla \xrightarrow{} A[[\bfu]]_\Delta $ is {\em combinatorial} \index{combinatorial substitution maps} if $\varphi(t) \in \bfu$ for all $t\in\bft$. A combinatorial substitution map has constant coefficients and is determined by (and determines) a map $\bft \to \bfu$, necessarily with finite fibers. If $\iota : \bft \to \bfu$ is such a map, we will also denote by $\iota: A[[\bft]]_\nabla \xrightarrow{} A[[\bfu]]_{\iota_*(\nabla)} $ the corresponding substitution map, with
$$ \iota_*(\nabla):= \{\beta \in \N^{(\bfu)}\ |\ \beta \pcirc \iota \in \nabla\}.
$$
\end{definition}

\numero \label{nume:cont-maps-power-series-rings}
Let $\varphi: A[[\bfs]]_\nabla \xrightarrow{} A[[\bft]]_\Delta $ be a continuous $A$-linear map. It is determined by the family
$ K =\{ K_{e,\alpha}, e\in \Delta, \alpha \in\nabla \} \subset A$, with
$\displaystyle
\varphi( \bfs^\alpha ) = 
\sum_{ \scriptscriptstyle e\in \Delta}    K_{e,\alpha} \bft^e$. We will assume that 
\begin{itemize}
\item $\varphi$ is compatible with the order filtration, i.e. $\varphi (\nabla^n_A/\nabla_A) \subset \Delta^n_A/\Delta_A$ for all $n\geq 0$. 
\item $\varphi$ is compatible with the natural augmentations $A[[\bfs]]_\nabla \to A$ and $A[[\bft]]_\Delta \to A$.
\end{itemize}
These properties are equivalent to the fact that $K_{e,\alpha}=0$ whenever $|\alpha|>|e|$ and $K_{0,0}=1$.
\medskip

Let $ K =\{ K_{e,\alpha}$, $e\in \Delta, \alpha \in\nabla$, $|\alpha|\leq |e| \}$ be a family of elements of $A$ with 
$$\# \{\alpha\in\nabla\ |\ |\alpha|\leq |e|, K_{e,\alpha}\neq 0\} < +\infty,\ \  \forall e\in \Delta,
$$
and $K_{0,0} = 1$, 
and 
let $\varphi: A[[\bfs]]_\nabla \xrightarrow{} A[[\bft]]_\Delta $ be the $A$-linear map given by 
$$ \varphi\left( \sum_{\scriptscriptstyle\alpha\in\nabla} a_\alpha \bfs^\alpha \right) = 
\sum_{\scriptscriptstyle e\in \Delta}
\left(  \sum_{\substack{\scriptscriptstyle \alpha\in\nabla\\ \scriptscriptstyle |\alpha|\leq |e| }} K_{e,\alpha} a_\alpha \right) \bft^e.
$$
It is clearly continuous and since 
$\displaystyle
\varphi( \bfs^\alpha ) = 
\sum_{\substack{ \scriptscriptstyle e\in\Delta\\ \scriptscriptstyle  |\alpha|\leq |e|}}    K_{e,\alpha} \bft^e, 
$
it determines the family $K$.

\begin{proposition} \label{prop:identidad-Cs}
With the above notations, the following properties are equivalent:
\begin{enumerate}
\item[(a)] $\varphi$ is a substitution map.
\item[(b)] For each $\mu,\nu \in\nabla$ and for each $e\in \Delta$ with $|\mu+\nu|\leq |e|$, the following equality holds:
$$ K_{e,\mu+\nu} =
 \sum_{\substack{\scriptscriptstyle \beta+\gamma=e\\ \scriptscriptstyle |\mu|\leq |\beta|, |\nu|\leq |\gamma|  }} K_{\beta,\mu} K_{\gamma,\nu}. 
$$
Moreover, if the above equality holds, then $K_{e,0}=0$ whenever $|e|>0$ and
$\varphi$ is the substitution map determined by
$$  \varphi(u) = \sum_{\substack{\scriptscriptstyle e \in \Delta \\ \scriptscriptstyle 0 <|e| }} K_{e,\bfs^u} \bft^e,\quad u\in \bfs.
$$ 
\end{enumerate}
\end{proposition}

\begin{proof}
(a) $\Rightarrow$ (b) If $\varphi$ is a substitution map, there is a family
$$c^s = 
\sum_{\scriptscriptstyle \beta\in\Delta} c^s_\beta \bft^\beta \in   A[[\bft]]_\Delta,\ \  s\in \bfs,
$$
such that $\varphi(s) = c^s$. So, from (\ref{eq:exp-substi}), we deduce
$$ K_{e,\alpha}= {\bf C}_e(\varphi,\alpha)=  \sum_{\scriptstyle \mathcal{f}^{\bullet\bullet} \in \Par(e,\alpha)}
C_{\mathcal{f}^{\bullet\bullet}}
 \quad \text{for\ }\ |\alpha| \leq |e|,
$$
with
$ \displaystyle C_{\mathcal{f}^{\bullet\bullet}} = \prod_{\scriptstyle s\in \supp \alpha} \prod_{\scriptstyle r=1}^{\scriptstyle \alpha_s}  c^s_{\mathcal{f}^{sr}}.
$
\smallskip

For each ordered pair $(r,s)$ of non-negative integers there are natural injective maps
$$ i \in [r] \mapsto i \in [r+s],\quad i \in [s] \mapsto r+i \in [r+s]
$$
inducing a natural bijection $[r]\sqcup [s] \longleftrightarrow [r+s]$. Consequently, for $(\mu,\nu) \in\N^{(\bfs)}\times \N^{(\bfs)}$ there are natural injective maps
$ [\mu]  \hookrightarrow [\mu+\nu]  \hookleftarrow [\nu] 
$
inducing a natural bijection
$[\mu]\sqcup [\nu] \longleftrightarrow [\mu+\nu]$.
 So, for each $e\in \N^{(\bft)}$ and each $\mathcal{f}^{\bullet\bullet} \in \Par(e,\mu + \nu)$, we can consider the restrictions $\mathcal{g}^{\bullet\bullet} = \mathcal{f}^{\bullet\bullet}|_{[\mu]} \in \Par(\beta,\mu)$, $\mathcal{h}^{\bullet\bullet} = \mathcal{f}^{\bullet\bullet}|_{[\nu]} \in \Par(\gamma,\nu)$, with $\beta = | \mathcal{g}^{\bullet\bullet} |$ and $\gamma = | \mathcal{h}^{\bullet\bullet} |$, $\beta+\gamma=e$. The correspondence $ \mathcal{f}^{\bullet\bullet } \longmapsto (\beta,\gamma,\mathcal{g}^{\bullet\bullet}, \mathcal{h}^{\bullet\bullet})$
 establishes a bijection between $\Par(e,\mu + \nu) $ and the set of 
 $(\beta,\gamma,\mathcal{g}^{\bullet\bullet}, \mathcal{h}^{\bullet\bullet})$ with $\beta,\gamma\in \N^{(\bft)}$, $\mathcal{g}^{\bullet\bullet} \in \Par(\beta,\mu)$, 
$\mathcal{h}^{\bullet\bullet} \in \Par(\gamma,\nu)$ and $|\beta|\geq |\mu|, |\gamma|\geq |\nu|,\beta + \gamma =e$.
\inhibe
$$ \{ (\beta,\gamma,\mathcal{g}^{\bullet\bullet}, \mathcal{h}^{\bullet\bullet})\ |\ \beta,\gamma\in \N^{(\bft)}, \mathcal{g}^{\bullet\bullet} \in \Par(\beta,\mu), 
\mathcal{h}^{\bullet\bullet} \in \Par(\gamma,\nu), |\beta|\geq |\mu|, |\gamma|\geq |\nu|,\beta + \gamma =e \},
$$
\endinhibe
Moreover, under this bijection we have
$C_{\mathcal{f}^{\bullet\bullet}} = C_{\mathcal{g}^{\bullet\bullet}} C_{\mathcal{h}^{\bullet\bullet}}$ and we deduce
\begin{eqnarray*}
& 
\displaystyle K_{e,\mu+\nu}={\bf C}_e(\varphi,\mu + \nu)= \sum_{\scriptstyle \mathcal{f}^{\bullet\bullet} } C_{\mathcal{f}^{\bullet\bullet}} = \sum_{\substack{\scriptscriptstyle \beta+\gamma=e\\ \scriptscriptstyle |\mu|\leq |\beta|\\ \scriptscriptstyle  |\nu|\leq |\gamma|  }} \sum_{\substack{\scriptscriptstyle \mathcal{g}^{\bullet\bullet}, \mathcal{h}^{\bullet\bullet} }} C_{\mathcal{g}^{\bullet\bullet}} C_{\mathcal{h}^{\bullet\bullet}} =&\\
&
\displaystyle 
 \sum_{\substack{\scriptscriptstyle \beta+\gamma=e\\ \scriptscriptstyle |\mu|\leq |\beta|\\ \scriptscriptstyle  |\nu|\leq |\gamma|  }} \left( \sum_{\scriptscriptstyle \mathcal{g}^{\bullet\bullet} } C_{\mathcal{g}^{\bullet\bullet}} \right)  \left( \sum_{\scriptscriptstyle \mathcal{h}^{\bullet\bullet}} C_{\mathcal{h}^{\bullet\bullet}} \right) = \sum_{\substack{\scriptscriptstyle \beta+\gamma=e\\ \scriptscriptstyle |\mu|\leq |\beta|\\ \scriptscriptstyle  |\nu|\leq |\gamma|  }} {\bf C}_\beta(\varphi,\mu )
 {\bf C}_\gamma(\varphi,\nu ) = \sum_{\substack{\scriptscriptstyle \beta+\gamma=e\\ \scriptscriptstyle |\mu|\leq |\beta|\\ \scriptscriptstyle  |\nu|\leq |\gamma|  }}
 K_{\beta,\mu} K_{\gamma,\nu}.
\end{eqnarray*}
where $\mathcal{f}^{\bullet\bullet}\in \Par(e,\mu + \nu)$, $\mathcal{g}^{\bullet\bullet} \in \Par(\beta,\mu)$ and $\mathcal{h}^{\bullet\bullet} \in \Par(\gamma,\nu)$.
\medskip

\noindent
(b) $\Rightarrow$ (a) First, one easily proves by induction on $|e|$ that $K_{e,0}=0$ whenever $|e|>0$, and so $\varphi(1) = \varphi(\bfs^0) = K_{0,0}=1$.
Let 
$ a = \sum_\alpha a_\alpha \bfs^\alpha, b = \sum_\alpha b_\alpha \bfs^\alpha $
be elements in $A[[t]]_\Delta$, and $c=ab=\sum_\alpha c_\alpha \bfs^\alpha$ with 
$c_\alpha= \sum_{\mu+\nu=\alpha} a_{\mu} b_{\nu}$. We have:
\begin{eqnarray*}
&
\displaystyle \varphi (ab) = \varphi (c) = \sum_{\scriptscriptstyle e\in\Delta} \left( \sum_{\substack{\scriptscriptstyle \alpha\in\nabla  \\ \scriptscriptstyle |\alpha|\leq |e|}}
K_{e,\alpha} c_\alpha\right) \bft^e = \sum_{\scriptscriptstyle e\in\Delta} \left( \sum_{\substack{ \scriptscriptstyle \mu,\nu\in\nabla\\  \scriptscriptstyle |\mu+\nu|\leq |e|}}
K_{e,\mu+\nu} a_{\mu} b_{\nu}\right) \bft^e =&\\
&
\displaystyle 
\sum_e \left( \sum_{\scriptscriptstyle |\mu+\nu|\leq |e|}
\sum_{\substack{\scriptscriptstyle \beta+\gamma=e\\ \scriptscriptstyle |\mu|\leq |\beta|, |\nu|\leq |\gamma|  }} K_{\beta,\mu} K_{\gamma,\nu}
a_{\mu} b_{\nu}\right) \bft^e = \dots = \varphi (a) \varphi (b).
\end{eqnarray*}
We conclude that $\varphi$ is a (continuous) $A$-algebra map determined by the images
$$
\varphi(u) = \varphi\left(\bfs^{\bfs^u}\right)=\sum_{\substack{\scriptscriptstyle e \in \Delta \\ \scriptscriptstyle 0 <|e| }} K_{e,\bfs^u} \bft^e,\quad u\in \bfs,
$$
(remember that $\{\bfs^u\}_{u\in\bfs}$ is the canonical basis of $\N^{(\bfs)}$) and so it is a substitution map.
\end{proof}

\begin{definition} \label{defi:tensor-prod-of-substi}
The {\em tensor product} \index{tensor product (of substitution maps)} of two substitution maps $\varphi:A[[\bfs]]_\nabla \to A[[\bft]]_\Delta$, $\psi:A[[\bfu]]_{\nabla'} \to A[[\bfv]]_{\Delta'}$ is the unique substitution map
$$ \varphi \otimes \psi: A[[\bfs\sqcup \bfu]]_{\nabla \times \nabla'} \longrightarrow A[[\bft\sqcup \bfv]]_{\Delta \times \Delta'}
$$
making commutative the following diagram
$$
\begin{CD}
A[[\bfs]]_\nabla @>>> A[[\bfs\sqcup \bfu]]_{\nabla \times \nabla'} @<<< A[[\bfu]]_{\nabla'} \\
@VV{\varphi}V @VV{\varphi \otimes \psi}V @VV{\psi}V \\
A[[\bft]]_\Delta @>>> A[[\bft\sqcup \bfv]]_{\Delta \times \Delta'} @<<< A[[\bfv]]_{\Delta'},
\end{CD}
$$
where the horizontal arrows are the combinatorial substitution maps induced by the inclusions $\bfs, \bfu \hookrightarrow \bfs\sqcup \bfu$, $\bft, \bfv \hookrightarrow \bft\sqcup \bfv$\footnote{Let us notice that there are canonical continuous isomorphisms of $A$-algebras $A[[\bfs\sqcup \bfu]]_{\nabla \times \nabla'} \simeq A[[\bfs]]_\nabla \widehat{\otimes}_A A[[\bfu]]_{\nabla'}$,  
$A[[\bfs\sqcup \bfu]]_{\Delta \times \Delta'} \simeq A[[\bfs]]_\Delta \widehat{\otimes}_A A[[\bfu]]_{\Delta'}$.}.
\end{definition}

For all $(\alpha,\beta)\in \nabla \times \nabla' \subset  \N^{(\bfs)} \times \N^{(\bfu)} \equiv \N^{(\bfs\sqcup \bfu)}$ we have
$$ (\varphi \otimes \psi)(\bfs^\alpha \bfu^\beta) =  \varphi (\bfs^\alpha) \psi(\bfu^\beta) =
\cdots = 
\sum_{\substack{\scriptscriptstyle  e\in \Delta, f\in \Delta'\\ \scriptscriptstyle  |e|\geq |\alpha| \\ \scriptscriptstyle |f|\geq |\beta|  }} 
{\bf C}_e(\varphi,\alpha) {\bf C}_f(\psi,\beta) \bft^e \bfv^f
$$
and so, for all $(e,f)\in \Delta \times \Delta'$ and all $(\alpha,\beta) \in \nabla \times \nabla'$
 with $ |e|+|f|= |(e,f)| \geq |(\alpha,\beta)| = |\alpha|+|\beta|$ we have
$$
{\bf C}_{(e,f)}(\varphi\otimes \psi,(\alpha,\beta)) = \left\{ 
\begin{array}{ll}
{\bf C}_e(\varphi,\alpha) {\bf C}_f(\psi,\beta) & \text{if\ }\ |\alpha|\leq |e|\ \text{and\ }\ |\beta|\leq |f|, \\
0 & \text{otherwise}.
\end{array}  \right.
$$


\section{The action of substitution maps} \label{sec:action-substi}

In this section $k$ will be a commutative ring, $A$ a commutative $k$-algebra,
$M$ an $(A;A)$-bimodule, $\bfs$ and $\bft$ sets and $\nabla\subset\N^{(\bfs)}$, $\Delta\subset \N^{(\bft)}$ non-empty co-ideals.
\medskip

Any $A$-linear continuous map $\varphi: A[[\bfs]]_\nabla \to A[[\bft]]_\Delta$ satisfying the assumptions in \ref{nume:cont-maps-power-series-rings} induces 
$(A;A)$-linear maps  
$$\varphi_M := \varphi \widehat{\otimes} \Id_M : M[[\bfs]]_\nabla \equiv A[[\bfs]]_\Delta \widehat{\otimes}_A M \longrightarrow M[[\bft]]_\Delta \equiv A[[\bft]]_\Delta \widehat{\otimes}_A M
$$ 
and 
$$\sideset{_M}{}\opvarphi := \Id_M \widehat{\otimes} \varphi  : M[[\bfs]]_\nabla \equiv M \widehat{\otimes}_A A[[\bfs]]_\nabla \longrightarrow M[[\bft]]_\Delta \equiv M \widehat{\otimes}_A A[[\bft]]_\Delta.
$$
If $\varphi$ is determined by the family
$ K =\{ K_{e,\alpha}, e\in \nabla, \alpha \in\Delta, |\alpha|\leq |e| \} \subset A$, with
$\displaystyle
\varphi( \bfs^\alpha ) = 
\sum_{\substack{\scriptscriptstyle e\in \Delta\\ \scriptscriptstyle |e|\geq |\alpha| }}    K_{e,\alpha} \bft^e$, then
\begin{eqnarray*}
&\displaystyle \varphi_M \left(\sum_{ \scriptscriptstyle \alpha\in\nabla} m_\alpha \bfs^\alpha \right) = \sum_{\scriptscriptstyle \alpha\in\nabla} \varphi( \bfs^\alpha ) m_\alpha = \sum_{ \scriptscriptstyle e\in \Delta}   \left( \sum_{\substack{\scriptscriptstyle \alpha\in\nabla\\ \scriptscriptstyle  |\alpha|\leq |e| } }
K_{e,\alpha} m_\alpha \right) \bft^e,\quad m\in M[[\bfs]]_\nabla,
&\\
&\displaystyle 
\sideset{_M}{}\opvarphi \left(\sum_{ \scriptscriptstyle \alpha\in\nabla} m_\alpha \bfs^\alpha \right) = \sum_{\scriptscriptstyle \alpha\in\nabla}  m_\alpha \varphi( \bfs^\alpha ) = \sum_{ \scriptscriptstyle e\in\Delta }   \left( \sum_{ \substack{\scriptscriptstyle \alpha\in\nabla\\ \scriptscriptstyle  |\alpha|\leq |e| } } 
m_\alpha K_{e,\alpha}  \right) \bft^e,\quad m\in M[[\bfs]]_\nabla.&
\end{eqnarray*}
If $\varphi': A[[\bft]]_\Delta \to A[[\bfu]]_\Omega$ is another $A$-linear continuous map satisfying the assumptions in \ref{nume:cont-maps-power-series-rings} and $\opvarphi'' = \opvarphi\pcirc \opvarphi'$, we have
$\opvarphi''_M = \varphi_M \pcirc \varphi'_M$, $\sideset{_M}{}\opvarphi'' = \sideset{_M}{}\opvarphi \pcirc \sideset{_M}{}\opvarphi'$. 
\bigskip

If $\varphi: A[[\bfs]]_\nabla \to A[[\bft]]_\Delta$ is a substitution map and $m\in M[[\bfs]]_\nabla$,  $a \in A[[\bfs]]_\nabla$, we have
$$ \varphi_M (a m) = \varphi(a) \varphi_M(m),\ \sideset{_M}{}\opvarphi (m a) =  \sideset{_M}{}\opvarphi (m) \varphi(a),
$$
i.e. $\varphi_M$ is $(\varphi;A)$-linear and $\sideset{_M}{}\opvarphi$ is $(A;\varphi)$-linear. Moreover, $\varphi_M$ and $\sideset{_M}{}\opvarphi$ are compatible with the augmentations, i.e.
\begin{equation} \label{eq:compat-augment}
\varphi_M(m) \equiv m_0,\sideset{_M}{}\opvarphi(m) \equiv m_0
\!\!\!\!\mod \fn^M_0(\bft)/\Delta_M,\quad m\in M[[\bfs]]_\nabla.
\end{equation}
If $\varphi$ is the trivial substitution map (i.e. $\varphi(s)=0$ for all $s\in\bfs$), then $\varphi_M : M[[\bfs]]_\nabla \to M[[\bft]]_\Delta$ and $\sideset{_M}{}\opvarphi: M[[\bfs]]_\nabla \to M[[\bft]]_\Delta$ are also trivial, i.e.
$$ \varphi_M (m) = \sideset{_M}{}\opvarphi
(m) = m_0,\ m\in M[[\bfs]]_\nabla.
$$

\numero \label{nume:def-sbullet}
The above constructions apply in particular to the case of any $k$-algebra $R$ over $A$, for which we have two induced continuous maps,
$\varphi_R= \varphi \widehat{\otimes} \Id_R : R[[\bfs]]_\nabla \to R[[\bft]]_\Delta$, which is $(A;R)$-linear, and  $\sideset{_R}{}\opvarphi= \Id_R \widehat{\otimes} \varphi  : R[[\bfs]]_\nabla \to R[[\bft]]_\Delta$, which is $(R;A)$-linear. 
\medskip

For $r\in R[[\bfs]]_\nabla$ we will denote
$$ \varphi \sbullet r := \varphi_R(r),\quad r \sbullet \varphi := \sideset{_R}{}\opvarphi(r).
$$
Explicitely, if $r=\sum_\alpha
r_\alpha \bfs^\alpha$ with $\alpha\in\nabla$, then
\begin{equation} \label{eq:explicit-bullet}
\varphi \sbullet r
 = 
\sum_{\scriptscriptstyle e \in\Delta}
\left( \sum_{\substack{\scriptscriptstyle \alpha\in\nabla\\ \scriptscriptstyle  |\alpha|\leq |e|}} {\bf C}_e(\varphi,\alpha) r_\alpha \right) \bft^e,\quad
 r \sbullet \varphi 
 = 
\sum_{\scriptscriptstyle e \in\Delta}
\left( \sum_{\substack{\scriptscriptstyle \alpha\in\nabla\\ \scriptscriptstyle  |\alpha|\leq |e|}} r_\alpha {\bf C}_e(\varphi,\alpha)  \right) \bft^e.
\end{equation}
From (\ref{eq:compat-augment}), we deduce that 
$\varphi_R(\U^\bfs(R;\nabla)) \subset \U^\bft(R;\Delta)$ and $\sideset{_R}{}\opvarphi(\U^\bfs(R;\nabla)) \subset \U^\bft(R;\Delta)$. 
We also have $\varphi \sbullet 1 = 1 \sbullet \varphi = 1$. 
\medskip

\noindent
If $\varphi$ is a substitution map with \underline{constant coefficients}, then $\varphi_R = \sideset{_R}{}\opvarphi$ is a ring homomorphism over $\varphi$. 
In particular, $\varphi \sbullet r = r \sbullet \varphi$ and $\varphi \sbullet (rr') = (\varphi \sbullet r) (\varphi \sbullet r')$.
\medskip

\noindent 
If $\varphi = \mathbf{0}: A[[\bfs]]_\nabla \to A[[\bft]]_\Delta$ is the trivial substitution map, then 
$\mathbf{0} \sbullet r = r\sbullet \mathbf{0} = r_0$ for all $r\in R[[\bfs]]_\nabla$. In particular,
$\mathbf{0} \sbullet r = r \sbullet \mathbf{0} = 1$ for all $r\in \U^\bfs(R;\nabla)$.
\medskip

\noindent 
If $\psi:R[[\bft]]_\Delta \to R[[\bfu]]_\Omega$ is another substitution map, one has
$$ \psi \sbullet (\varphi \sbullet r) = (\psi \pcirc \varphi) \sbullet r,\quad
(r \sbullet \varphi) \sbullet \psi  =  r \sbullet (\psi \pcirc \varphi).
$$
Since $\left(R[[\bfs]]_\nabla\right)^{\text{opp}} = R^{\text{opp}}[[\bfs]]_\nabla$, for any substitution map $\varphi: A[[\bfs]]_\nabla \to A[[\bft]]_\Delta$  we have $\left( \varphi_R \right)^{\text{opp}} = \sideset{_{R^{\text{opp}}}}{}\opvarphi$ and 
$\left( \sideset{_R}{}\opvarphi \right)^{\text{opp}} = \varphi_{R^{\text{opp}}}$. 
\medskip

The proof of the following lemma is straighforward and it is left to the reader.

\begin{lemma} \label{lemma:phi-linearity-of-phi_R}
If $\varphi: A[[\bfs]]_\nabla \to A[[\bft]]_\Delta$ is a substitution map, then:
\begin{enumerate}
\item[(i)] $\varphi_R$ is left $\varphi$-linear, i.e. $\varphi_R(ar) = \varphi(a) \varphi_R(r)$ for all $a\in A[[\bfs]]_\nabla$ and for all $r\in R[[\bfs]]_\nabla$.
\item[(ii)] $\sideset{_R}{}\opvarphi$ is right $\varphi$-linear, i.e. $\sideset{_R}{}\opvarphi(ra)= \sideset{_R}{}\opvarphi(r) \varphi(a)$  for all $a\in A[[\bfs]]_\nabla$ and for all $r\in R[[\bfs]]_\nabla$.
\end{enumerate}
\end{lemma}

Let us assume again that $\varphi: A[[\bfs]]_\nabla \to A[[\bft]]_\Delta$ is an $A$-linear continuous map satisfying the assumptions in \ref{nume:cont-maps-power-series-rings}. 
We define the $(A;A)$-linear map
$$ \varphi_*: f\in \Hom_k(A,A[[\bfs]]_\nabla) \longmapsto \varphi_*(f)= \varphi \pcirc f \in \Hom_k(A,A[[\bft]]_\Delta)
$$
which induces another one $\overline{\varphi_*}: \End_{k[[\bfs]]_\nabla}^{\text{\rm top}}(A[[\bfs]]_\nabla) \longrightarrow  \End_{k[[\bft]]_\Delta}^{\text{\rm top}}(A[[\bft]]_\Delta)$ defined by
$$  \overline{\varphi_*}(f):=
\left(\varphi_*\left(f|_A\right)\right)^e =
\left(\varphi \pcirc f|_A\right)^e,\quad f\in \End_{k[[\bfs]]_\nabla}^{\text{\rm top}}(A[[\bfs]]_\nabla) .
$$
More generally, for a given left $A$-module $E$ (which will be considered as a trivial $(A;A)$-bimodule) we have $(A;A)$-linear maps
\begin{eqnarray*}
&\displaystyle
 (\varphi_{E})_*: f\in \Hom_k(E,E[[\bfs]]_\nabla) \mapsto (\varphi_{E})_*(f)= \varphi_E \pcirc f \in \Hom_k(E,E[[\bft]]_\Delta),
&\\
&\displaystyle
\overline{(\varphi_E)_*}: \End_{k[[\bfs]]_\nabla}^{\text{\rm top}}(E[[\bfs]]_\nabla) \rightarrow  \End_{k[[\bft]]_\Delta}^{\text{\rm top}}(E[[\bft]]_\Delta),\quad \overline{(\varphi_E)_*}(f):=\left(\varphi_E \pcirc f|_A\right)^e.
\end{eqnarray*}
Let us denote $R=\End_k(E)$. For each $r\in R[[\bfs]]_\nabla$ and for each $e\in E$ we have
$$ \widetilde{\varphi_R(r)}(e) = \varphi_E\left(\widetilde{r}(e) \right),
$$
or more graphically, the following diagram is commutative (see (\ref{eq:comple_formal_1})):
\begin{equation*} 
\begin{CD}
R[[\bfs]]_\nabla @>{\sim}>{r\mapsto \widetilde{r}}> \End_{k[[\bfs]]_\nabla}^{\text{\rm top}}(E[[\bfs]]_\nabla) @>{\sim}>{\text{rest.}}>
\Hom_k(E,E[[\bfs]]_\nabla) \\
 @V{\varphi_R}VV @VV{\overline{(\varphi_E)_*}}V @V{(\varphi_{E})_*}VV\\
R[[\bft]]_\Delta @>{\sim}>{r\mapsto \widetilde{r}}> \End_{k[[\bft]]_\Delta}^{\text{\rm top}}(E[[\bft]]_\Delta) @>{\sim}>{\text{rest.}}>
\Hom_k(E,E[[\bft]]_\Delta).
\end{CD}
\end{equation*}
In order to simplify notations, we will also write
$$ \varphi \sbullet f := \overline{(\varphi_E)_*}(f)\quad \forall f \in \End_{k[[\bfs]]_\nabla}^{\text{\rm top}}(E[[\bfs]]_\nabla)
$$
and so have
$ \widetilde{\varphi \sbullet r} = \varphi \sbullet \widetilde{r}$ for all $r\in R[[\bfs]]_\nabla$. 
Let us notice that $(\varphi \sbullet f)(e) = (\varphi_E \pcirc f)(e)$ for all $e\in E$, i.e.  
\begin{equation} \label{eq:atencion-bullet}
\framebox{$(\varphi \sbullet f)|_E = (\varphi_E \pcirc f)|_E$, but in general $\varphi \sbullet f \neq  \varphi_E \pcirc f$.}
\end{equation}
\medskip

If $\varphi$ is the trivial substitution map, then $(\varphi_E)_*$ (resp. $\overline{(\varphi_E)_*}$) is also trivial in the sense that if $f=\sum_\alpha f_\alpha \bfs^\alpha \in \Hom_k(E,E[[\bfs]]_\nabla)$ (resp. $f=\sum_\alpha f_\alpha \bfs^\alpha \in\End_k(E)[[\bfs]]_\nabla \equiv \End^{\text{\rm top}}_{k[[\bfs]]_\nabla}(E[[\bfs]]_\nabla)$), then $(\varphi_E)_*(f)= f_0 \in \End_k(E) \subset \Hom_k(E,E[[\bfs]]_\nabla)$ (resp. $\overline{(\varphi_E)_*}(f)= f_0^e \in \End^{\text{\rm top}}_{k[[\bfs]]_\nabla}(E[[\bfs]]_\nabla)$, with $f_0^e(\sum_\alpha e_\alpha \bfs^\alpha) = \sum_\alpha f_0(e_\alpha) \bfs^\alpha$).
\bigskip

If $\varphi: A[[\bfs]]_\nabla \to A[[\bft]]_\nabla$ is a substitution map, we have
\begin{eqnarray*}
&\displaystyle
(\varphi_{E})_*(af) = \varphi(a) (\varphi_{E})_*(f) \quad \forall a\in A[[\bfs]]_\nabla, \forall f\in \Hom_k(E,E[[\bfs]]_\nabla)
\end{eqnarray*}
and so
\begin{eqnarray*}
&\displaystyle
\overline{(\varphi_{E})_*}(af) = \varphi(a) \overline{(\varphi_{E})_*}(f) \quad \forall a\in A[[\bfs]]_\nabla, \forall f\in \End_{k[[\bfs]]_\nabla}^{\text{\rm top}}(E[[\bfs]]_\nabla).
\end{eqnarray*}
Moreover, the following inclusions hold
\begin{eqnarray*}
&(\varphi_E)_*(\Hom_k^\pcirc(E,M[[\bfs]]_\nabla)) \subset \Hom_k^\pcirc(E,E[[\bft]]_\Delta),
&\\
&\overline{(\varphi_E)_*} \left( \Aut_{k[[\bfs]]_\nabla}^\pcirc(E[[\bfs]]_\nabla) \right) \subset
\Aut_{k[[\bft]]_\Delta}^\pcirc(E[[\bft]]_\Delta),
\end{eqnarray*} 
and so we have a commutative diagram:
\begin{equation}   \label{eq:diag-funda}
\begin{CD}
 \U^\bfs(R;\nabla) @>{\sim}>{r\mapsto \widetilde{r}}> \Aut_{k[[\bfs]]_\nabla}^\pcirc(E[[\bfs]]_\nabla) @>{\sim}>{\text{rest.}}>
\Hom_k^\pcirc(E,E[[\bfs]]_\nabla) \\
 @V{\varphi_R}VV @VV{\overline{(\varphi_E)_*}}V @V{(\varphi_{E})_*}VV\\
 \U^\bft(R;\Delta) @>{\sim}>{r\mapsto \widetilde{r}}> \Aut_{k[[\bft]]_\Delta}^\pcirc(E[[\bft]]_\Delta) @>{\sim}>{\text{rest.}}>
\Hom_k^\pcirc(E,E[[\bft]]_\Delta).
\end{CD}
\end{equation}

\begin{lemma} \label{lemma:const_coeff_subst_map}
With the notations above, 
if $\varphi: k[[\bfs]]_\nabla \to k[[\bft]]_\Delta$ is a substitution map with constant coefficients, then
$$
\langle   \varphi \sbullet r,\varphi_E(e)\rangle =  \varphi_E(\langle r,e\rangle),\quad \forall r \in R[[\bfs]]_\nabla, \forall e\in E[[\bfs]]_\nabla.
$$
\end{lemma} 
\begin{proof}
Let us write $r=\sum_\alpha r_\alpha \bfs^\alpha$, $r_\alpha \in R = \End_k(E)$ and $e=\sum_\alpha e_\alpha \bfs^\alpha$, $e_\alpha \in E$. We have
\begin{eqnarray*}
& \displaystyle
\langle   \varphi \sbullet r,\varphi_E(e)\rangle = 
(\widetilde{\varphi \sbullet r})(\varphi_E(e)) =
\left( \sum_\alpha \varphi(\bfs^\alpha) \widetilde{r_\alpha} \right)
\left( \sum_\alpha \varphi(\bfs^\alpha) e_\alpha \right) =&
\\
& \displaystyle 
\sum_{\alpha,\beta}\varphi(\bfs^\alpha) \widetilde{r_\alpha} \left(\varphi(\bfs^\beta) e_\beta\right)=
\sum_{\alpha,\beta}\varphi(\bfs^\alpha)\varphi(\bfs^\beta) \widetilde{r_\alpha} \left(e_\beta\right)=
\sum_{\alpha,\beta}\varphi(\bfs^{\alpha+\beta}) \widetilde{r_\alpha} (e_\beta)= &
\\
& \displaystyle
\sum_{\gamma} \varphi(\bfs^\gamma) \left(\sum_{\alpha+\beta=\gamma} \widetilde{r_\alpha} (e_\beta)\right) =
\varphi_E \left( \sum_{\gamma}  \left(\sum_{\alpha+\beta=\gamma} \widetilde{r_\alpha} (e_\beta)\right)
\bfs^\gamma \right) 
&\\
& \displaystyle 
= \varphi_E\left(\widetilde{r}(e)\right) =\varphi_E(\langle r,e\rangle).
\end{eqnarray*}
\end{proof}

Notice that if $\varphi: k[[\bfs]]_\nabla \to k[[\bft]]_\Delta$ is a substitution map with constant coefficients, we already pointed out that $\sideset{_R}{}\opvarphi = \varphi_R$, and indeed, $\varphi \sbullet r = r \sbullet \varphi$ for all $r\in R[[\bfs]]_\nabla$.
\bigskip

\numero \label{nume:properties-external-x} 
Let us denote $\iota: A[[\bfs]]_\nabla \to A[[\bfs \sqcup \bft]]_{\nabla \times \Delta}$, $\kappa: A[[\bft]]_\Delta \to A[[\bfs \sqcup \bft]]_{\nabla \times \Delta}$ the combinatorial substitution maps given by the inclusions $\bfs \hookrightarrow \bfs \sqcup \bft$, $\bft \hookrightarrow \bfs \sqcup \bft$.

Let us notice that for $r\in R[[\bfs]]_\nabla$ and $r'\in R[[\bft]]_\Delta$, we have (see Definition \ref{defi:external-x})
$ r\btimes r' = (\iota\sbullet r) (\kappa\sbullet r') \in R[[\bfs \sqcup \bft]]_{\nabla \times \Delta}$. 

If $\nabla'\subset\nabla\subset\N^{(\bfs)}$, $\Delta'\subset\Delta\subset \N^{(\bft)}$ are non-empty co-ideals, we have 
$$\tau_{\nabla \times \Delta,\nabla' \times \Delta'}(r\btimes r') = \tau_{\nabla,\nabla'}(r) \btimes \tau_{\Delta,\Delta'}(r').
$$
If we denote by $\Sigma: R[[\bfs \sqcup \bfs]]_{\nabla\times\nabla}  \rightarrow R[[\bfs]]_\nabla$ the combinatorial substitution map given by the co-diagonal map $\bfs \sqcup \bfs \to \bfs$, it is clear that for each $r,r'\in R[[\bfs]]_\nabla$ we have
\begin{equation} \label{eq:prod-in-terms-x}
r r' = \Sigma\sbullet (r \btimes r').
\end{equation}
If $\varphi:A[[\bfs]]_\nabla \to A[[\bfu]]_\Omega$ and $\psi:A[[\bft]]_\Delta \to A[[\bfv]]_{\Omega'}$ are substitution maps, we have new substitution maps 
$\varphi \otimes \Id:  A[[\bfs \sqcup \bft]]_{\nabla\times\Delta} \to A[[\bfu \sqcup \bft]]_{\Omega\times\Delta}$ and $\Id \otimes \psi: A[[\bfs \sqcup \bft]]_{\nabla\times\Delta} \to A[[\bfs \sqcup \bfv]]_{\nabla\times\Omega'}$ (see Definition \ref{defi:tensor-prod-of-substi}) taking part in the following commutative diagrams of $(A;A)$-bimodules
$$
\begin{CD}
R[[\bfs]]_\nabla \otimes_R R[[\bft]]_\Delta @>{\varphi_R \otimes\Id }>> R[[\bfu]]_\Omega \otimes_R R[[\bft]]_\Delta\\
@V{\text{can.}}VV @VV{\text{can.}}V\\
R[[\bfs \sqcup \bft]]_{\nabla\times\Delta}  @>{\left(\varphi\otimes\Id\right)_R}>> R[[\bfu \sqcup \bft]]_{\Omega\times\Delta}
\end{CD}
$$
and
$$
\begin{CD}
R[[\bfs]]_\nabla \otimes_R R[[\bft]]_\Delta @>{\Id\otimes\psi }>> R[[\bfs]]_\nabla \otimes_R R[[\bfv]]_{\Omega'}\\
@V{\text{can.}}VV @VV{\text{can.}}V\\
R[[\bfs \sqcup \bft]]_{\nabla\times\Delta}  @>{\left(\Id\otimes\varphi\right)_R}>> R[[\bfs \sqcup \bfv]]_{\nabla\times\Omega'}.
\end{CD}
$$
So $(\varphi\sbullet r) \btimes r'=
(\varphi\otimes \Id)\sbullet (r\btimes r')$ and $r\btimes (r'\sbullet \psi)= (r\btimes r')\sbullet (\Id \otimes \psi)$.


\section{Multivariate Hasse-Schmidt derivations}
\label{section:HS}

In this section we study multivariate (possibly $\infty$-variate) Hasse--Schmidt 
deri\-vations. The original reference for 1-variate Hasse--Schmidt derivations is \cite{has37}. This notion has been studied and developed in \cite[\S 27]{mat_86} (see also \cite{vojta_HS} and \cite{nar_2012}). In \cite{hoff_kow_2016} the authors study ``finite dimensional'' Hasse--Schmidt derivations, which correspond in our terminology to $p$-variate Hasse--Schmidt derivations.
\medskip

From now on $k$ will be a commutative ring, $A$ a commutative $k$-algebra, $\bfs$ a set and $\Delta \subset \N^{(\bfs)}$ a non-empty co-ideal.

\begin{definition} 
A {\em $(\bfs,\Delta)$-variate Hasse-Schmidt derivation}\index{$(\bfs,\Delta)$-variate Hasse-Schmidt derivation}, or 
a {\em $(\bfs,\Delta)$-variate HS-derivation} for short\index{$(\bfs,\Delta)$-variate HS-derivation}, of $A$ over $k$  
 is a family $D=(D_\alpha)_{\alpha\in \Delta}$ 
 of $k$-linear maps $D_\alpha:A
\longrightarrow A$, satisfying the following Leibniz type identities: $$ D_0=\Id_A, \quad
D_\alpha(xy)=\sum_{\beta+\gamma=\alpha}D_\beta(x)D_\gamma(y) $$ for all $x,y \in A$ and for all
$\alpha\in \Delta$. We denote by
$\HS^{\bfs}_k(A;\Delta)$ the set of all $(\bfs,\Delta)$-variate HS-derivations of $A$ over
$k$ and $\HS^{\bfs}_k(A)=$ for $\Delta = \N^{(\bfs)}$. In the case where $\bfs = \{1,\dots,p\}$, a $(\bfs,\Delta)$-variate HS-derivation will be simply called a {\em $(p,\Delta)$-variate HS-derivation}  \index{$(p,\Delta)$-variate HS-derivation} and we denote $\HS^p_k(A;\Delta):= \HS^{\bfs}_k(A;\Delta)$ and $\HS^p_k(A) := \HS^{\bfs}_k(A)$. For $p=1$, a $1$-variate HS-derivation will be simply called a {\em Hasse--Schmidt derivation} \index{Hasse--Schmidt derivation}  (a HS-derivation for short\index{HS-derivation}), or a {\em higher derivation}\footnote{This terminology is used for instance in \cite{mat_86}.\index{higher derivation}}, and we will simply write $\HS_k(A;m):= \HS^1_k(A;\Delta)$ for $\Delta=\{q\in\N\ |\ q\leq m\}$\footnote{These HS-derivations are called of length $m$ in \cite{nar_2012}.} and $\HS_k(A) := \HS^1_k(A)$.
\end{definition}

\noindent \numero \label{nume:leibniz-HS} The above Leibniz identities for $D\in \HS^{\bfs}_k(A;\Delta)$ can be written as
\begin{equation} \label{eq:leibniz-HS}
D_\alpha x=\sum_{\beta+\gamma=\alpha}D_\beta(x)D_\gamma,\quad \forall x\in A, \forall \alpha\in \Delta.
\end{equation}

Any $(\bfs,\Delta)$-variate HS-derivation $D$ of $A$ over $k$  can be understood as a power series 
$$ \sum_{\scriptscriptstyle \alpha\in\Delta} D_\alpha \bfs^\alpha \in \End_k(A)[[\bfs]]_\Delta$$
and so we consider $\HS^\bfs_k(A;\Delta) \subset \End_k(A)[[\bfs]]_\Delta$.

\begin{proposition} Let $D\in \HS^{\bfs}_k(A;\Delta)$ be a HS-derivation. Then, for each $\alpha\in\Delta$, the component $D_\alpha:A \to A$ is a $k$-linear differential operator or order $\leq |\alpha|$ vanishing on $k$. In particular, if $|\alpha|=1$ then $D_\alpha:A \to A$ is a $k$-derivation.
\end{proposition}

\begin{proof} The proof follows by induction on $|\alpha|$ from \ref{eq:leibniz-HS}.
\end{proof}

The map
\begin{equation} \label{eq:HS1-Der}
D\in \HS^\bfs_k(A;\ft_1(\bfs)) \mapsto \{D_\alpha\}_{|\alpha|=1} \in \Der_k(A)^\bfs \end{equation} is clearly a bijection.
\medskip

The proof of the following proposition is straightforward and it is left to the reader (see Notation \ref{notacion:Ump} and \ref{nume:tilde}).

\begin{proposition} \label{prop:caracter_HS}
Let us denote $R=\End_k(A)$ and
let $D= \sum_\alpha D_\alpha \bfs^\alpha \in R[[\bfs]]_\Delta$ be a power series. The following properties are equivalent:
\begin{enumerate}
\item[(a)] $D$ is a $(\bfs,\Delta)$-variate HS-derivation of $A$ over $k$.
\item[(b)] The map $\widetilde{D}: A[[\bfs]]_\Delta \to A[[\bfs]]_\Delta$ is a (continuous) $k[[\bfs]]_\Delta$-algebra homomorphism compatible with the natural augmentation $A[[\bfs]]_\Delta \to A$.
\item[(c)] $D\in\U^\bfs(R;\Delta)$ and for all $a\in A[[\bfs]]_\Delta$ we have $D a  = \widetilde{D}(a) D$.
\item[(d)] $D\in\U^\bfs(R;\Delta)$ and for all  $a\in A$ we have $D a  = \widetilde{D}(a) D$.
\end{enumerate}
Moreover, in such a case $\widetilde{D}$ is a bi-continuous $k[[\bfs]]_\Delta$-algebra automorphism of $A[[\bfs]]_\Delta$.
\end{proposition}

\begin{corollary} Under the above hypotheses, $\HS^\bfs_k(A;\Delta)$ is a (multiplicative) sub-group of $\U^\bfs(R;\Delta)$.
\end{corollary}

If $\Delta' \subset \Delta \subset \N^{(\bfs)}$ are non-empty co-ideals, we obviously have group homomorphisms $\tau_{\Delta \Delta'}: \HS^\bfs_k(A;\Delta) \longrightarrow \HS^\bfs_k(A;\Delta')$.
 Since any $D\in\HS_k^\bfs(A;\Delta)$ is determined by its finite truncations, we have a natural group isomorphism
$$ \HS_k^\bfs(A) =  \lim_{\stackrel{\longleftarrow}{\substack{\scriptscriptstyle \Delta' \subset \Delta\\ \scriptscriptstyle  \sharp \Delta'<\infty}}} \HS_k^\bfs(A;\Delta').$$

In the case $\Delta' = \Delta^1 = \Delta\cap \ft_1(\bfs)$, since $\HS_k^\bfs(A;\Delta^1)\simeq \Der_k(A)^{\Delta^1}$, we can think on $\tau_{\Delta \Delta^1}$ as a group homomorphism $\tau_{\Delta\Delta^1}:\HS_k^\bfs(A;\Delta)\to \Der_k(A)^{\Delta^1}$ whose kernel is the normal subgroup of $\HS_k^\bfs(A;\Delta)$ consisting of HS-derivations $D$ with $D_\alpha = 0$ whenever $|\alpha|=1$.
\medskip

In the case $\Delta'= \Delta^n = \Delta\cap \ft_n(\bfs)$, for $n\geq 1$, we will simply write $\tau_n = \tau_{\Delta,\Delta^n}: \HS^\bfs_k(A;\Delta) \longrightarrow \HS^\bfs_k(A;\Delta^n)$.

\begin{remark} \label{nota:HS-as-Rees}
Since for any $D\in \HS^\bfs_k(A;\Delta)$ we have $D_{\alpha}\in \diff_{A/k}^{|\alpha|}(A)$, we may also think on $D$ as an element in a generalized Rees ring of the ring of differential operators:
$$ \widehat{\Rees}^\bfs\left(\DD_{A/k}(A);\Delta\right) := \left\{ \sum_{\scriptscriptstyle\alpha\in\Delta} r_\alpha \bfs^\alpha\in
\DD_{A/k}(A)[[\bfs]]_\Delta \ |\ r_\alpha \in \diff_{A/k}^{|\alpha|}(A)\right\}.
$$
\end{remark}

The group operation in $\HS^{\bfs}_k(A;\Delta)$ is explicitely given by
$$ (D,E) \in \HS^{\bfs}_k(A;\Delta) \times \HS^{\bfs}_k(A;\Delta) \longmapsto D \pcirc E \in \HS^{\bfs}_k(A;\Delta)$$
with
$$ (D \pcirc E)_\alpha = \sum_{\scriptscriptstyle \beta+\gamma=\alpha} D_\beta \pcirc E_\gamma,
$$
and the identity element of $\HS^{\bfs}_k(A;\Delta)$ is $\mathbb{I}$ with $\mathbb{I}_0 = \Id$ and 
$\mathbb{I}_\alpha = 0$ for all $\alpha \neq 0$. The inverse of a $D\in \HS^\bfs_k(A;\Delta)$ will be denoted by $D^*$. 
\medskip

\begin{proposition} Let $D\in \HS^{\bfs}_k(A;\Delta)$, $E \in \HS^{\bft}_k(A;\nabla)$  be HS-derivations. Then their external product $D\btimes E$ (see Definition \ref{defi:external-x}) is a $(\bfs\sqcup \bft,\nabla\times\Delta)$-variate HS-derivation. \index{external product (of HS-derivations)}
\end{proposition}

\begin{proof} From Lemma \ref{lemma:tilde-otimes} we know that $\widetilde{D\btimes E} = \widetilde{D}\btimes \widetilde{E}$ and we conclude by Proposition \ref{prop:caracter_HS}.
\end{proof}

\begin{definition}  \label{defi:sbullet-0}
For each $a\in A^\bfs$ and for each $D\in \HS_k^\bfs(A;\Delta)$, we define $a \sbullet D$ as
$$ (a\sbullet D)_\alpha := a^\alpha D_\alpha,\quad \forall \alpha\in \Delta.
$$
It is clear that $a \sbullet D\in \HS_k^\bfs(A;\Delta)$, $a'\sbullet (a\sbullet D) = (a'a)\sbullet D$, $1 \sbullet D=D$ and $0 \sbullet D=\mathbb{I}$.
\end{definition}

If $\Delta' \subset \Delta \subset \N^{(\bfs)}$ are non-empty co-ideals,
we have $\tau_{\Delta \Delta'}(a\sbullet D)=a\sbullet \tau_{\Delta \Delta'}(D)$. Hence, 
in the case $\Delta' = \Delta^1 = \Delta\cap \ft_1(\bfs)$, since $\HS_k^\bfs(A;\Delta^1)\simeq \Der_k(A)^{\Delta^1}$, the image of $\tau_{\Delta\Delta^1}:\HS_k^\bfs(A;\Delta)\to \Der_k(A)^{\Delta^1}$
is an $A$-submodule.
\medskip

The following lemma provides a dual way to express the Leibniz identity (\ref{eq:leibniz-HS}), \ref{nume:leibniz-HS}.

\begin{lemma} \label{lema:dual-leibniz-HS}
For each $D\in \HS^\bfs_k(A;\Delta)$ and for each $\alpha\in \Delta$, we have
$$
x D_\alpha =\sum_{\beta+\gamma=\alpha}D_\beta\, D^*_\gamma(x),\quad \forall x\in A.
$$
\end{lemma}

\begin{proof} We have
\begin{eqnarray*}
&\displaystyle \sum_{\beta+\gamma=\alpha}D_\beta\, D^*_\gamma(x) = \sum_{\beta+\gamma=\alpha} \sum_{\mu+\nu=\beta} D_\mu(D^*_\gamma(x)) D_\nu =
&\\
&\displaystyle  \sum_{e+\nu=\alpha} \left(
\sum_{\mu+\gamma=e} D_\mu(D^*_\gamma(x))\right) D_\nu = x D_\alpha.
\end{eqnarray*} 
\end{proof}

It is clear that the map
(\ref{eq:HS1-Der}) is an isomorphism of groups (with the addition on $\Der_k(A)$ as internal operation) and so $\HS_k^\bfs(A;\ft_1(\bfs))$ is abelian.
\medskip

\begin{notation} \label{notacion:pcirc-HS}
Let us denote
\begin{eqnarray*}
&\displaystyle  \Hom_{k-\text{\rm alg}}^\pcirc(A,A[[\bfs]]_\Delta) :=
&\\
&\displaystyle  \left\{ f \in \Hom_{k-\text{\rm alg}}(A,A[[\bfs]]_\Delta)\ |\ f(a) \equiv a\!\!\!\!\mod \fn^A_0(\bfs)/\Delta_A\  \forall a\in A \right\}, &\\
&\displaystyle  \Aut_{k[[\bfs]]_\Delta-\text{\rm alg}}^\pcirc(A[[\bfs]]_\Delta) :=
&\\
&\displaystyle 
 \left\{ f \in \Aut_{k[[\bfs]]_\Delta-\text{\rm alg}}^{\text{\rm top}}(A[[\bfs]]_\Delta)\ |\ f(a) \equiv a_0\!\!\!\!\mod \fn^A_0(\bfs)/\Delta_A\ \forall a\in A[[\bfs]]_\Delta \right\}.
\end{eqnarray*}
\end{notation}
It is clear that (see Notation \ref{notacion:pcirc}) $\Hom_{k-\text{\rm alg}}^\pcirc(A,A[[\bfs]]_\Delta) \subset \Hom_k^\pcirc(A,A[[\bfs]]_\Delta)$ and $\Aut_{k[[\bfs]]_\Delta-\text{\rm alg}}^\pcirc(A[[\bfs]]_\Delta) \subset \Aut_{k[[\bfs]]_\Delta}^\pcirc(A[[\bfs]]_\Delta)$  are subgroups and
we have  group isomorphisms (see (\ref{eq:Aut-iso-pcirc}) and  (\ref{eq:U-iso-pcirc})):
\begin{equation} \label{eq:HS-funda}
\begin{CD}
\HS^\bfs_k(A;\Delta) @>{D \mapsto \widetilde{D}}>{\simeq}> \Aut_{k[[\bfs]]_\Delta-\text{\rm alg}}^\pcirc(A[[\bfs]]_\Delta)
@>{\text{restriction}}>{\simeq}> \Hom_{k-\text{\rm alg}}^\pcirc(A,A[[\bfs]]_\Delta).
\end{CD}
\end{equation}
The composition of the above isomorphisms is given by
\begin{equation} \label{eq:HS-iso-Hom(A,A[[s]])}
 D\in \HS^\bfs_k(A;\Delta) \stackrel{\sim}{\longmapsto} \Phi_D := \left[a\in A \mapsto \sum_{\scriptscriptstyle \alpha\in\Delta} D_\alpha(a) \bfs^\alpha\right] \in \Hom_{k-\text{\rm alg}}^\pcirc(A,A[[\bfs]]_\Delta).
\end{equation}
For each HS-derivation $D\in \HS^\bfs_k(A;\Delta)$ we have
$$ \widetilde{D}\left( \sum_{\scriptscriptstyle \alpha\in\Delta} a_\alpha \bfs^\alpha \right) = 
\sum_{\scriptscriptstyle \alpha\in\Delta} \Phi_D(a_\alpha) \bfs^\alpha,
$$
for all $\sum_\alpha a_\alpha \bfs^\alpha \in A[[\bfs]]_\Delta$, and for any $E\in \HS^\bfs_k(A;\Delta)$ we have $\Phi_{D \smallcirc E} = \widetilde{D} \pcirc \Phi_E$. 
If $\Delta'\subset \Delta$ is another non-empty co-ideal and
 we denote by  $\pi_{\Delta \Delta'}: A[[\bfs]]_\Delta \to A[[\bfs]]_{\Delta'}$ the projection, one has $\Phi_{\tau_{\Delta \Delta'} (D)} = \pi_{\Delta \Delta'}\pcirc \Phi_D$. 

\begin{definition} \label{def:ell-new} For each HS-derivation
$E\in\HS^\bfs_k(A;\Delta)$, we denote 
$$ \elln(E) := \min \{r\geq 1 \ |\ \exists \alpha\in \Delta, |\alpha|= r, E_\alpha \neq 0 \} \geq 1$$
if $E\neq \mathbb{I}$ and $ \elln(E) = \infty$ if $E= \mathbb{I}$. In other words, $\ell (E) = \ord (E-\mathbb{I})$. 
Clearly, if $\Delta$ is bounded, then $\elln(E)>\max\{|\alpha|\ |\ \alpha\in\Delta\}  \Longleftrightarrow \elln(E)=\infty \Longleftrightarrow E = \mathbb{I}$.
\end{definition}

We obviously have $\elln(E\pcirc E') \geq \min \{\elln(E),\elln(E')\}$ and $\elln(E^*) = \elln(E)$. Moreover, if $\elln(E') > \elln(E)$, then $\elln(E\pcirc E')=\elln(E)$:
$$ \elln(E\pcirc E') = \ord (E\pcirc E' - \mathbb{I}) = \ord( E \pcirc (E' - \mathbb{I}) + (E - \mathbb{I}))
$$
and since $\ord (E \pcirc (E' - \mathbb{I})) \geq\footnote{Actually, here an equality holds since the $0$-term of $E$ (as a series) is $1$.} \ord ( E' - \mathbb{I}) = \elln(E') > \elln(E) =  \ord ( E - \mathbb{I})$ we obtain 
$$ \elln(E\pcirc E') = \cdots = \ord( E \pcirc (E' - \mathbb{I}) + (E - \mathbb{I})) = \ord (E - \mathbb{I}) = \elln(E).
$$

\begin{proposition}  \label{prop:ell-new}
For each $D\in \HS^{\bfs}_k(A;\Delta)$ we have that $D_\alpha$ is a $k$-linear differential operator or order $\leq \lfloor \frac{|\alpha|}{\elln(D)}\rfloor$
for all $\alpha\in \Delta$. In particular,  $D_\alpha$ is a $k$-derivation if $|\alpha|=\elln(D)$, whenever $\elln(D) < \infty$ (  $ \Leftrightarrow D\neq \mathbb{I}$). 
\end{proposition}

\begin{proof} We may assume $D\neq \mathbb{I}$. Let us call $n:=\elln(D) <\infty$ and, for each $\alpha\in \Delta$,  $q_\alpha := \lfloor \frac{|\alpha|}{n}\rfloor$ and $r_\alpha := |\alpha| - q_\alpha n $, 
 $0\leq r_\alpha < n$. We proceed by induction on $q_\alpha$. If $q_\alpha=0$, then $|\alpha|< n$, $D_\alpha =0$ and the result is clear. Assume that the order of $D_\beta$ is less or equal than $q_\beta$ whenever $0\leq q_\beta \leq q$. Now take $\alpha\in \Delta$ with $q_\alpha = q+1$. For any $a\in A$ we have
$$ [D_\alpha,a] = \sum_{\substack{\scriptscriptstyle \gamma+\beta=\alpha\\ \scriptscriptstyle |\gamma|>0}} D_\gamma(a) D_\beta = \sum_{\substack{\scriptscriptstyle \gamma+\beta=\alpha\\ \scriptscriptstyle |\gamma|\geq n}} D_\gamma(a) D_\beta,
$$
but any $\beta$ in the index set of the above sum must have norm $\leq |\alpha|-n$ and so $q_\beta < q_\alpha =q+1$ and $D_\beta$ has order $\leq q_\beta$. Hence $[D_\alpha,a]$ has order $\leq q$ for any $a\in A$ and $D_\alpha$ has order $\leq q+1 =q_\alpha$.
\end{proof}

The following example shows that the group structure on HS-derivations takes into account the Lie bracket on usual derivations.

\begin{example} If $D,E\in \HS^{\bfs}_k(A;\Delta)$, then we may apply the above proposition to $[D,E] = D\pcirc E\pcirc D^*\pcirc E^*$ to deduce that $[D,E]_\alpha \in \Der_k(A)$ whenever $|\alpha|=2$.
Actually, for $|\alpha|=2$ we have:
$$[D,E]_\alpha = \left\{ 
\begin{array}{lcl}
[D_{\bfs^t},E_{\bfs^t}] & \text{if} & \alpha= 2 \bfs^t \\  
 {[}D_{\bfs^t},E_{\bfs^u}{]} + [D_{\bfs^u},E_{\bfs^t}] & \text{if} & \alpha= \bfs^t + \bfs^u,\ \text{with}\ t\neq u.
\end{array}
\right.
$$
\end{example}

\begin{proposition} \label{prop:ell-corchete}
For any $D,E\in \HS^{\bfs}_k(A;\Delta)$ we have $ \elln([D,E]) \geq \elln(D) + \elln(E)$.
\end{proposition}

\begin{proof} We may assume $D,E\neq \mathbb{I}$. Let us write $m=\elln(D)=\ell(D^*)$,  $n=\elln(E)=\elln(E^*)$.  We have $D_\beta=D^*_\beta = 0$ whenever $0<|\beta| < m$
and
$E_\gamma=E^*_\gamma = 0$ whenever $0<|\gamma| < n$.

Let $\alpha\in\Delta$ be with $0<|\alpha|<m+n$. If $|\alpha|< m$ or $|\alpha|<n$ it is clear that $[D,E]_\alpha=0$. Assume that $m,n\leq |\alpha| < m+n$:
\begin{eqnarray*}
& \displaystyle
 [D,E]_\alpha = \sum_{\scriptscriptstyle \beta + \gamma + \lambda + \mu =\alpha} D_\beta \pcirc E_\gamma \, D^*_\lambda \, E^*_\mu =
\sum_{\scriptscriptstyle \gamma  + \mu =\alpha}  \ E_\gamma  \, E^*_\mu +
&\\
& \displaystyle
 \sum_{\substack{\scriptscriptstyle \beta + \gamma + \lambda + \mu =\alpha\\ \scriptscriptstyle |\beta+\lambda|> 0}} D_\beta \, E_\gamma \, D^*_\lambda \, E^*_\mu = 
0 +
\sum_{\substack{\scriptscriptstyle \gamma + \lambda + \mu =\alpha\\ \scriptscriptstyle |\lambda|>0}}   E_\gamma \, D^*_\lambda \, E^*_\mu +
\sum_{\substack{\scriptscriptstyle \beta +\gamma  + \mu =\alpha\\ \scriptscriptstyle |\beta|>0}}  D_\beta\, E_\gamma  \, E^*_\mu +
&\\
&\displaystyle
\sum_{\substack{\scriptscriptstyle \beta + \gamma + \lambda + \mu =\alpha\\ \scriptscriptstyle
|\beta|, |\lambda| >0 }} D_\beta \, E_\gamma \, D^*_\lambda \, E^*_\mu =
\sum_{\substack{\scriptscriptstyle \gamma + \lambda + \mu =\alpha\\ \scriptscriptstyle |\lambda|\geq m}}   E_\gamma \, D^*_\lambda \, E^*_\mu +
\sum_{\substack{\scriptscriptstyle \beta +\gamma  +
 \mu =\alpha\\ \scriptscriptstyle |\beta|\geq m}}  D_\beta\, E_\gamma  \, E^*_\mu +
&\\
&\displaystyle
\sum_{\substack{\scriptscriptstyle \beta + \gamma + \lambda + \mu =\alpha\\ \scriptscriptstyle
|\beta|, |\lambda| \geq m }} D_\beta \, E_\gamma \, D^*_\lambda \, E^*_\mu =
D^*_\alpha+
\sum_{\substack{\scriptscriptstyle \gamma + \lambda + \mu =\alpha\\ \scriptscriptstyle |\lambda|\geq m, |\gamma+\mu|>0}}   E_\gamma \, D^*_\lambda \, E^*_\mu +
D_\alpha+
&\\
&\displaystyle
\sum_{\substack{\scriptscriptstyle \beta   +
 \mu =\alpha\\ \scriptscriptstyle |\beta|\geq m\\ \scriptscriptstyle  |\gamma  +
 \mu|>0 }}  D_\beta\, E_\gamma  \, E^*_\mu +
\sum_{\substack{\scriptscriptstyle \beta  + \lambda  =\alpha\\ \scriptscriptstyle
|\beta|, |\lambda| \geq m }} D_\beta \, D^*_\lambda  +
\sum_{\substack{\scriptscriptstyle \beta + \gamma + \lambda + \mu =\alpha\\ \scriptscriptstyle
|\beta|, |\lambda| \geq m\\ \scriptscriptstyle |\gamma + \mu|>0  }} D_\beta \, E_\gamma \, D^*_\lambda \, E^*_\mu =
&\\
&\displaystyle
D^*_\alpha+ 0 +
D_\alpha + 0 +
\sum_{\substack{\scriptscriptstyle \beta  + \lambda  =\alpha\\ \scriptscriptstyle
|\beta|, |\lambda| > 0 }} D_\beta \, D^*_\lambda  + 0 = \sum_{\scriptscriptstyle \beta  + \lambda  =\alpha} D_\beta \, D^*_\lambda = 0.
\end{eqnarray*}
So, $ \elln([D,E]) \geq \elln(D) + \elln(E)$.
\end{proof}

\begin{corollary} Assume that $\Delta$ is bounded and let $m$ be the $\max$ of $|\alpha|$ with $\alpha\in\Delta$. Then, the group $\HS^{\bfs}_k(A;\Delta)$ is nilpotent of nilpotent class $\leq m$, where a central series is\footnote{Let us notice that 
$\{E\in\HS^\bfs_k(A;\Delta)\ |\ \elln(E) > r\} = \ker \tau_{\Delta,\Delta_r}$.}
$$ \{\mathbb{I}\} = \{E|\ \elln(E)  >m \} \vartriangleleft \{E|\ \elln(E) \geq m\} \vartriangleleft \cdots \vartriangleleft \{E|\ \elln(E) \geq 1\} = \HS^{\bfs}_k(A;\Delta).
$$
\end{corollary}

\begin{proposition} \label{prop:expresion-inverse-multi-HS}
For each $D\in \HS_k^\bfs(A;\Delta)$, its inverse $D^*$ is given by $D^*_0=\Id$ and
$$ 
D^*_\alpha = \sum_{d=1}^{|\alpha|} (-1)^d \sum_{\alpha^\bullet \in \Par(\alpha,d)}
D_{\alpha^1}\pcirc \cdots \pcirc D_{\alpha^d},\quad \alpha\in\Delta.$$
Moreover, $\sigma_{|\alpha|}(D^*_\alpha) = (-1)^{|\alpha|}
\sigma_{|\alpha|}(D_\alpha)$. 
\end{proposition}

\begin{proof} The first assertion is a straightforward consequence of Lemma \ref{lemma:unit-psr}.
For the second assertion, first we have $D^*_\alpha = - D_\alpha$ for all $\alpha$ with $|\alpha|=1$, and if we denote by $\mathbf{-1} \in A^\bfs$ the constant family $-1$ and $E= D \pcirc ((\mathbf{-1}) \sbullet D)$, we have $\elln(E)> 1$. So, $D^* = ((\mathbf{-1}) \sbullet D) \pcirc E^*$ and $$D^*_\alpha = \sum_{\scriptscriptstyle \beta+\gamma=\alpha} (-1)^{|\beta|} D_\beta E^*_\gamma = (-1)^{|\alpha|} D_\alpha +
\sum_{\substack{\scriptscriptstyle \beta+\gamma=\alpha\\ \scriptscriptstyle |\gamma|>0}} (-1)^{|\beta|} D_\beta E^*_\gamma.
$$
From Proposition \ref{prop:ell-new}, we know that $E^*_\gamma$ is a differential operator of order strictly less than $|\gamma|$ and so 
$\sigma_{|\alpha|}(D^*_\alpha) = (-1)^{|\alpha|}
\sigma_{|\alpha|}(D_\alpha)$.
\end{proof}

\section{The action of substitution maps on HS-derivations}

In this section, $k$ will be a commutative ring, $A$ a commutative $k$-algebra, $R=\End_k(A)$, $\bfs$, $\bft$ sets and $\Delta \subset \N^{(\bfs)}$, $\nabla \subset \N^{(\bft)}$ non-empty co-ideals.
\medskip

We are going to extend the operation $(a,D) \in A^\bfs \times \HS^{\bfs}_k(A;\Delta)  \mapsto a\sbullet D \in \HS^{\bfs}_k(A;\Delta)$ (see Definition \ref{defi:sbullet-0}) by means of the constructions in section 
\ref{sec:action-substi}. 

\begin{proposition} \label{prop:equiv-action-subs-HS}
For any substitution map $\varphi:A[[\bfs]]_\Delta \to A[[\bft]]_\nabla$, we have:
\begin{enumerate}
\item[(1)] $\varphi_*\left(\Hom_{k-\text{\rm alg}}^\pcirc(A,A[[\bfs]]_\Delta)\right) \subset \Hom_{k-\text{\rm alg}}^\pcirc(A,A[[\bft]]_\nabla)$,
\item[(2)] $\varphi_R \left( \HS^\bfs_k(A;\Delta)\right) \subset \HS^\bft_k(A;\nabla)$,
\item[(3)] $ \overline{\varphi_*} \left( \Aut_{k[[\bfs]]_\Delta-\text{\rm alg}}^\pcirc(A[[\bfs]]_\Delta) \right) \subset
\Aut_{k[[\bft]]_\nabla-\text{\rm alg}}^\pcirc(A[[\bft]]_\nabla)$.
\end{enumerate}
\end{proposition}

\begin{proof} By using diagram (\ref{eq:diag-funda}) and (\ref{eq:HS-funda}), it is enough to prove 
the first inclusion, but if $f\in \Hom_{k-\text{\rm alg}}^\pcirc(A,A[[\bfs]]_\Delta)$, it is clear that $\varphi_*(f)= \varphi \pcirc f: A \to A[[\bft]]_\nabla$ is a $k$-algebra map. Moreover, since 
$\varphi(\ft^A_0(\bfs)/\Delta_A) \subset \ft^A_0(\bft)/\nabla_A$ (see \ref{nume:operaciones-con-substitutions}) and 
$f(a) \equiv a \!\mod \ft^A_0(\bfs)/\Delta_A$ for all $a\in A$, we deduce that
$\varphi(f(a)) \equiv \varphi(a) \!\mod \ft^A_0(\bft)/\nabla_A$ for all $a\in A$, but $\varphi$ is an $A$-algebra map and $\varphi(a) = a$. So $\varphi_*(f)\in\Hom_{k-\text{\rm alg}}^\pcirc(A,A[[\bft]]_\nabla)$.
\end{proof}

As a consequence of the above proposition and diagram (\ref{eq:diag-funda}) we have a commutative diagram:
\begin{equation} \label{eq:diag-funda-HS}
\begin{CD}
\Hom_{k-\text{\rm alg}}^\pcirc(A,A[[\bfs]]_\Delta)  @<{\sim}<{\Phi_D \mapsfrom D}<
\HS^\bfs_k(A;\Delta) @>{\sim}>> \Aut_{k[[\bfs]]_\Delta-\text{\rm alg}}^\pcirc(A[[\bfs]]_\Delta)\\
@V{\varphi_*}VV @V{\varphi_R}VV @VV{\overline{\varphi_*}}V\\
\Hom_{k-\text{\rm alg}}^\pcirc(A,A[[\bft]]_\nabla)  @<{\sim}<{\Phi_D \mapsfrom D}<
\HS^\bft_k(A;\nabla) @>{\sim}>> \Aut_{k[[\bft]]_\nabla-\text{\rm alg}}^\pcirc(A[[\bft]]_\nabla).
\end{CD}
\end{equation}
The inclusion 2) in Proposition  \ref{prop:equiv-action-subs-HS} can be rephrased by saying that for any substitution map $\varphi:A[[\bfs]]_\Delta \to A[[\bft]]_\nabla$ and for any HS-derivation $D\in \HS^\bfs_k(A;\Delta)$ we have $\varphi \sbullet D\in \HS^\bft_k(A;\nabla)$ (see \ref{nume:def-sbullet}). Moreover
$ \Phi_{\varphi \sbullet D} = \varphi \pcirc \Phi_D$.
\medskip

It is clear that for any co-ideals $\Delta' \subset \Delta$ and $\nabla' \subset \nabla$ with $\varphi \left( \Delta'_A/\Delta_A \right) \subset \nabla'_A/\nabla_A$
we have
\begin{equation} \label{eq:trunca_bullet}
\tau_{\nabla \nabla'}(\varphi \sbullet D) = \varphi' \sbullet \tau_{\Delta \Delta'}(D),
\end{equation}
where $\varphi': A[[\bfs]]_{\Delta'} \to A[[\bft]]_{\nabla'}$ is the substitution map induced by $\varphi$.
\medskip

Let us notice that any  $a\in A^\bfs$ gives rise to a substitution map $\varphi: A[[\bfs]]_\Delta  \to A[[\bfs]]_\Delta$ given by $\varphi(s)= a_s s$ for all $s\in\bfs$, and one has $a \sbullet D = \varphi \sbullet D$.
\bigskip

\numero
\label{nume:def-sbullet-HS} Let $\varphi\in\Sub_A(\bfs,\bft;\nabla,\Delta)$, $\psi\in\Sub_A(\bft,\bfu;\Delta,\Omega)$ be substitution maps and $D,D'\in\HS_k^\bfs(A;\nabla)$ HS-derivations. From \ref{nume:def-sbullet} we deduce the following properties:\medskip

\noindent -) If we denote $E:= \varphi \sbullet D \in\HS_k^\bft(A;\Delta)$,  we have
\begin{equation} \label{eq:expression_phi_D} 
E_0=\Id,\quad  E_e = \sum_{\substack{\scriptstyle \alpha\in\nabla\\ \scriptstyle |\alpha|\leq |e| }} {\bf C}_e(\varphi,\alpha) D_\alpha,\quad \forall e\in \Delta.
\end{equation}
-) If $\varphi$ has \underline{constant coefficients}, then $\varphi \sbullet (D\pcirc D') = (\varphi \sbullet D) \pcirc (\varphi \sbullet D')$. The general case will be treated in Proposition \ref{prop:varphi-D-main}.
\medskip

\noindent 
-) If $\varphi = \mathbf{0}$ is the trivial substitution map or if $D=\mathbb{I}$, then 
$\varphi \sbullet D = \mathbb{I}$.
\medskip

\noindent 
-) $ \psi \sbullet (\varphi \sbullet D) = (\psi \pcirc \varphi) \sbullet D$.

\begin{remark} We recall that a HS-derivation $D\in\HS_k(A)$ is called {\em iterative} \index{iterative (HS-derivation)} (see \cite[pg. 209]{mat_86}) if
$$  D_i \pcirc D_j = \binom{i+j}{i} D_{i+j}\quad \forall i,j\geq 0.
$$
This notion makes sense for $\bfs$-variate HS-derivations of any length. Actually, iterativity may be understood through the action of substitution maps. Namely, if we denote by $\iota,\iota': s \hookrightarrow s \sqcup s$ the two canonical inclusions and $\iota+\iota': A[[\bfs]] \to A[[\bfs \sqcup \bfs]]$ is the substitution map determined by
$$ (\iota+\iota') (s) = \iota(s) + \iota'(s),\quad \forall s\in \bfs,
$$
then a HS-derivation $D\in \HS^\bfs_k(A)$ is iterative if and only if 
$$ (\iota+\iota')\sbullet D = (\iota\sbullet D) \pcirc (\iota' \sbullet D).
$$
A similar remark applies for any formal group law instead of $\iota+\iota'$ (cf. \cite{hoff_kow_2015}).
\end{remark}

\begin{proposition} \label{prop:varphi-D-main}
Let $\varphi:A[[\bfs]]_\nabla  \to A[[\bft]]_\Delta$ be a substitution map. Then, 
the following assertions hold:
\begin{enumerate}
\item[(i)] For each $D\in\HS_k^\bfs(A;\nabla)$ there is a unique substitution map
$\varphi^D: A[[\bfs]]_\nabla  \to A[[\bft]]_\Delta$ such that 
$\left(\widetilde{\varphi \sbullet D}\right) \pcirc \varphi^D =  \varphi \pcirc \widetilde{D}$. 
Moreover, $\left(\varphi\sbullet D\right)^* = \varphi^D \sbullet D^*$ and $\varphi^{\mathbb{I}} = \varphi$.
\item[(ii)] For each $D,E\in\HS_k^\bfs(A;\nabla)$, we have $\varphi \sbullet (D \pcirc E) = (\varphi \sbullet D) \pcirc (\varphi^D \sbullet E)$ and 
 $\left(\varphi^D\right)^E = \varphi^{D \pcirc E}$. 
 In particular, $\left(\varphi^D\right)^{D^*} = \varphi$. 
\item[(iii)] If $\psi$ is another composable substitution map, then $(\varphi \pcirc \psi)^D = \varphi^{\psi\sbullet D} \pcirc \psi^D$.
\item[(iv)] $\tau_{n}(\varphi^D) = \tau_{n}(\varphi)^{\tau_{n}(D)}$, for all $n\geq 1$.
\item[(v)] If $\varphi$ has constant coefficients then $\varphi^D = \varphi$.
\end{enumerate}
\end{proposition} 

\begin{proof} (i) We know that 
$$\widetilde{D}\in \Aut_{k[[\bfs]]_\nabla-\text{\rm alg}}^\pcirc(A[[\bfs]]_\nabla)\quad \text{and}\quad
\widetilde{\varphi \sbullet D}\in 
\Aut_{k[[\bft]]_\Delta-\text{\rm alg}}^\pcirc(A[[\bft]]_\Delta).
$$ 
The only thing to prove is that 
$$\varphi^D := 
\left(\widetilde{\varphi \sbullet D}\right)^{-1} \pcirc \varphi \pcirc \widetilde{D}
$$ 
is a substitution map $A[[\bfs]]_\nabla  \to A[[\bft]]_\Delta$ (see Definition \ref{def:substitution_maps}). 
Let start by proving that $\varphi^D$ is an $A$-algebra map. Let us write $E=\varphi \sbullet D$. For each $a\in A$ we have
\begin{eqnarray*}
& \varphi^D (a) = \widetilde{E}^{-1} \left( \varphi \left( \widetilde{D}(a)\right)\right) = \widetilde{E}^{-1} \left( \varphi \left( \Phi_D(a)\right)\right) = 
&\\
&
\widetilde{E}^{-1} \left( (\varphi \pcirc \Phi_D)(a))\right) = 
\widetilde{E}^{-1} \left( \Phi_{\varphi \sbullet D}(a)\right) = \widetilde{E}^{-1} \left( \left(\widetilde{\varphi \sbullet D}\right)(a)\right) = a,
\end{eqnarray*}
and so $\varphi^D$ is $A$-linear. The continuity of $\varphi^D$ is clear, since it is the composition of continuous maps. For each $s\in\bfs$, let us write
$$ \varphi(s) = \sum_{\substack{\scriptscriptstyle \beta \in \Delta\\ \scriptscriptstyle |\beta| > 0}}
c^s_\beta \bft^\beta.
$$
Since $\varphi$ is a substitution map, property (\ref{eq:cond-fini-c}) holds:
$$
 \#\{s\in \bfs\ |\ c^s_\beta \neq 0\} < \infty\quad\quad \text{for all\ }\ \beta \in \Delta.
$$
We have
$$ \varphi^D (s) = \widetilde{E^*} \left( \varphi ( \widetilde{D}(s))\right) =  
\widetilde{E^*} \left( \varphi (s)\right) = \sum_{\scriptscriptstyle \beta \in \Delta} 
\left( \sum_{\scriptscriptstyle \alpha + \gamma=\beta} E^*_\alpha (c^s_\gamma) \right) \bft^\beta =
\sum_{\scriptscriptstyle \beta \in \Delta}  d^s_\beta \bft^\beta
$$
with
$ d^s_\beta = \sum_{\scriptscriptstyle \alpha + \gamma=\beta} E^*_\alpha (c^s_\gamma)$.
So, for each $\beta \in \Delta$ we have
$$ \{s\in \bfs\ |\ c^s_\beta \neq 0\} \subset \bigcup_{\gamma \leq \beta} \{s\in \bfs\ |\ c^s_\gamma \neq 0\}
$$
and $\varphi^D$ satisfies property (\ref{eq:cond-fini-c}) too. We conclude that $\varphi^D$ is a substitution map, and obviously it is the only one such that  $\left(\widetilde{\varphi \sbullet D}\right) \pcirc \varphi^D =  \varphi \pcirc \widetilde{D}$. From there, we have 
$$\varphi^D \pcirc \widetilde{D^*} = \varphi^D \pcirc \widetilde{D}^{-1} = \left(\widetilde{\varphi \sbullet D}\right)^{-1} \pcirc \varphi = \widetilde{\left(\varphi \sbullet D\right)^*} \pcirc \varphi,
$$
and taking restrictions to $A$ we obtain
$ \varphi^D \pcirc \Phi_{D^*} = \Phi_{\left(\varphi \sbullet D\right)^*}
$ and so $\varphi^D \sbullet D^* = \left(\varphi\sbullet D\right)^*$.
\medskip

On the other hand, it is clear that if $D=\mathbb{I}$, then $\varphi^\mathbb{I} = \varphi$ and if $\varphi=\mathbf{0}$, $\mathbf{0}^D=\mathbf{0}$.
\bigskip

\noindent (ii) In order to prove the first equality, we need to prove the equality
$ \widetilde{\varphi \sbullet (D \pcirc E)} = \left(\widetilde{\varphi \sbullet D}\right) \pcirc \left(\widetilde{\varphi^D \sbullet E}\right)
$. For this it is enough to prove the equality after restriction to $A$, but
$$ \left(\widetilde{\varphi \sbullet (D \pcirc E)}\right)|_A = \Phi_{\varphi \sbullet (D \pcirc E)} = \varphi \pcirc \Phi_{D \pcirc E} = \varphi \pcirc \widetilde{D} \pcirc \Phi_E,
$$
$$ \left( \left(\widetilde{\varphi \sbullet D}\right) \pcirc \left(\widetilde{\varphi^D \sbullet E}\right) \right)|_A = 
\left(\widetilde{\varphi \sbullet D}\right) \pcirc \Phi_{\varphi^D \sbullet E} =
\left(\widetilde{\varphi \sbullet D}\right) \pcirc \varphi^D \pcirc \Phi_E
$$
and both are equal by (i). For the second equality, we have $\left(\varphi^D\right)^{D^*} = \varphi^{\mathbb{I}}= \varphi$.
\bigskip

\noindent (iii) Since 
\begin{eqnarray*}
& \widetilde{\left((\varphi \pcirc \psi) \sbullet D\right)} \pcirc
\left( \varphi^{\psi\sbullet D} \pcirc \psi^D \right) = \widetilde{\left(\varphi \sbullet (\psi \sbullet D)\right)} \pcirc
 \varphi^{\psi\sbullet D} \pcirc \psi^D = 
&\\
&
\varphi \pcirc \left(\widetilde{\psi\sbullet D}\right) \pcirc \psi^D =
 \varphi \pcirc \psi \pcirc  \widetilde{D},
\end{eqnarray*}
we deduce that $(\varphi \pcirc \psi)^D = \varphi^{\psi\sbullet D} \pcirc \psi^D$ from the uniqueness in (i).
\bigskip

\noindent Part (iv) is also a consequence of the uniqueness property in (i).
\bigskip

\noindent (v) Let us assume that $\varphi$ has constant coefficients. We know from Lemma \ref{lemma:const_coeff_subst_map} that $ \langle \varphi \sbullet D, \varphi(a) \rangle = \varphi \left( \langle D, a\rangle \right)$ for all $a\in A[[\bfs]]_\nabla$, and so 
$\left(\widetilde{\varphi \sbullet D}\right) \pcirc \varphi =  \varphi \pcirc \widetilde{D}$. Hence, by the uniqueness property in (i) we deduce that $\varphi^D = \varphi$.
\end{proof}

The following proposition gives a recursive formula to obtain $\varphi^D$ from $\varphi$. 

\begin{proposition} \label{prop:varphi-D} 
With the notations of Proposition \ref{prop:varphi-D-main}, we have
$$ {\bf C}_e(\varphi,f+\nu) = \sum_{\substack{\scriptscriptstyle\beta+\gamma=e\\ 
\scriptscriptstyle |f+g|\leq |\beta|,|\nu|\leq |\gamma| }} {\bf C}_\beta(\varphi,f+g) D_g({\bf C}_\gamma(\varphi^D,\nu)) 
$$
for all $e\in \Delta$ and for all $f,\nu\in\nabla$ with $|f+\nu|\leq |e|$.
In particular, we have
the following recursive formula
$$ {\bf C}_e(\varphi^D,\nu) := {\bf C}_e(\varphi,\nu) - \sum_{\substack{\scriptscriptstyle\beta+\gamma=e\\ 
\scriptscriptstyle |g|\leq |\beta|, |\nu|\leq |\gamma| < |e| }} {\bf C}_\beta(\varphi,g) D_g({\bf C}_\gamma(\varphi^D,\nu)).
$$
for $e\in \Delta$, $\nu\in\nabla$ with $|e|\geq 1$ and $|\nu|\leq |e|$, starting with ${\bf C}_0(\varphi^D,0) = 1$.
\end{proposition}

\begin{proof} First, the case $f=0$ easily comes from the equality
$$  
\sum_{\substack{\scriptscriptstyle  e \in\Delta  \\ 
\scriptscriptstyle |\nu|\leq |e|}} {\bf C}_e(\varphi,\nu) \bft^e = \varphi (\bfs^\nu) = (\varphi \pcirc \widetilde{D})(\bfs^\nu) = 
\left(\left(\widetilde{\varphi \sbullet D}\right) \pcirc \varphi^D\right) (\bfs^\nu) \quad \forall \nu\in\nabla.
$$
For arbitrary $f$ one has to use Proposition \ref{prop:identidad-Cs}. Details are left to the reader.
\end{proof}

The proof of the following corollary is a consequence of Lemma \ref{lema:dual-leibniz-HS}.

\begin{corollary} \label{cor:aux-phi-D} Under the hypotheses of Proposition \ref{prop:varphi-D-main}, the following identity holds for each $e\in\Delta$
$$ \left(\varphi\sbullet D\right)^*_e =  
\sum_{ |\mu+\nu|\leq |e| } D^*_\mu \cdot D_\nu \left( {\bf C}_e(\varphi^D,\mu+\nu) \right).
$$
\end{corollary}

\begin{proposition} \label{prop:D-circ-varphi}
Let $D\in\HS^\bft_k(A;\Delta)$ be a HS-derivation and $\varphi:A[[\bfs]]_\nabla \to A[[\bft]]_\Delta$ a substitution map. Then, the following identity holds:
$$   \widetilde{D} \pcirc \varphi = \left(D(\varphi) \otimes \pi\right) \pcirc \left(\widetilde{\kappa \sbullet D}\right) \pcirc \iota,
$$
where:
\begin{itemize}
\item $D(\varphi): A[[\bfs]]_\nabla \to A[[\bft]]_\Delta$ is the substitution map determined by $D(\varphi)(s)= \widetilde{D}(\varphi(s))$ for all $s\in\bfs$.
\item $\pi:  A[[\bft]]_\Delta \rightarrow A$ is the augmentation, or equivalently, the substitution map\footnote{The map $\pi$ can be also understood as the truncation $\tau_{\Delta,\{0\}}: A[[\bft]]_\Delta \rightarrow A[[\bft]]_{\{0\}} =A $.} given by $\pi (t) = 0$ for all $t\in\bft$.
\item $\iota: A[[\bfs]]_\nabla \to A[[\bfs\sqcup \bft]]_{\nabla \times \Delta}$ and $\kappa: A[[\bft]]_\Delta \to A[[\bfs\sqcup \bft]]_{\nabla \times \Delta}$ are the combinatorial substitution maps determined by the inclusions  $\bfs \hookrightarrow \bfs\sqcup \bft$ and $\bft \hookrightarrow \bfs\sqcup \bft$, respectively.
\end{itemize}
\end{proposition}

\begin{proof} It is enough to check that both maps coincide on any $a\in A$ and on any $s\in\bfs$. Details are left to the reader.
\end{proof}

\begin{remark} Let us notice that with the notations of Propositions \ref{prop:varphi-D-main} and \ref{prop:D-circ-varphi}, we have
$ \varphi^D = (\varphi \sbullet D)^*(\varphi)$.
\end{remark}

The following proposition will not be used in this paper and will be stated without proof.

\begin{proposition} For any HS-derivation $D\in \HS_k^\bfs(A;\nabla)$ and any substitution map $\varphi \in \Sub(\bft,\bfu;\Delta,\Omega) $, there exists a substitution map $D\star \varphi \in \Sub(\bfs\sqcup \bft,\bfs\sqcup \bfu;\nabla \times \Delta, \nabla \times \Omega)$ such that for each HS-derivation 
$E\in \HS_k^\bft(A;\Delta)$  we have:
$$ D \btimes (\varphi\sbullet E)  = (D\star \varphi)\sbullet (D\btimes E).
$$
\end{proposition}

\section{Generating HS-derivations}

In this section we show how the action of substitution maps allows us to express any HS-derivation in terms of a fixed one under some natural hypotheses. We will be concerned with $(\bfs,\ft_m(\bfs))$-variate HS-derivations, where $\ft_m(\bfs) = \{\alpha\in\N^{(\bfs)}\ |\ |\alpha|\leq m\}$. To simplify we will write  $A[[\bfs]]_m := A[[\bfs]]_{\ft_m(\bfs)}$ and $\HS^\bfs_k(A;m):= \HS^\bfs_k(A;\ft_m(\bfs))$ for any integer $m\geq 1$, and $\HS^\bfs_k(A;\infty) := \HS^\bfs_k(A)$. For $m\geq n\geq 1$ we will denote $\tau_{mn}: \HS^\bfs_k(A;m) \to \HS^\bfs_k(A;n)$ the truncation map.
\medskip

Assume that $m\geq 1$ is an integer and 
let  $\varphi: A[[\bfs]]_m  \to A[[\bft]]_m$ be a substitution map. Let us write
 $$\varphi(s)=c^s = 
\sum_{\substack{\scriptstyle \beta\in\N^{(\bft)}\\ \scriptstyle 0<|\beta|\leq m }} c^s_\beta \bft^\beta \in \fn_0(\bft)/\ft_m(\bft)  \subset  A[[\bft]]_m,\quad s\in \bfs
$$ 
and let us denote by 
$\varphi_m, \varphi_{<m}:A[[\bfs]]_m \to A[[\bft]]_m$
the substitution maps determined by
\begin{eqnarray*}
& \displaystyle \varphi_m(s) = c_m^s:=  \sum_{\substack{\scriptstyle \beta\in\N^{(\bft)}\\ \scriptstyle |\beta|= m }} c^s_\beta \bft^\beta \in \fn_0(\bft)/\ft_m(\bft)  \in  A[[\bft]]_m,\quad s\in \bfs,&\\
& \displaystyle \varphi_{<m}(s) = c_{<m}^s:=  \sum_{\substack{\scriptstyle \beta\in\N^{(\bft)}\\ \scriptstyle 0<|\beta|< m }} c^s_\beta \bft^\beta \in \fn_0(\bft)/\ft_m(\bft)  \in  A[[\bft]]_m,
\quad s\in \bfs.&
\end{eqnarray*}
We have $c^s = c_m^s + c_{<m}^s$ and so $\varphi = \varphi_m + \varphi_{<m}$ (see \ref{nume:operaciones-con-substitutions}).

\begin{proposition} \label{prop:previa_main} 
With the above notations, for any HS-derivation $D\in \HS^\bfs_k(A;m)$ the following properties hold:
\begin{enumerate}
\item[(1)] $\left(\varphi_m\sbullet D\right)_{e} = 0$ for $0<|e| < m$ and $\left(\varphi_m\sbullet D\right)_{e} = \sum_{t\in \bfs} c^t_e D_{\bfs^t}$ for $|e| = m$,
where the $\bfs^t$ are the elements of the canonical basis of $\N^{(\bfs)}$.
\item[(2)] $\varphi \sbullet D = (\varphi_{m} \sbullet D) \pcirc (\varphi_{<m} \sbullet
D) = (\varphi_{<m} \sbullet
D)\pcirc (\varphi_{m} \sbullet D)$.
\end{enumerate}
\end{proposition}

\begin{proof} (1) Let us denote $E'= \varphi_m\sbullet D$.  Since $\tau_{m,m-1}(E')$ coincides with\\ $\tau_{m,m-1}(\varphi_m) \sbullet \tau_{m,m-1}(D)$ (see (\ref{eq:trunca_bullet})) and $\tau_{m,m-1}(\varphi_m)$ is the trivial substitution map, we deduce that $\tau_{m,m-1}(E') = \mathbb{I}$, i.e. $E_e=0$ whenever $0<|e| < m$.

From (\ref{eq:expression_phi_D}) and (\ref{eq:explicit-C_e(varphi,alpha)}), for $|e|>0$ we have $E'_e = \sum_{\scriptscriptstyle 0<|\alpha|\leq |e|} {\bf C}_e(\varphi_m,\alpha) D_\alpha$, with 
$$ {\bf C}_e(\varphi_m,\alpha)=  \sum_{\scriptscriptstyle \mathcal{f}^{\bullet\bullet} \in \Par(e,\alpha)}
C_{\mathcal{f}^{\bullet\bullet}}
 \quad \text{for\ }\ |\alpha| \leq |e|,\quad 
 C_{\mathcal{f}^{\bullet\bullet}} = \prod_{\scriptscriptstyle s\in \supp \alpha} \prod_{\scriptstyle r=1}^{\scriptstyle \alpha_s}  (c^s_m)_{\mathcal{f}^{sr}}.
$$
Assume now that $|e|=m$, $1 < |\alpha|\leq m$ and let $\mathcal{f}^{\bullet\bullet} \in \Par(e,\alpha)$. 
Since 
$$\sum_{\scriptscriptstyle s\in \supp \alpha} \sum_{\scriptscriptstyle r=1}^{\scriptscriptstyle  \alpha_s} \mathcal{f}^{sr} = e,$$
we deduce that $|\mathcal{f}^{sr}| < |e|=m$ for all $s,r$ and so $(c^s_m)_{\mathcal{f}^{sr}}=0$ and $C_{\mathcal{f}^{\bullet\bullet}}=0$. Consequently, ${\bf C}_e(\varphi_m,\alpha)=0$.
\medskip

If $|\alpha|= 1$, then $\alpha$ must be an element $\bfs^t$ of the canonical basis of $\N^{(\bfs)}$
and from Lemma \ref{lemma:behavior_C_alpha}, (1), we know that ${\bf C}_e(\varphi_m,\bfs^t)=(c^t_m)_e$. We conclude that
$$
E'_e =\cdots = \sum_{\scriptscriptstyle t\in \bfs} {\bf C}_e(\varphi_m,\bfs^t) D_{\bfs^t} = \sum_{\scriptscriptstyle t\in \bfs} (c^t_m)_e D_{\bfs^t}=  \sum_{\scriptscriptstyle t\in \bfs} c^t_e D_{\bfs^t}.
$$

\noindent (2)
Let us write $E=\varphi \sbullet D$, $E'= \varphi_m\sbullet D$ and $E''=\varphi_{<m} \sbullet
D$. We have
\begin{eqnarray*}
&\tau_{m,m-1}(E) = \tau_{m,m-1}(\varphi) \sbullet \tau_{m,m-1}(D) =
&\\
& \tau_{m,m-1}(\varphi_{<m}) \sbullet \tau_{m,m-1}(D) = \tau_{m,m-1}(E''). 
\end{eqnarray*}
By property (1), we know that $\tau_{m,m-1}(E')$ is the identity and we deduce that $\tau_{m,m-1}(E) =
\tau_{m,m-1}(E'\pcirc E'') = \tau_{m,m-1}(E''\pcirc E')$. So
$ E_e = (E'\pcirc E'')_e = (E''\pcirc E')_e $ for $|e|<m$.

Now, let $e\in \N^{(\bft)}$ be with $|e|=m$. By using again that $\tau_{m,m-1}(E')$ is the identity, we have
$(E'\pcirc E'')_e = \dots = E'_e + E''_e = \cdots = (E'' \pcirc E')_e$, and we conclude that 
$E'\pcirc E'' = E'' \pcirc E'$.

On the other hand, from Lemma \ref{lemma:behavior_C_alpha}, (1), we have
that ${\bf C}_e(\varphi_{<m},\alpha) =0$ whenever $|\alpha|=1$, and one can see
 that ${\bf C}_e(\varphi,\alpha) = {\bf C}_e(\varphi_{<m},\alpha)$ whenever that 
$2\leq |\alpha|\leq |e|$. So:
\begin{eqnarray*}
&\displaystyle E_e= \sum_{\scriptscriptstyle 1\leq |\alpha|\leq m}
{\bf C}_e(\varphi,\alpha) D_\alpha = \sum_{\scriptscriptstyle |\alpha|=1}
{\bf C}_e(\varphi,\alpha) D_\alpha + \sum_{\scriptscriptstyle 2\leq |\alpha|\leq m}
{\bf C}_e(\varphi,\alpha) D_\alpha=&\\
&\displaystyle \sum_{\scriptscriptstyle t\in \bfs} c^t_e D_{\bfs^t} + \sum_{\scriptscriptstyle 2\leq |\alpha|\leq m}
{\bf C}_e(\varphi_{<m},\alpha) D_\alpha=  E'_e + \sum_{\scriptscriptstyle 1\leq |\alpha|\leq m}
{\bf C}_e(\varphi_{<m},\alpha) D_\alpha=  E'_e + E''_e
\end{eqnarray*}
and $E=E'\pcirc E'' = E'' \pcirc E'$.
\end{proof}

The following theorem generalizes Theorem 2.8 in
\cite{magda_nar_hs} to the case where $\Der_k(A)$ is not necessarily a finitely generated $A$-module.
The use of substitution maps makes its proof more conceptual.

\begin{theorem} \label{thm:main_1} Let $m\geq 1$ be an integer, or $m=\infty$, and $D\in \HS^\bfs_k(A;m)$ a $\bfs$-variate HS-derivation of length $m$ such that $\{D_\alpha, |\alpha|=1\}$ is a system of generators of the $A$-module $\Der_k(A)$. Then, for each set $\bft$ and each HS-derivation $G\in\HS^\bft_k(A;m)$ there is a substitution map $\varphi:A[[\bfs]]_m \to A[[\bft]]_m$ such that $G = \varphi \sbullet D$. Moreover, if $\{D_\alpha, |\alpha|=1\}$ is a basis of $\Der_k(A)$, $\varphi$ is uniquely determined.
\end{theorem}

\begin{proof} For $m$ finite, we will proceed by induction on $m$. For $m=1$ the result is clear. Assume that the result is true for HS-derivations of length $m-1$ and consider a $D\in \HS^\bfs_k(A;m)$ such that $\{D_\alpha, |\alpha|=1\}$ is a system of generators of the $A$-module $\Der_k(A)$ and a $G\in\HS^\bft_k(A;m)$. By the induction hypothesis, there is a substitution map $\varphi':
A[[\bfs]]_{m-1} \to A[[\bft]]_{m-1}$, given by
$\varphi'(s) = \sum_{\scriptscriptstyle |\beta|\leq m-1} c^s_\beta \bft^\beta$, $s\in \bfs$, and 
such that $\tau_{m,m-1}(G)= \varphi' \sbullet \tau_{m,m-1}(D)$. 
Let $\varphi'':A[[\bfs]]_m \to A[[\bfu]]_m$ be the substitution map lifting $\varphi'$ (i.e. $\tau_{m,m-1}(\varphi'') = \varphi'$) given by
$\varphi''(s) = \sum_{\scriptscriptstyle |\beta|\leq m-1} c^s_\beta \bft^\beta \in A[[\bft]]_m$, $s\in \bfs$, 
 and consider $F= \varphi'' \sbullet D$. We obviously have $\tau_{m,m-1}(F)= \tau_{m,m-1}(G)$ and so, for $H=G \pcirc F^*$, the truncation $\tau_{m,m-1}(H)$ is the identity and $H_e=0$ for $0<|e|<m$. We deduce that each component of $H$ of highest order, $H_e$ with $|e|=m$, must be a $k$-derivation of $A$ and so there is a family $\{c^s_e, s\in \bfs\}$ of elements of $A$ such that $c^s_e=0$ for all $s$ except  
a finite number of indices and $H_e= \sum_{\scriptscriptstyle s\in \bfs} c^\bfs_e D_{\bfs^s}$, where $\{\bfs^s, s\in \bfs\}$ is the canonical basis of $\N^{(\bfs)}$.
To finish, let us consider the substitution map $\varphi:A[[\bfs]]_m \to A[[\bft]]_m$ given by $\varphi(s) = \sum_{\scriptscriptstyle |\beta|\leq m} c^s_\beta \bft^\beta$, $s\in \bfs$. From Proposition \ref{prop:previa_main} we have
$$ \varphi \sbullet D = (\varphi_{m} \sbullet D) \pcirc (\varphi_{<m} \sbullet
D) = H \pcirc (\varphi'' \sbullet D) = H \pcirc F = G. $$

For HS-derivations of infinite length, following the above procedure we can construct $\varphi$ as a projective limit of substitution maps $A[[\bfs]]_m \to A[[\bft]]_m$, $m\geq 1$.
\medskip

Now assume that the set $\{D_\alpha, |\alpha|=1\}$ is linearly independent over $A$ and let us prove that 
\begin{equation} \label{eq:uni-varphi}
\varphi \sbullet D = \psi \sbullet D\quad \Longrightarrow\quad \varphi = \psi. 
\end{equation}
The infinite length case can be reduced to the finite case since $\varphi = \psi$ if and only if all their finite truncations are equal. For the finite length case, we proceed by induction on the length $m$. Assume that the substitution maps are given by
\begin{eqnarray*}
& \displaystyle \varphi(s) = c^s:=  \sum_{\substack{\scriptscriptstyle \beta\in\N^{(\bft)}\\ \scriptscriptstyle 0<|\beta|\leq m }} c^s_\beta \bft^\beta \in \fn_0(\bft)/\ft_m(\bft)  \subset  A[[\bft]]_m,\quad s\in \bfs&\\
& \displaystyle \psi(s) = d^s:=  \sum_{\substack{\scriptscriptstyle \beta\in\N^{(\bft)}\\ \scriptscriptstyle 0<|\beta|\leq m }} d^s_\beta \bft^\beta \in \fn_0(\bft)/\ft_m(\bft)  \subset  A[[\bft]]_m,\quad s\in \bfs.
&
\end{eqnarray*}
If $m=1$, then $\varphi = \varphi_1$ and $\psi = \psi_1$ and for each $e\in\N^{(\bft)}$ with $|e|=1$ we have from Proposition \ref{prop:previa_main}
$$ \sum_{\scriptscriptstyle s\in \bfs} c^s_e D_{\bfs^s} = (\varphi_1 \sbullet D)_e = (\varphi \sbullet D)_e = (\psi \sbullet D)_e = (\psi_1 \sbullet D)_e = \sum_{\scriptscriptstyle s\in \bfs} d^s_e D_{\bfs^s}
$$
and we deduce that $c^s_e= d^s_e$ for all $s\in \bfs$ and so $\varphi=\psi$.
\medskip

Now assume that (\ref{eq:uni-varphi}) is true whenever the length is $m-1$ and take $D, \varphi $ and $\psi$ as before of length $m$ with $\varphi \sbullet D = \psi \sbullet D$. By considering $(m-1)$-truncations and using the induction hypothesis we deduce that $\tau_{m,m-1}(\varphi)=\tau_{m,m-1}(\psi)$, or equivalently $\varphi_{<m} = \psi_{<m}$.

From Proposition \ref{prop:previa_main} we obtain first that $\varphi_m\sbullet D = \psi_m\sbullet D$ and second that for each $e\in\N^{(\bft)}$ with $|e|=m$
$$ \sum_{\scriptscriptstyle s\in \bfs} c^s_e D_{\bfs^s} = \sum_{\scriptscriptstyle s\in \bfs} d^s_e D_{\bfs^s}.$$
We conclude that $\varphi_m = \psi_m$ and so $\varphi=\psi$.
\end{proof}

Now we recall the definition of integrability.

\begin{definition} (Cf. \cite{brown_1978,mat-intder-I})  \label{def:HS-integ}
Let $m\geq 1$ be an integer or $m=\infty$ and $\bfs$ a set.
\begin{enumerate}
\item[(i)] We say that a $k$-derivation $\delta:A\to
A$ is {\em $m$-integrable} (over $k$) if there is a
Hasse--Schmidt derivation $D\in \HS_k(A;m)$ such that $D_1=\delta$. Any such $D$ will
be called an  {\em $m$-integral} of $\delta$. The set of $m$-integrable
$k$-derivations of $A$ is denoted by $\Ider_k(A;m)$. We simply say that $\delta$ is
{\em integrable} if it is $\infty$-integrable and we denote
$\Ider_k(A):=\Ider_k(A;\infty)$.
\item[(ii)] We say that a $\bfs$-variate  HS-derivation $D'\in\HS^\bfs_k(A;n)$, with $1\leq n < m$,
is {\em $m$-integrable} (over $k$) if there
is a $\bfs$-variate  HS-derivation $D\in \HS^\bfs_k(A;m)$ such that $\tau_{mn}D=D'$. Any such
$D$ will be called an {\em $m$-integral} of $D'$. The set of $m$-integrable $\bfs$-variate  HS-derivations of $A$ over $k$ of length $n$ is denoted by
$\IHS^\bfs_k(A;n;m)$. We simply say that $D'$ is {\em integrable} if it is
$\infty$-integrable and we denote $\IHS^\bfs_k(A;n) := \IHS^\bfs_k(A;n;\infty)$.
\end{enumerate}
\end{definition}

\begin{corollary} Let $m\geq 1$ be an integer or $m=\infty$. The following properties are equivalent:
\begin{enumerate}
\item[(1)] $\Ider_k(A;m) = \Der_k(A)$.
\item[(2)] $\IHS^\bfs_k(A;n;m) = \HS^\bfs_k(A;n)$ for all $n$ with $1\leq n <m$ and all sets $\bfs$.
\end{enumerate}
\end{corollary}

\begin{proof} We only have to prove (1) $\Longrightarrow$ (2). Let $\{\delta_t, t\in \bft\}$ be a system of generators of the $A$-module $\Der_k(A)$, and for each $t\in\bft$  let $D^t\in \HS_k(A;m)$ be an $m$-integral of $\delta_t$. By considering some total ordering $<$ on $\bft$, we can define $D\in \HS^\bft_k(A;m)$ as the external product (see Definition \ref{defi:external-x}) of the ordered family $\{D^t, t\in\bft\}$, i.e. $D_0=\Id$ and for each $\alpha \in\N^{(\bft)}$, $\alpha\neq 0$,
$$  D_\alpha = D^{t_1}_{\alpha_{t_1}} \pcirc \cdots \pcirc D^{t_e}_{\alpha_{t_e}}\quad \text{with\ }\ 
\supp \alpha = \{t_1 <  \cdots < t_e\}.
$$
Let $n$ be an integer with $1\leq n <m$, $\bfs$ a set and 
$E\in \HS^\bfs_k(A;n)$. After Theorem \ref{thm:main_1}, there exists a substitution map $\varphi:A[[\bft]]_n \to A[[\bfs]]_n$ such that $E=\varphi \sbullet \tau_{mn}(D)$. By considering any substitution map $\varphi':A[[\bft]]_m \to A[[\bfs]]_m$ lifting $\varphi$ we find that $\varphi'\sbullet D$ is an $m$-integral of $E$ and so $E\in \IHS^\bfs_k(A;n;m)$.
\end{proof}

\bigskip

\noindent \href{http://personal.us.es/narvaez/}{L. Narv\'{a}ez Macarro}\\
Departamento de \'Algebra \&\ Instituto de Matem\'aticas (IMUS)\\ 
Universidad de Sevilla, Spain\\
email: narvaez@us.es


\begin{thebibliography}{99}


\bibitem{brown_1978}
W.~C. Brown.
\newblock \href{https://projecteuclid.org/euclid.ojm/1200771280}{On the embedding of derivations of finite rank into derivations of infinite rank}.
\newblock {\em Osaka J. Math.} 15 (1978), 381--389.

\bibitem{bourbaki_algebra_II_chap_4_7}
N.~Bourbaki,
\newblock Elements of Mathematics. Algebra II. Chapters 4-7. 
\newblock Springer-Verlag, Berlin, 2003. 

\bibitem{magda_nar_hs}
M.~Fern\'{a}ndez-Lebr\'{o}n and L.~Narv\'{a}ez-Macarro.
\newblock \href{https://doi.org/10.1016/S0021-8693(03)00238-2}{
Hasse-{S}chmidt derivations and coefficient fields in positive characteristics}.
\newblock {\em J. Algebra} 265 (1) (2003), 200--210.
\newblock (\href{https://arxiv.org/abs/math/0206261}{\tt arXiv:math/0206261}).

\bibitem{has37}
 H.~Hasse and F.~K.~Schmidt.
 \newblock \href{https://doi.org/10.1515/crll.1937.177.215}{Noch eine
Begr\"{u}ndung der Theorie der h\"{o}heren Differrentialquotienten in einem algebraischen
Funktionenk\"orper einer Unbestimmten}.
\newblock {\em J. Reine U. Angew. Math.} 177 (1937), 223-239.
	
\bibitem{hoff_kow_2015}
D.~Hoffmann and P.~Kowalski.
\newblock \href{https://doi.org/10.1016/j.jpaa.2014.05.024}{Integrating Hasse--Schmidt derivations}.
\newblock {\em J. Pure and Appl. Algebra}, 219 (2015), 875--896.

\bibitem{hoff_kow_2016}
D.~Hoffmann and P.~Kowalski.
\newblock \href{https://doi.org/10.1112/jlms/jdw009}{Existentially closed fields with {$G$}-derivations}.
\newblock {\em J. London Math. Soc. (2)}, 93 (3) (2016), 590--618.

\bibitem{mat-intder-I}
H.~Matsumura.
\newblock \href{https://projecteuclid.org/euclid.nmj/1118786907}{Integrable derivations}.
\newblock  {\em Nagoya Math. J.} 87 (1982), 227--245.

\bibitem{mat_86}
H.~Matsumura.
\newblock Commutative Ring Theory. Vol. 8 of Cambridge studies in
  advanced mathematics,
\newblock Cambridge Univ. Press, Cambidge, 1986.

\bibitem{nar_2009}
L.~Narv\'aez Macarro.
\newblock \href{https://doi.org/10.5802/aif.2513}{Hasse--Schmidt derivations, divided powers and differential smoothness}.
\newblock {\em Ann. Inst. Fourier (Grenoble)} 59 (7) (2009), 2979--3014.
\newblock (\href{https://arxiv.org/abs/0903.0246}{\tt arXiv:0903.0246}).

\bibitem{nar_2012}
L.~Narv\'aez Macarro.
\newblock \href{https://doi.org/10.1016/j.aim.2012.01.015}{On the modules of {$m$}-integrable derivations in non-zero characteristic}.
\newblock {\em Adv. Math.} 229 (5) (2012), 2712--2740.
\newblock (\href{https://arxiv.org/abs/1106.1391}{\tt arXiv:1106.1391}).

\bibitem{nar_Maringa}
L.~Narv\'aez Macarro.
\newblock \href{http://personal.us.es/narvaez/DSCA_course_2014.pdf}{Differential Structures in Commutative Algebra}. 
\newblock Mini-course at the XXIII Brazilian Algebra Meeting, July 27 - August 1, 2014, Maring\'a, Brazil.
	
\bibitem{nar_in_prep}
L.~Narv\'{a}ez~Macarro.  
\newblock Rings of differential operators as enveloping algebras of Hasse--Schmidt derivations.
\newblock (\href{https://arxiv.org/abs/1807.10193}{\tt  arXiv:1807.10193}).
  
\bibitem{vojta_HS}
P.~Vojta.
\newblock Jets via {H}asse--{S}chmidt derivations.
\newblock In ``Diophantine geometry'', CRM Series, vol. 4,
Ed. Norm., Pisa, 2007, 335--361.
 \newblock (\href{https://arxiv.org/abs/math/0407113}{\tt  arXiv:math/1201.3594}).





\end{thebibliography}
\end{document}